\documentclass{article}

\usepackage{PRIMEarxiv}
\usepackage[utf8]{inputenc}
\usepackage[T1]{fontenc}
\usepackage{amsmath,amssymb,amsthm,mathtools}
\usepackage{mathrsfs}
\usepackage{bm}
\usepackage{booktabs}
\usepackage{enumitem}
\setlist[enumerate]{label=(\arabic*)}
\usepackage{microtype}
\usepackage{xcolor}
\usepackage{hyperref}
\usepackage{url}
\usepackage[numbers,sort&compress]{natbib}
\usepackage{fancyhdr}
\usepackage{indentfirst}
\numberwithin{equation}{section}

\hypersetup{
  colorlinks=true,
  linkcolor=blue,
  citecolor=blue,
  urlcolor=blue
}

\newtheorem{theorem}{Theorem}[section]
\newtheorem{proposition}[theorem]{Proposition}
\newtheorem{lemma}[theorem]{Lemma}

\theoremstyle{definition}

\theoremstyle{remark}
\newtheorem{remark}[theorem]{Remark}

\newcommand{\E}{\mathbb{E}}
\newcommand{\Pp}{\mathbb{P}}
\newcommand{\R}{\mathbb{R}}
\newcommand{\N}{\mathbb{N}}
\newcommand{\Z}{\mathbb{Z}}

\newcommand{\cF}{\mathcal{F}}
\newcommand{\dd}{\,\mathrm{d}}
\newcommand{\1}{\mathbf{1}}
\newcommand{\Var}{\operatorname{Var}}
\newcommand{\Cov}{\operatorname{Cov}}

\newcommand{\e}{\mathrm{e}}

\newcommand{\norm}[1]{\left\lVert #1\right\rVert}

\newcommand{\Ptilde}{\widetilde{\Pp}}
\newcommand{\Etilde}{\widetilde{\E}}

\title{Disorder Thresholds and Free Energy of Brownian Directed Polymers with\\
Product and Radial Spatial Correlations}

\author{%
  \textnormal{Junjie Cao}%
  \thanks{Junjie Cao: Beijing Normal-Hong Kong Baptist University, \texttt{caojunjie@bnbu.edu.cn}}
  \and
  \textnormal{Guanglin Rang}%
  \thanks{Guanglin Rang: School of Mathematics and Statistics, Wuhan University, \texttt{glrang.math@whu.edu.cn}}
  \and
  \textnormal{Jianglun Wu}%
  \thanks{Jianglun Wu: Beijing Normal-Hong Kong Baptist University, \texttt{jianglunwu@bnbu.edu.cn}}
}

\begin{document}
\maketitle

\begin{abstract}
We study a Brownian directed polymer in a centered Gaussian environment that is white in time and 
colored in space having long-range spatial correlations. For product-type covariances
\(Q(x)\asymp\prod_{j=1}^d(1+|x_j|)^{-\alpha_j}\), with
\(\alpha_j\in(0,1)\) and \(\kappa=\sum_j\alpha_j\), we identify the disorder transition at the marginal value \(\kappa=2\). For \(\kappa>2\), weak disorder holds at sufficiently small inverse temperature; for \(\kappa<2\), the quenched free energy $p(\beta)$ satisfies
\(-p(\beta)\asymp\beta^{4/(2-\kappa)}\) as \(\beta\downarrow0\). For \(\kappa=2\), strong disorder holds for every \(\beta>0\), while \(p(\beta)=0\) for all sufficiently small \(\beta\), so \(\beta_c=0<\bar\beta_c\). We also consider the radial covariance cases, where $ Q(x)\asymp(1+|x|)^{-\vartheta}$, when \(d\ge3,\vartheta=2\) and \(d=2,\vartheta\ge2\), which was left unanswered in Lacoin~\cite{Lacoin2011}. When $d=2$, we get
\(\ln(-p(\beta))\asymp-\beta^{-2}\) for \(\vartheta>2\) and
\(\ln(-p(\beta))\asymp-\beta^{-1}\) for \(\vartheta=2\). The proofs consist of replica coupling, Feynman--Kac variational formula, overlap methods, and continuous-space fractional moments with ordered Wiener-chaos changes of measure.
\end{abstract}
 
\section{Introduction and main results}
\label{sec:introduction}

\subsection{The Gaussian Brownian polymer}
\label{subsec:model}

Directed polymers were introduced as effective models for random interfaces
by Huse and Henley~\cite{HuseHenley1985}.  We refer to
\cite{denHollander2009,Comets2017,Zygouras2024} for general accounts of the
model.  The continuous Gaussian model studied below goes back to
Rovira and Tindel~\cite{RoviraTindel2005}, Lacoin~\cite{Lacoin2011}.

In the usual space--time i.i.d. setting, the disorder transition is often
described through two related quantities.  The first is the strong-disorder
threshold \(\beta_c\), defined by the extinction of the normalized partition-function
martingale.  The second is the very-strong-disorder threshold \(\bar\beta_c\), defined by the
strict negativity of the normalized quenched free energy.  Very strong disorder always implies strong disorder, but the converse is a subtler question. The possible influencing factors therein will be discussed in Section~\ref{subsec:literature-strong-verystrong} below.

This paper studies this question for Brownian directed polymers in Gaussian
environments whose spatial covariance has long-range decay.  Our results show
that the decay exponent of the covariance governs not only the small-\(\beta\)
asymptotics of the free energy, but also the coincidence of the strong- and
very-strong-disorder thresholds.  In particular, at the critical
decay scale for product-type correlations, strong disorder may hold for every
\(\beta>0\) while the quenched free energy remains zero at
small \(\beta\). The same result also holds for radial kernels. Thus critical spatial correlations produce a
separation of the two disorder thresholds.

We deal with two types of correlations.  The first is anisotropic product kernels.
The second is radial kernels in Lacoin~\cite{Lacoin2011}.  For both classes we identify the
corresponding free-energy scale.

Let \(d\ge1\), and let \(B=(B_t)_{t\ge0}\) be standard Brownian motion in
\(\R^d\).  We write \(P_x\) for its law when \(B_0=x\), and \(P=P_0\).
Let \(Q:\R^d\to[0,\infty)\) be a bounded, continuous and positive definite function,
with \(Q(0)=1\).  The environment is a real centered Gaussian field, with covariance
\begin{equation*}
  \E\big[\omega(t,x)\omega(s,y)\big]=(t\wedge s)Q(x-y).
\end{equation*}
Equivalently, its distributional time derivative is white in time:
\begin{equation*}
  \E\big[\omega(\dd t,x)\omega(\dd s,y)\big]
  =\delta_0(t-s)Q(x-y)\dd t\dd s.
\end{equation*}
A precise realization of the environment is obtained from an isonormal
Gaussian process over the Hilbert space \(\mathcal H_Q\), which is separable, defined as the
completion of \(C_c^\infty(\R_+\times\R^d)\) under
\[
  \langle f,g\rangle_{\mathcal H_Q}
  :=
  \int_0^\infty\!
  \iint_{\R^d\times\R^d}
  f(t,x)Q(x-z)g(t,z)\dd x\dd z\dd t.
\]

Let \(\mathfrak H_Q\) be the Hilbert space generated by the symbols
\((e_x)_{x\in\R^d}\), with
\begin{equation}
  \langle e_x,e_y\rangle_{\mathfrak H_Q}
  :=
  Q(x-y).
  \label{eq:e-x-inner-product}
\end{equation}
For \(f\in C_c^\infty(\R_+\times\R^d)\), set
\[
  (\mathcal Jf)(t)
  :=
  \int_{\R^d}f(t,x)e_x\dd x.
\]
Then
\[
  \langle \mathcal Jf,\mathcal Jg\rangle_
  {L^2(\R_+;\mathfrak H_Q)}
  =
  \langle f,g\rangle_{\mathcal H_Q},
\]
so that \(\mathcal J\) extends to the canonical identification
\begin{equation}
  \mathcal H_Q
  \simeq
  L^2(\R_+;\mathfrak H_Q).
  \label{eq:space-time-Hilbert-factorization}
\end{equation}

Let \(W\) be the isonormal Gaussian process over
\(L^2(\R_+;\mathfrak H_Q)\), so that
\[
  \E[W(h)W(k)]
  =
  \langle h,k\rangle_{L^2(\R_+;\mathfrak H_Q)}.
\]
The environment is realized by
\[
  \omega(f):=W(\mathcal Jf).
\]
For a fixed continuous path \(B\), define
\begin{equation}
  h_B(s):=\1_{[0,T]}(s)e_{B_s},
  \qquad
  H_T(B):=W(h_B)
  =
  \int_0^T\omega(\dd s,B_s).
  \label{eq:H-as-isonormal}
\end{equation}
Consequently,
\begin{equation}
  \E[H_T(B)H_T(B')]
  =
  \langle h_B,h_{B'}\rangle
  =
  \int_0^TQ(B_s-B_s')\dd s.
  \label{eq:Hamiltonian-covariance}
\end{equation}
In particular, \(\operatorname{Var}(H_T(B))=T\).  The representation \eqref{eq:space-time-Hilbert-factorization} will be used in
Section~\ref{subsec:detector-construction} and Section~\ref{subsec:critical-detector} to construct the ordered
Wiener-chaos detector.
For inverse temperature \(\beta\ge0\), define the normalized partition function
\begin{equation}
  W_T(\beta)
  :=P\left[\exp\left\{\beta H_T(B)-\frac{\beta^2}{2}T\right\}\right]
  \label{eq:partition-function}
\end{equation}
and the polymer measure
\begin{equation}
  \mu_{T,\omega}^{\beta}(\dd B)
  :=\frac{1}{W_T(\beta)}
  \exp\left\{\beta H_T(B)-\frac{\beta^2}{2}T\right\}P(\dd B).
  \label{eq:polymer-measure}
\end{equation}
When no confusion can arise, the superscript \(\beta\) is omitted.

Given a fixed $\beta$, the process \((W_T(\beta))_{T\ge0}\) is a positive mean-one martingale in
the environment filtration \(\mathcal{F}_t = \sigma(\omega(s, x) : 0 \leq s \leq t, \, x \in \mathbb{R}^d) \).  It therefore converges almost surely to a limit
\(W_\infty(\beta)\).  The weak-disorder phase is defined by
\(W_\infty(\beta)>0\) almost surely, and the strong-disorder phase by
\(W_\infty(\beta)=0\) almost surely.

Set
\begin{equation}
  a_T(\beta):=\E\log W_T(\beta).
  \label{eq:aT-definition}
\end{equation}
We have the following limit (see \cite[Lemma~2.4,Propostion~2.6]{RoviraTindel2005}),
\begin{equation}
  p(\beta)
  :=\lim_{T\to\infty}\frac{a_T(\beta)}{T}
  =\sup_{T>0}\frac{a_T(\beta)}{T}\le \sup_{T>0}\frac{\log\E[W_T(\beta)]}{T}=0.
  \label{eq:free-energy}
\end{equation}
We distinguish the two critical points:
\begin{equation}
  \beta_c:=\sup\{\beta\ge0:W_\infty(\beta)>0\ \text{a.s.}\},
  \qquad
  \bar\beta_c:=\sup\{\beta\ge0:p(\beta)=0\}.
  \label{eq:two-critical-points}
\end{equation}

\subsection{The Covariance Structure}
\label{subsec:product-kernel}

Fix exponents
\begin{equation}
  \alpha_j\in(0,1),
  \qquad \kappa:=\sum_{j=1}^d\alpha_j,
  \label{eq:kappa-definition}
\end{equation}
and suppose that there are constants \(0<c_Q\le C_Q<\infty\) such that
\begin{equation}
  c_Q\prod_{j=1}^d(1+|x_j|)^{-\alpha_j}
  \le Q(x)\le
  C_Q\prod_{j=1}^d(1+|x_j|)^{-\alpha_j},
  \qquad x\in\R^d.
  \label{eq:product-tail}
\end{equation}
A concrete positive-definite example is
\begin{equation*}
  Q_{\bm\alpha}(x):=\prod_{j=1}^d(1+x_j^2)^{-\alpha_j/2}.
  \label{eq:product-bessel-example}
\end{equation*}
The second type is radial.  We assume that for some \(\vartheta>0\), there are constants \(0<c_Q\le C_Q<\infty\) such that
\begin{equation}
  c_Q(1+|x|)^{-\vartheta}
  \le Q(x)\le C_Q(1+|x|)^{-\vartheta},
  \qquad x\in\R^d.
  \label{eq:radial-tail-class}
\end{equation}

\subsection{Main results}
\label{subsec:main-theorem}

\begin{theorem}
\label{thm:main-product}
Assume \eqref{eq:product-tail}.
\begin{enumerate}[label=\textup{(\roman*)}]
  \item If \(\kappa>2\), then
  \begin{equation}
    0<\beta_c\le\bar\beta_c<\infty.
    \label{eq:integrable-thresholds}
  \end{equation}
  \item If \(\kappa<2\), then there exist \(0<c<C<\infty\) and
  \(\beta_0>0\) such that
  \begin{equation}
    c\,\beta^{4/(2-\kappa)}
    \le -p(\beta)\le C\,\beta^{4/(2-\kappa)},
    \qquad 0<\beta\le\beta_0.
    \label{eq:main-power-law}
  \end{equation}
  In particular, \(\beta_c=\bar\beta_c=0\).

  \item If \(\kappa=2\), then
  \begin{equation}
    \beta_c=0<\bar\beta_c<\infty.
    \label{eq:critical-threshold-separation}
  \end{equation}
  More explicitly, strong disorder holds for every \(\beta>0\), while
  there exists \(\beta_0>0\) such that
  \begin{equation}
    p(\beta)=0,
    \qquad 0\le\beta\le\beta_0.
    \label{eq:critical-zero-free-energy}
  \end{equation}
\end{enumerate}
All constants may depend on \(d,\bm\alpha,c_Q,C_Q\), but not on
\(\beta\). In particular, the critical case is possible only when \(d\ge3\), since every
\(\alpha_j<1\).  
\end{theorem}

\begin{theorem}
\label{thm:radial-regimes}
Assume \eqref{eq:radial-tail-class}.
\begin{enumerate}[label=\textup{(\roman*)}]
  \item If \(d\ge3\) and \(\vartheta=2\) then
  \begin{equation}
    \beta_c=0<\bar\beta_c<\infty.
    \label{eq:riesz-threshold-separation}
  \end{equation}
  More explicitly, strong disorder holds for every \(\beta>0\), whereas
  there is \(\beta_0>0\) such that
  \begin{equation}
    p(\beta)=0,
    \qquad 0\le\beta\le\beta_0.
    \label{eq:riesz-zero-free-energy}
  \end{equation}

  \item If \(d=2\) and \(\vartheta>2\),
  then there are \(0<c<C<\infty\) and \(\beta_0>0\) such that
  \begin{equation}
    \exp\left\{-\frac{C}{\beta^2}\right\}
    \le -p(\beta)\le
    \exp\left\{-\frac{c}{\beta^2}\right\},
    \qquad 0<\beta\le\beta_0.
    \label{eq:two-d-integrable-free-energy}
  \end{equation}
  Consequently, \( \beta_c=\bar\beta_c=0\).

  \item If \(d=2\) and \(\vartheta=2\),
  then there are \(0<c<C<\infty\) and \(\beta_0>0\) such that
  \begin{equation}
    \exp\left\{-\frac{C}{\beta}\right\}
    \le -p(\beta)\le
    \exp\left\{-\frac{c}{\beta}\right\},
    \qquad 0<\beta\le\beta_0.
    \label{eq:two-d-critical-free-energy}
  \end{equation}
  Consequently, \( \beta_c=\bar\beta_c=0\).
\end{enumerate}
All constants may depend on \(d,\bm\alpha,c_Q,C_Q\), but not on
\(\beta\).
\end{theorem}

\subsection{Separation and coincidence for strong and very strong}
\label{subsec:literature-strong-verystrong}

For the directed polymer with i.i.d.\ space--time disorder,
strong disorder holds at every positive
inverse temperature in dimensions
$d=1,2$ by \cite{CarmonaHu2002,CometsShigaYoshida2003,CometsYoshida2006}.  The free-energy
estimates of \cite{CometsVargas2006,Lacoin2010} show that disorder is
in fact very strong throughout this regime. 
The success in dimensions \(d=1,2\) gave rise to the expectation that the equivalence between strong and very strong disorder would continue to hold in higher dimensions.  Equality of the two critical points was conjectured
in \cite{CarmonaHu2006,CometsYoshida2006}.  Junk and Lacoin first proved it
for the directed polymer when the i.i.d. disorder is bounded from
above \cite{JunkLacoin2024}.  Their subsequent work
\cite{JunkLacoin2026} removes the upper-bound assumption on the environment
and allows general random walks. In contrast, much less work has been done in the correlated setting. Lacoin~\cite{Lacoin2011} determined the order of the free energy for the Brownian polymer with radially long-range correlated spatial disorder in $d\ge 2$, for certain decay exponents $\vartheta$ in \eqref{eq:radial-tail-class}. In particular, when $\vartheta<2$, one has $\beta_c=\bar\beta_c$.

Recently, however, examples of non-coincidence have also been found. Viveros~\cite{Viveros2023} considered a directed polymer, whose random walk has very heavy-tailed increments, in i.i.d. environment  in \(d=1\). He showed that the normalized quenched free energy vanishes for every \(\beta>0\), while strong disorder may still occur. An earlier geometric mechanism was identified by Cosco, Seroussi and Zeitouni~\cite{CoscoSeroussiZeitouni2021}, who studied directed polymers on general graphs. In particular, for upwardly biased random walks on certain supercritical Galton–Watson trees, they exhibited regimes in which strong disorder holds but very strong disorder does not, yielding \(\beta_c<\bar\beta_c\).
A more recent geometric example is provided by Nitzschner~\cite{Nitzschner2025} and Cottini and Nitzschner~\cite{CottiniNitzschner2025}. Nitzschner proved that, for almost every realization of a supercritical percolation cluster, strong disorder holds at every positive inverse temperature. Cottini and Nitzschner subsequently showed that there is a non-empty interval of positive inverse temperatures in which strong disorder holds but very strong disorder fails.

In this paper, Theorem~\ref{thm:main-product}(iii) and
Theorem~\ref{thm:radial-regimes}\textup{(i)} provide a further mechanism. Here the reference path is ordinary Brownian motion, the
underlying geometry is homogeneous Euclidean space, and only the spatial
correlation of the environment is changed. Consequently,
\begin{equation}
  W_\infty(\beta)=0
  \quad\text{and}\quad
  p(\beta)=0,
  \qquad\text{for all sufficiently small }\beta>0,
  \label{eq:intro-new-separation-mechanism}
\end{equation}
which provides the new non-coincidence cases. 

\subsection{Free-energy asymptotics}

We first briefly recall some results in i.i.d. environments.  In the
\(1+1\)-dimensional model, Lacoin~\cite{Lacoin2010} proved that the
normalized quenched free energy is of order \(-\beta^4\) at high
temperature.  Under suitable assumptions on the disorder,
Nakashima~\cite{Nakashima2019} obtained the sharp asymptotic
\[
  p(\beta)\sim -\frac{\beta^4}{6},
  \qquad \beta\downarrow0.
\]
The concentration assumption in that result was subsequently removed in
\cite{Nakashima2025}.

The \(1+2\)-dimensional model is marginal.  Berger and
Lacoin~\cite{BergerLacoin2017} proved that
\[
  \lim_{\beta\downarrow0}\beta^2\log[-p(\beta)]=-\pi.
\]
More recently, Berger and Nakajima~\cite{BergerNakajima2026} obtained the
matching estimate
\[
  -p(\beta)\asymp
  \exp\left\{-\frac{\pi}{\sigma^2(\beta)}\right\},
  \]
  where, \(
  \sigma^2(\beta)
  =\exp\{\lambda(2\beta)-2\lambda(\beta)\}-1
  \) and \(\lambda(\beta):=\ln\E[e^{\beta\omega(1,0)}].
\)

For \(d\ge3\), the free energy vanishes throughout the weak-disorder
phase:
\[
  p(\beta)=0,
  \qquad 0\le\beta\le\beta_c.
\]
Thus there is no non-trivial high-temperature asymptotic as
\(\beta\downarrow0\).  The relevant asymptotic behavior occurs instead
near the critical point.  More precisely, Lacoin~\cite{LacoinSmooth2025}
proved that
\[
  \lim_{u\downarrow0}
  \frac{\log[-p(\beta_c+u)]}{\log u}
  =+\infty.
\]
Equivalently, for every \(k>0\),
\[
  -p(\beta_c+u)=o(u^k),
  \qquad u\downarrow0.
\]

We now turn to spatially correlated environments.  For the Brownian
polymer model, Lacoin~\cite{Lacoin2011} proved that a radial covariance
with decay \(|x|^{-\vartheta}\), \(\vartheta<2\), gives \(-p(\beta)\asymp\beta^{4/(2-\vartheta)}\) for \(d\ge2\). In \(d=1\), he also proved \( -p(\beta)\asymp\beta^4\) when the spatial covariance \(Q\) is integrable.  Although the cases \(d=2,\vartheta\ge2\) and \(d=3,\vartheta=2\) are left unanswered in the same work,  the scales are proposed as (see the statements in Remark~1.5 there  )
\begin{align}
  -p(\beta)&\asymp\exp\{-C\beta^{-2}\},
  &\text{for}~& d=2,\ \vartheta>2
     \ \text{or}\ d=3,\ \vartheta=2;
  \notag\\
  -p(\beta)&\asymp\exp\{-C\beta^{-1}\},
  &\text{for}~& d=2,\ \vartheta=2.
  \label{eq:lacoin-remark-scales}
\end{align}

In this paper, we clarify the answer to some extent in these regimes.  In the three-dimensional critical case, strong disorder holds for every
\(\beta>0\), but \(p(\beta)=0\) for all sufficiently small \(\beta\).  Thus the behavior in \(d=3,\vartheta=2\) differs from the scale anticipated in
Remark~1.5 of Lacoin~\cite{Lacoin2011}.  In dimension \(2\), we prove
the scale in the first line of \eqref{eq:lacoin-remark-scales} for
\(\vartheta>2\), and the scale in the second line for \(\vartheta=2\). It is worth noting that, in computing the order of the free energy in $d=2$, we develop a continuous implementation of the multilinear change-of-measure strategy of Berger and Lacoin~\cite{BergerLacoin2017}.

\subsection{Organization}
Section~\ref{sec:lower-bounds} collects the second-moment and pinning estimates needed below. It then applies the replica-coupling comparison
to obtain the upper bound on \(-p(\beta)\) for \(\kappa<2\), as well as
the vanishing of the free energy at the critical product exponent. Section~\ref{sec:fractional} implements the coarse-graining and fractional-moment argument to prove the matching lower bound for \(\kappa<2\) and negativity of the free energy at sufficiently large coupling. Section~\ref{sec:critical-upper} proves the overlap criterion and the endpoint central limit theorem under weak disorder, and then establishes \(\beta_c=0\) for \(\kappa\le2\); the same argument gives the radial case \(d\ge3,\vartheta=2\). Section~\ref{sec:two-dimensional-integrable} deals with the radial regime \(d=2,\vartheta>2\) by an ordered Wiener-chaos change of measure. In Section~\ref{sec:two-dimensional-critical} we consider the critical case \(d=2,\vartheta=2\), where the non-integrability of the covariance requires additional estimates for partial contractions. Appendix~\ref{app:analytic} collects the analytic inequalities used throughout.

\section{Second moments, variational formula, and replica coupling}
\label{sec:lower-bounds}

Write
\begin{equation}
  G_{\bm\alpha}(x):=
  \prod_{j=1}^d(1+|x_j|)^{-\alpha_j}.
  \label{eq:Galpha}
\end{equation}
This section proves Theorem~\ref{thm:main-product}(i), the upper estimate on
\(-p(\beta)\) in Theorem~\ref{thm:main-product}(ii), and the identity
\eqref{eq:critical-zero-free-energy}.

\subsection{The second moment and weak disorder}
\label{subsec:weak-disorder}

\begin{proof}
Let \(B^1,B^2\) be independent Brownian motions.  Integrating the Gaussian
environment, we obtain
\begin{align}
  \E\big[W_T(\beta)^2\big]
  &=
  P^{\otimes2}\left[
    \exp\left\{
      \beta^2\int_0^TQ(B_s^1-B_s^2)\dd s
    \right\}
  \right] \notag\\
  &=
  P\left[
    \exp\left\{
      \beta^2\int_0^TQ(\sqrt2 B_s)\dd s
    \right\}
  \right].
  \label{eq:second-moment}
\end{align}
By Lemma~\ref{lem:product-gaussian-convolution},
\begin{equation*}
  \sup_{x\in\R^d}P_x\big[Q(\sqrt2B_s)\big]
  \le C(1+s)^{-\kappa/2}.
\end{equation*}
If \(\kappa>2\), then
\begin{equation*}
  \eta_\infty
  :=\sup_{x\in\R^d}
  P_x\left[\int_0^\infty Q(\sqrt2B_s)\dd s\right]
  <\infty.
\end{equation*}
By Lemma~\ref{lem:upper}, applied with
\(\lambda=\beta^2\), gives
\begin{equation*}
  \sup_{T\ge0}\E[W_T(\beta)^2]
  \le\frac{1}{1-\beta^2\eta_\infty}<\infty
  \qquad\text{whenever }\beta^2\eta_\infty<1.
\end{equation*}
The martingale \(W_T(\beta)\) then converges in \(L^2\), and its limit has
mean one.  The zero-one law gives \(W_\infty(\beta)>0\) almost surely. This proves Theorem~\ref{thm:main-product}(i).
\end{proof}

\subsection{Feynman--Kac variational formula for pinning free energy}

\begin{lemma}
\label{lem:variational}
Let \(V:\R^d\to[0,\infty)\) be bounded and Borel measurable.  Then, for \(h\ge0\), the limit
\begin{equation}
  \lim_{T\to\infty}\frac1T
  \log P_0\left[
    \exp\left\{h\int_0^T V(B_s)\dd s\right\}
  \right]
  \label{free-energy}
\end{equation}
exists and is denoted by $F_V(h)$, which is referred to as the pinning free energy. Furthermore, we have
\begin{equation}
  F_V(h)
  =
  \sup_{\substack{f\in H^1(\R^d)\\ \norm{f}_2=1}}
  \left\{
    h\int_{\R^d}V(x)f(x)^2\dd x
    -\frac12\norm{\nabla f}_2^2
  \right\}.
  \label{eq:pinning-variational}
\end{equation}
\end{lemma}

\begin{proof}
Let
\[
  \mathcal L_h:=\frac12\Delta+hV
\]
be the Schrödinger operator on \(L^2(\R^d)\), understood through its quadratic
form
\[
  \mathcal E_h(f,f)
  :=
  h\int_{\R^d}V(x)f(x)^2\dd x
  -\frac12\int_{\R^d}|\nabla f(x)|^2\dd x,
  \qquad f\in H^1(\R^d).
\]
Since \(V\) is bounded, this is a closed form bounded from above.  Hence
\(\mathcal L_h\) is self-adjoint and bounded from above, and the spectral
theorem gives
\begin{equation}
  \sup \sigma(\mathcal L_h)
  =
  \sup_{\substack{f\in H^1(\R^d)\\ \norm{f}_2=1}}
  \mathcal E_h(f,f).
  \label{spectral}
\end{equation}

By the Feynman--Kac formula, the semigroup generated by
\(\mathcal L_h\) is
\[
  \bigl(e^{T\mathcal L_h}\varphi\bigr)(x)
  =
  P_x\left[
    \exp\left\{h\int_0^T V(B_s)\,\dd s\right\}
    \varphi(B_T)
  \right]
\]
for every bounded measurable \(\varphi\).  In particular, taking \(\varphi\equiv 1\)
\[
  P_0\left[
    \exp\left\{h\int_0^T V(B_s)\,\dd s\right\}
  \right]
  =
  \bigl(e^{T\mathcal L_h}\mathbf 1\bigr)(0).
\]

Let
\[
  M:=\|V\|_\infty,
  \qquad
  W:=M-V\ge0,
  \qquad
  \widetilde{\mathcal L}_h
  :=
  \frac12\Delta-hW.
\]
Then
\[
  \mathcal L_h
  =
  hM+\widetilde{\mathcal L}_h.
\]
By \cite[Lemma~3.2]{ArendtBatty1996}, consequently,
\[
  \limsup_{T\to\infty}
  \frac1T
  \log
  P_0\left[
    \exp\left\{h\int_0^T V(B_s)\,\dd s\right\}
  \right]
  \le
  hM+\sup\sigma(\widetilde{\mathcal L}_h)
  =
  \sup\sigma(\mathcal L_h).
\]

It remains to prove the reverse inequality.  Let
\(f\in C_c^\infty(\R^d)\) be nonnegative and normalized in \(L^2\).
By the spectral theorem and Jensen's inequality,
\[
  \langle f,e^{T\mathcal L_h}f\rangle
  \ge
  \exp\left\{
    T\langle f,\mathcal L_h f\rangle
  \right\}.
\]
Choose \(f\) so that
\[
  \langle f,\mathcal L_h f\rangle
  \ge
  \sup\sigma(\mathcal L_h)-\varepsilon.
\]
Since the Feynman--Kac kernel is strictly positive, and
\(f\) has compact support, there exists \(C_f<\infty\) such that
\[
  \langle f,e^{T\mathcal L_h}f\rangle
  \le
  C_f\bigl(e^{(T+1)\mathcal L_h}\mathbf 1\bigr)(0).
\]
Therefore,
\[
  \liminf_{T\to\infty}
  \frac1T
  \log
  \bigl(e^{T\mathcal L_h}\mathbf 1\bigr)(0)
  \ge
  \sup\sigma(\mathcal L_h)-\varepsilon.
\]
Letting \(\varepsilon\downarrow0\), we obtain
\[
  \lim_{T\to\infty}
  \frac1T
  \log
  \bigl(e^{T\mathcal L_h}\mathbf 1\bigr)(0)
  =
  \sup\sigma(\mathcal L_h).
\]
Combining this identity with \eqref{spectral} proves
\eqref{eq:pinning-variational}.
\end{proof}

\begin{proposition}
\label{prop:pinning-scale}
Assume \(0<\kappa<2\) and
\[
  c_VG_{\bm\alpha}(x)\le V(x)\le C_VG_{\bm\alpha}(x).
\]
Then there exist \(0<c<C<\infty\) such that, for all sufficiently small \(h>0\),
\begin{equation}
  c h^{2/(2-\kappa)}
  \le F_V(h)\le
  C h^{2/(2-\kappa)}.
  \label{eq:pinning-scale}
\end{equation}
\end{proposition}

\begin{proof}
Let \(f\in H^1(\R^d)\) with \(\norm f_2=1\), by Lemma~\ref{lem:variational}. And Lemma~\ref{lem:product-pitt} imply
\begin{equation*}
  \int_{\R^d}V(x)f(x)^2\dd x
  \le C\norm{\nabla f}_2^\kappa.
\end{equation*}
Writing \(A=\norm{\nabla f}_2^2\), the expression in curely braces in
\eqref{eq:pinning-variational} is bounded above by
\[
  ChA^{\kappa/2}-\frac12A.
\]
Its maximum over \(A\ge0\) is at most \(Ch^{2/(2-\kappa)}\), which proves
the upper bound.

For the reverse bound, choose a nonnegative
\(g\in C_c^\infty((1,2)^d)\) with \(\norm g_2=1\), and set
\(f_R(x)=R^{-d/2}g(x/R)\).  If \(R\ge1\), then
\begin{align*}
  \int_{\R^d}V(x)f_R(x)^2\dd x
  &=\int_{\R^d}V(Ru)g(u)^2\dd u
  \ge cR^{-\kappa},\\
  \norm{\nabla f_R}_2^2
  &=R^{-2}\norm{\nabla g}_2^2.
\end{align*}
Choose \(R=Lh^{-1/(2-\kappa)}\).  Then both terms in
\eqref{eq:pinning-variational} have order \(h^{2/(2-\kappa)}\).  Taking
\(L\) sufficiently large makes the positive coefficient larger than the
negative one, and proves the lower bound.
\end{proof}

At the critical exponent, the same inequality implies the existence of a non-empty region near the origin where \eqref{free-energy} vanishes.

\begin{proposition}
\label{prop:critical-pinning-threshold}
Assume \(\kappa=2\), and suppose that \(0\le V\le C_VG_{\bm\alpha}\).
There is \(C_{\mathrm{form}}<\infty\) such that
\begin{equation}
  \int_{\R^d}V(x)f(x)^2\dd x
  \le C_{\mathrm{form}}\norm{\nabla f}_2^2,
  \qquad f\in H^1(\R^d).
  \label{eq:critical-product-form}
\end{equation}
Consequently,
\begin{equation}
  F_V(h)=0
  \qquad\text{for }0\le h\le
  h_{\mathrm{form}}:=\frac{1}{2C_{\mathrm{form}}}.
  \label{eq:critical-pinning-zero}
\end{equation}
\end{proposition}

\begin{proof}
Equation \eqref{eq:critical-product-form} is
\eqref{eq:Galpha-pitt} with \(\kappa=2\).  Substitution into
\eqref{eq:pinning-variational} gives, for every normalized \(f\),
\[
  h\int V f^2-\frac12\norm{\nabla f}_2^2
  \le
  \left(hC_{\mathrm{form}}-\frac12\right)
  \norm{\nabla f}_2^2.
\]
The right-hand side is nonpositive when
\(h\le(2C_{\mathrm{form}})^{-1}\).  Hence \(F_V(h)\le0\), while
\(F_V(h)\ge0\), proving \eqref{eq:critical-pinning-zero}.
\end{proof}

\subsection{Gaussian replica coupling}
\label{subsec:replica-coupling}

We now compare the quenched free energy with the homogeneous pinning free
energy.  This method was introduced for directed polymers in
\cite{Lacoin2010} and was used for the spatially correlated Brownian polymer model in \cite{Lacoin2011}. Following the computation in \cite[Section~3.1]{Lacoin2011}, which does not depend on the particular form of the covariance kernel, we skip the details and arrive at \eqref{eq:replica-finite-T}

For \(r\in[0,1]\), define
\begin{equation*}
  \Phi_T(r)
  :=
  \frac1T\E\log
  P\left[
    \exp\left\{
      \sqrt r\,\beta H_T(B)-\frac{r\beta^2}{2}T
    \right\}
  \right].
\end{equation*}
Thus \(\Phi_T(0)=0\) and \(\Phi_T(1)=a_T(\beta)/T\), see Section~\ref{subsec:model} for the definition. For
\(\lambda\ge0\), set
\begin{align*}
  \Psi_T(r,\lambda)
  &:=
  \frac1{2T}\E\log P^{\otimes2}
  \exp\Bigg\{
    \sqrt r\,\beta\big(H_T(B^1)+H_T(B^2)\big) \notag
    +\beta^2\int_0^T
      \big(\lambda Q(B_s^1-B_s^2)-r\big)\dd s
  \Bigg\}.
  \label{eq:Psi-definition}
\end{align*}

\begin{proof}
We have
\begin{equation}
  \frac{a_T(\beta)}T=\Phi_T(1)
  \ge-(e-1)\Psi_T(0,2).
  \label{eq:replica-finite-T}
\end{equation}
Since \(B^1-B^2\) has the law of \(\sqrt2B\),
\begin{equation*}
  \Psi_T(0,2)
  =
  \frac1{2T}\log
  P\left[
    \exp\left\{
      2\beta^2\int_0^TQ(\sqrt2B_s)\dd s
    \right\}
  \right].
\end{equation*}
Letting \(T\to\infty\) in \eqref{eq:replica-finite-T} gives the comparison
\begin{equation}
  p(\beta)
  \ge
  -\frac{e-1}{2}
  F_{Q(\sqrt2\,\cdot)}(2\beta^2).
  \label{eq:replica-pinning-comparison}
\end{equation}

If \(\kappa<2\), Proposition~\ref{prop:pinning-scale} applied to
\(V(x)=Q(\sqrt2x)\) yields
\begin{equation}
  p(\beta)\ge-C\beta^{4/(2-\kappa)}.
  \label{eq:power-lower-free-energy}
\end{equation}
It gives the upper estimate on
\(-p(\beta)\) in Theorem~\ref{thm:main-product}(ii).

If \(\kappa=2\), Proposition~\ref{prop:critical-pinning-threshold} gives a
constant \(C_{\mathrm{form}}\) such that
\[
  F_{Q(\sqrt2\,\cdot)}(h)=0
  \qquad\text{for }0\le h\le(2C_{\mathrm{form}})^{-1}.
\]
Hence \eqref{eq:replica-pinning-comparison} and \(p(\beta)\le0\) imply
\begin{equation}
  p(\beta)=0
  \qquad\text{whenever }
  0\le\beta\le\frac{1}{2\sqrt{C_{\mathrm{form}}}}.
  \label{eq:critical-p-zero-proof}
\end{equation}
This proves \eqref{eq:critical-zero-free-energy}.
\end{proof}

\begin{remark}
    When \(d \ge 3\) and \(\vartheta = 2\), it suffices to replace the Pitt-type inequality used in Proposition~\ref{prop:critical-pinning-threshold} by Lemma~\ref{lem:hardy}, and the covariance kernel by the radial potential \eqref{eq:radial-tail-class}; then the analogues of \eqref{eq:critical-pinning-zero} and \eqref{eq:replica-pinning-comparison} hold, and consequently \eqref{eq:riesz-zero-free-energy} follows.
\end{remark}
\section{Fractional moments and measure transformation}
\label{sec:fractional}

We prove the lower estimate on \(-p(\beta)\) in
\eqref{eq:main-power-law}, and then prove negativity of \(p(\beta)\) at
sufficiently large \(\beta\) for every \(\kappa\), which implies \(\bar\beta_c<\infty\). This section follows  the fractional moments + change of measure + coarse-graining framework of~\cite[Section~3.2]{Lacoin2011}, with the
radial estimates replaced by product-kernel estimates. 

\subsection{Coarse paths and a change of measure}
\label{subsec:coarse-graining}

Fix \(T>1\) and \(m\in\N\), and write \(t=mT\).  We split the
time interval \([0,mT]\) into \(m\) time blocks of length \(T\):
\[
  \mathcal I_k^T:=((k-1)T,kT],
  \qquad k=1,\ldots,m.
\]
The coarse graining records only the endpoint of the Brownian path at the
end of each block.  For \(y\in\Z^d\), set
\begin{equation}
  I_y^T := \sqrt T\,(y + [0,1)^d).
  \label{eq:terminal-box}
\end{equation}
Thus a coarse path \(Y=(y_1,\ldots,y_m)\in(\Z^d)^m\), with \(y_0=0\),
means that
\[
  B_{kT}\in I_{y_k}^T,
  \qquad k=1,\ldots,m.
\]
For such a coarse path, define
\begin{align}
  W_{mT}(Y)
  &:=
  P\Bigg[
  \exp\left\{
    \beta H_{mT}(B)-\frac{\beta^2}{2}mT
  \right\}
  \prod_{k=1}^m\1_{\{B_{kT}\in I_{y_k}^T\}}
  \Bigg].
  \label{eq:coarse-path-partition}
\end{align}
Since the boxes \((I_y^T)_{y\in\Z^d}\) partition \(\R^d\),
\[
  W_{mT}(\beta)=\sum_Y W_{mT}(Y).
\]
Therefore, for \(0<\theta<1\), by \((\sum_{j}a_j)^{\theta}\le \sum_{j}a^{\theta}_j\)
\begin{equation}
  \E[W_{mT}(\beta)^\theta]
  \le
  \sum_Y\E[W_{mT}(Y)^\theta].
  \label{eq:coarse-path-fractional}
\end{equation}

We now define the larger spatial sets on which the environment will be
tilted.  Fix a large constant \(L_0>4\).  For \(a\in\Z^d\), let
\begin{equation}
  \mathcal C_a^T
  :=
  \sqrt T\left(a+[-L_0,L_0+1]^d\right).
  \label{eq:hard-corridor}
\end{equation}
The set \(I_a^T\) records the coarse position at one block endpoint,
whereas \(\mathcal C_a^T\) is a corridor: if a block starts from
\(I_a^T\), then a Brownian path with typical fluctuation \(O(\sqrt T)\)
remains inside \(\mathcal C_a^T\).  Hence, on the \(k\)-th block of the
coarse path \(Y\), we tilt the environment in
\[
  \mathcal A_{k,Y}
  :=
  \mathcal I_k^T\times\mathcal C_{y_{k-1}}^T.
\]

Set
\begin{equation}
  \sigma_T^2
  :=
  T\iint_{\mathcal C_0^T\times\mathcal C_0^T}
  Q(u-v)\dd u\dd v.
  \label{eq:corridor-variance}
\end{equation}
For \(1\le k\le m\), define the normalized corridor average
\[
  \Omega_{k,Y}
  :=
  \frac1{\sigma_T}
  \int_{(k-1)T}^{kT}
  \int_{\mathcal C_{y_{k-1}}^T}
  \omega(\dd s,u)\dd u.
\]
Each \(\Omega_{k,Y}\) is a standard Gaussian random variable.  Moreover,
the variables \(\Omega_{1,Y},\ldots,\Omega_{m,Y}\) are independent, since
the corresponding time intervals are disjoint.

For this fixed coarse path \(Y\), we change the environment law by
\begin{equation}
  \frac{\dd\Ptilde_Y}{\dd\Pp}
  =
  \exp\left\{
    -\sum_{k=1}^m\Omega_{k,Y}-\frac m2
  \right\}.
  \label{eq:global-hard-change}
\end{equation}
Changing measure and applying
H\"older's inequality gives
\begin{equation}
  \E[W_{mT}(Y)^\theta]
  \le
  \exp\left\{
    \frac{m\theta}{2(1-\theta)}
  \right\}
  \big(\Etilde_Y W_{mT}(Y)\big)^\theta.
  \label{eq:global-holder}
\end{equation}
For later use, denote by \(\widetilde{\mathbb P}\) the one-block tilted
law
\begin{equation}
  \frac{\dd\widetilde{\mathbb P}}{\dd\mathbb P}
  =
  \exp\left\{
    -\frac1{\sigma_T}
    \int_0^T\int_{\mathcal C_0^T}\omega(\dd s,u)\dd u
    -\frac12
  \right\}.
\end{equation}

\subsection{Energy lowering for paths in one block}
\label{subsec:corridor-estimates}


\begin{lemma}
\label{lem:hard-corridor-estimates}
Assume \eqref{eq:product-tail}.  There exists a constant \(c>0\), depending
on \(L_0\), such that if \(B_0\in I_0^T\) and
\[
  \mathcal G_T
  :=
  \left\{
    \sup_{0\le s\le T}|B_s-B_0|_\infty
    \le (L_0-3)\sqrt T
  \right\},
\]
then, under the one-block tilted measure associated with
\(\mathcal C_0^T\),
\begin{equation}
  -\Etilde\left[\int_0^T\omega(\dd s,B_s)\right]
  \ge
  cT^{(2-\kappa)/4}
  \quad\text{on }~\mathcal G_T.
  \label{eq:hard-corridor-gain}
\end{equation}
\end{lemma}

\begin{proof}
For
\(R\ge1\), since \(\alpha_j\in(0,1)\),
\[
  \int_{[-R,R]^d}G_{\bm\alpha}(u)\dd u
  =
  \prod_{j=1}^d
  \int_{-R}^R(1+|u_j|)^{-\alpha_j}\dd u_j
  \asymp R^{d-\kappa}.
\]
Hence, by \eqref{eq:product-tail}, uniformly for
\(x\in\sqrt T[-L_0+2,L_0-1]^d\),
\begin{equation}
  \int_{\mathcal C_0^T}Q(u-x)\dd u
  \ge
  cT^{(d-\kappa)/2}.
  \label{eq:corridor-single-lower}
\end{equation}
Similarly,
\begin{equation}
  \iint_{\mathcal C_0^T\times\mathcal C_0^T}
  Q(u-v)\dd u\dd v
  \asymp
  T^{d-\kappa/2}.
  \label{eq:corridor-double-order}
\end{equation}

By the Cameron--Martin formula, on one block,
\[
  \widetilde{\E}[\omega(\dd s,x)]
  =
  -\frac1{\sigma_T}
  \int_{\mathcal C_0^T}Q(u-x)\dd u\,\dd s.
\]
Therefore, for a fixed path \(B\),
\begin{equation}
  -\Etilde\left[\int_0^T\omega(\dd s,B_s)\right]
  =
  \frac1{\sigma_T}
  \int_0^T\int_{\mathcal C_0^T}
  Q(u-B_s)\dd u\dd s.
  \label{eq:hard-mean-response}
\end{equation}
On \(\mathcal G_T\), since \(B_0\in I_0^T\), we have
\[
  B_s\in \sqrt T[-L_0+2,L_0-1]^d,
  \qquad 0\le s\le T.
\]
Thus \eqref{eq:corridor-single-lower} gives
\[
  \int_0^T\int_{\mathcal C_0^T}Q(u-B_s)\dd u\dd s
  \ge
  cT^{1+(d-\kappa)/2}.
\]
On the other hand, by the definition of \(\sigma_T\) and
\eqref{eq:corridor-double-order},
\[
  \sigma_T^2
  =
  T
  \iint_{\mathcal C_0^T\times\mathcal C_0^T}
  Q(u-v)\dd u\dd v
  \asymp
  T^{1+d-\kappa/2}.
\]
Substituting these two estimates into \eqref{eq:hard-mean-response} yields
\[
  -\Etilde\left[\int_0^T\omega(\dd s,B_s)\right]
  \ge
  c\frac{T^{1+(d-\kappa)/2}}
          {T^{(1+d-\kappa/2)/2}}
  =
  cT^{(2-\kappa)/4}.
\]
This proves \eqref{eq:hard-corridor-gain}.
\end{proof}

\subsection{Upper bound for the second term in \eqref{eq:global-holder}}
For \(z\in\Z^d\), let
\[
g_T(z) := \sup_{x\in I_0^T}P_x(B_T\in I_z^T), \qquad
h_{T,L_0}(z) := \sup_{x\in I_0^T} P_x(B_T\in I_z^T,\mathcal G_T^c).
\]
Gaussian tails and Dominated Convergence Theorem imply
\begin{equation}
  \sum_{z\in\Z^d}g_T(z)^\theta<\infty
  \qquad\text{and}\qquad
  \lim_{L_0\to\infty}
  \sum_{z\in\Z^d}h_{T,L_0}(z)^\theta
  =0,
  \label{ghT:estimate}
\end{equation}
uniformly in \(T\). Let \(z_k=y_k-y_{k-1}\).  The Markov property and spatial translation invariance, under \(\Ptilde_Y\), give
\begin{equation}
  \Etilde_YW_{mT}(Y)
  \le
  \prod_{k=1}^m\Gamma_T(z_k),
  \label{eq:shifted-block-product}
\end{equation}
where
\begin{equation*}
  \Gamma_T(z)
  :=
  \sup_{x\in I_0^T}
  P_x\left[
    \exp\left\{
      \beta\Etilde\int_0^T\omega(\dd s,B_s)
    \right\}
    \1_{\{B_T\in I_z^T\}}
  \right].
\end{equation*}
Since $Y \longleftrightarrow Z=(z_1, \ldots, z_m)$ is a bijection, raising to the power \(\theta\) and summing over the coarse path \(Y\) in \eqref{eq:shifted-block-product},
\begin{equation}
    \sum_{Y} [\Etilde_YW_{mT}(Y)]^{\theta}\le \sum_{Z}\prod_{k=1}^m\Gamma_T(z_k)^{\theta}=(\sum_{z\in \Z^d} \Gamma_T(z)^{\theta})^m
    \label{bijection transformation}
\end{equation}
Splitting according to \(\mathcal G_T\) and using
\eqref{eq:hard-corridor-gain},
\begin{equation*}
  \Gamma_T(z)
  \le
  \exp\left\{-c\beta T^{(2-\kappa)/4}\right\}g_T(z)
  +h_{T,L_0}(z).
\end{equation*}
Since \((a+b)^\theta\le a^\theta+b^\theta\),
\begin{equation}
  \sum_{z\in\Z^d}\Gamma_T(z)^\theta
  \le
  \e^{-c\theta\beta T^{(2-\kappa)/4}}
  \sum_zg_T(z)^\theta
  +
  \sum_zh_{T,L_0}(z)^\theta.
  \label{eq:Gamma-sum}
\end{equation}

\subsection{The proof of the upper bound of Theorem \ref{thm:main-product} (ii)}
\label{subsec:power-upper-completion}

Choose first \(L_0\) large so that the second term in
\eqref{eq:Gamma-sum} is smaller than
\(
\frac14\exp\{-\theta/[2(1-\theta)]\}
\).
Next choose \(K>1\) large so  that
$  T_\beta:=K\beta^{-4/(2-\kappa)}$ 
makes the first term smaller than \(\frac14\exp\{-\theta/[2(1-\theta)]\}
\).  Indeed,
\(
\beta T_\beta^{(2-\kappa)/4}=K^{(2-\kappa)/4}
\),
which is independent of \(\beta\) and tends to infinity with \(K\).
Thus
\begin{equation*}
  \exp\left\{
    \frac{\theta}{2(1-\theta)}
  \right\}
  \sum_z\Gamma_{T_\beta}(z)^\theta
  =:\rho<1/2.
\end{equation*}

Combining \eqref{eq:coarse-path-fractional},
\eqref{eq:global-holder}, \eqref{eq:shifted-block-product}, and \eqref{bijection transformation} gives
\begin{equation}
  \E[W_{mT_\beta}(\beta)^\theta]
  \le\rho^m.
  \label{eq:hard-fractional-decay}
\end{equation}
Therefore
\begin{equation*}
  p(\beta)
  \le
  \liminf_{m\to\infty}
  \frac1{\theta mT_\beta}
  \log\E[W_{mT_\beta}(\beta)^\theta]
  \le
  \frac{\log\rho}{\theta T_\beta}
  \le
  -c\beta^{4/(2-\kappa)}.
\end{equation*}
Together with \eqref{eq:power-lower-free-energy}, this proves
Theorem~\ref{thm:main-product}(ii).

\subsection{The proof of \(\bar\beta_c<\infty\)}
\label{subsec:critical-large-beta}

The same one-block change of measure also shows that the very strong
disorder critical point is finite for every product kernel covered by
\eqref{eq:product-tail}.  Indeed, the proof of
Lemma~\ref{lem:hard-corridor-estimates} only requires \(\kappa<d\), which follows from \(\alpha_j<1\) for all \(j\). We fix block length \(T_0=2\).  Then \eqref{eq:Gamma-sum} gives
\begin{equation}
  \sum_{z\in\Z^d}\Gamma_{T_0}(z)^\theta
  \le
  e^{-c\theta\beta}\sum_{z\in\Z^d}g_{T_0}(z)^\theta
  +\sum_{z\in\Z^d}h_{T_0,L_0}(z)^\theta .
  \label{eq:critical-large-beta-gamma}
\end{equation}
Set
\[
  K_\theta:=\exp\left\{\frac{\theta}{2(1-\theta)}\right\}.
\]
By \eqref{ghT:estimate}, we may first choose \(L_0\) so large that
\[
  K_\theta\sum_{z\in\Z^d}h_{T_0,L_0}(z)^\theta<\frac14 .
\]
With this \(L_0\) fixed, the first term in
\eqref{eq:critical-large-beta-gamma} tends to \(0\) as
\(\beta\to\infty\).  Hence there exists \(\beta_1<\infty\) such that, for
all \(\beta\ge\beta_1\),
\[
  K_\theta e^{-c\theta\beta}
  \sum_{z\in\Z^d}g_{T_0}(z)^\theta<\frac14 .
\]
Consequently,
\begin{equation}
  K_\theta\sum_{z\in\Z^d}\Gamma_{T_0}(z)^\theta
  \le \frac12 .
  \label{eq:large-beta-one-block-contraction}
\end{equation}
Then \eqref{eq:hard-fractional-decay} 
implies
\[
  \E[W_{mT_0}(\beta)^\theta]\le 2^{-m}.
\]
By Jensen's inequality,
\[
  \frac1{mT_0}\E\log W_{mT_0}(\beta)
  \le
  \frac1{\theta mT_0}\log \E[W_{mT_0}(\beta)^\theta]
  \le
  -\frac{\log 2}{\theta T_0}.
\]
Letting \(m\to\infty\), we obtain
\begin{equation}
  p(\beta)<0,
  \qquad \beta\ge\beta_1.
  \label{eq:critical-large-beta-negative}
\end{equation}
Thus \(\bar\beta_c<\infty\). It gives the finiteness of \(\bar\beta_c\) in Theorem~\ref{thm:main-product}. 

\begin{remark}
It is worth noting that the above computations remain valid for the covariance given by the radial potential \eqref{eq:radial-tail-class}, up to changes in the parameters in \eqref{eq:corridor-single-lower} and \eqref{eq:corridor-double-order}; the conclusion \(\bar\beta_c<\infty\) in Theorem~\ref{thm:radial-regimes}(i) follows analogously.
\end{remark}

\section{Proof of \(\beta_c=0\) in Theorem~\ref{thm:main-product}(ii) and (iii)}
\label{sec:critical-upper}

In this section we shall prove that \(\beta_c=0\) when \(\kappa\le2\).  To this end, let
\begin{equation}\label{terminal distribution}
  \nu_t^\omega(A)
  :=\mu_{t,\omega}^{\beta}(B_t\in A),
  \qquad A\subset{\mathcal B}(\R^d),
\end{equation}
be the quenched endpoint law, see \eqref{eq:polymer-measure} for the definition.  Its replica overlap is
\begin{equation}
  I_t
  :=\iint_{\R^d\times\R^d}
  Q(x-y)\nu_t^\omega(\dd x)\nu_t^\omega(\dd y)
  =
  \mu_{t,\omega}^{\beta,\otimes2}
  \big[Q(B_t^1-B_t^2)\big].
  \label{eq:instantaneous-overlap}
\end{equation}
Since \(Q\) is positive definite and \(Q(0)=1\), one has
\(0\le I_t\le1\).

\begin{lemma}[Overlap criterion]
\label{prop:overlap-criterion}
For every \(\beta>0\), consider the following three statements:
\begin{enumerate}[ leftmargin=0pt, labelsep=5pt, align=left]
  \item \(W_\infty(\beta)>0\) \quad \text{almost surely};
  \item \(\displaystyle \int_0^\infty I_s\,\dd s<\infty \quad\text{almost surely}\);
  \item \(\displaystyle W_t(\beta)\longrightarrow W_\infty(\beta) \quad\text{in }L^1(\Pp)\).
\end{enumerate}
Then \textup{(1)} and \textup{(2)} are equivalent, and both are implied by
\textup{(3)}. This is Proposition~2.8 of Rovira and
Tindel~\cite{RoviraTindel2005}.
\end{lemma}

\subsection{Endpoint diffusivity under weak disorder}
\label{subsec:weak-endpoint-CLT}

We need an endpoint central limit theorem. The corresponding theorem for lattice polymers is
\cite[Theorem~1.2]{CometsYoshida2006}.
For completeness, we give a short proof.

We begin by verifying the uniform integrability needed in the sequel. In the case of lattice polymers, this was proved in \cite[Proposition~3.1]{CometsYoshida2006}; for the short-range continuum polymers, the corresponding result is \cite[Proposition~4.3]{mukherjee2016weak}.

\begin{lemma}
\label{lem:weak-UI}
In the present Brownian polymer model,
\begin{equation}
  W_\infty>0~\ \text{a.s.}
  \quad\Longleftrightarrow\quad
  (W_t)_{t\ge0}\ \text{is uniformly integrable}.
  \label{eq:weak-UI-equivalence}
\end{equation}
In particular, under weak disorder,
\begin{equation}
  W_t\longrightarrow W_\infty
  \quad\text{in }L^1(\Pp),
  \qquad
  \E W_\infty=1.
  \label{eq:weak-L1}
\end{equation}
\end{lemma}

\begin{proof}
Uniform integrability implies \(W_t\to W_\infty\) in \(L^1\), hence
\(\E W_\infty=1\). By the zero-one law for \(W_\infty\), this gives
\(W_\infty>0\) almost surely.

Conversely assume \(W_\infty>0\) almost surely, and set
\(c:=\E W_\infty\in(0,1]\). Let \((\cF_t)_{t\ge0}\) be the filtration
generated by the environment up to time \(t\). For \(s,t\ge0\) and
\(x\in\R^d\), define the partition function on \( (t,t+s]\)
\[
  W_s^{t,x}
  :=
  P_x\left[
    \exp\left\{
      \beta\int_t^{t+s}\omega(\dd r,B_{r-t})
      -\frac{\beta^2}{2}s
    \right\}
  \right],
\]
and let \(W_\infty^{t,x}:=\lim_{s\to\infty}W_s^{t,x}\). By independence of
disjoint time intervals and space--time stationarity,
\[
  W_\infty^{t,x}\ \text{is independent of }\cF_t,
  \qquad
  \E W_\infty^{t,x}=c .
\]

Write
\[
  Z_t^\omega(\dd x)
  :=
  P\left[
    \exp\left\{\beta H_t(B)-\frac{\beta^2}{2}t\right\}
    \1_{\{B_t\in\dd x\}}
  \right].
\]
Then \(Z_t^\omega(\R^d)=W_t\), and the Markov property gives
\[
  W_{t+s}
  =
  \int_{\R^d}Z_t^\omega(\dd x)\,W_s^{t,x}.
\]
Letting \(s\to\infty\) and using Fatou's lemma,
\[
  \int_{\R^d}Z_t^\omega(\dd x)\,W_\infty^{t,x}
  \le W_\infty
  \qquad\text{a.s.}
\]
Taking expectations gives equality, because
\[
  \E\left[
    \int_{\R^d}Z_t^\omega(\dd x)\,W_\infty^{t,x}
  \right]
  =
  c\,\E Z_t^\omega(\R^d)
  =
  c\,\E W_t
  =
  c
  =
  \E W_\infty .
\]
Thus the Fatou inequality is an equality almost surely. Taking conditional
expectation with respect to \(\cF_t\), we obtain
\[
  \E[W_\infty\mid\cF_t]
  =
  \int_{\R^d}Z_t^\omega(\dd x)\,\E W_\infty^{t,x}
  =
  cW_t .
\]
Hence
\[
  W_t=\frac1c\,\E[W_\infty\mid\cF_t].
\]
The right-hand side is uniformly integrable and converges almost surely to
\(W_\infty/c\), while \(W_t\to W_\infty\) almost surely. Since
\(W_\infty>0\) almost surely, we get \(c=1\). Therefore
\(W_t=\E[W_\infty\mid\cF_t]\), so \((W_t)_{t\ge0}\) is uniformly integrable
and \(W_t\to W_\infty\) in \(L^1(\Pp)\).
\end{proof}

For \(v\in\R^d\), let \(P^{(v)}\) denote Brownian motion with constant drift
\(v\), and define
\begin{equation}
  W_t^{(v)}
  :=
  P^{(v)}\left[
    \exp\left\{\beta H_t(B)-\frac{\beta^2}{2}t\right\}
  \right].
  \label{eq:drifted-partition}
\end{equation}
For fixed \(v\), this is again a positive mean-one martingale.  Introduce
the shift
\begin{equation*}
  (\tau_v\omega)(\dd s,x):=\omega(\dd s,x+vs).
\end{equation*}
The white-in-time covariance gives
\[
  \E[(\tau_v\omega)(\dd s,x)(\tau_v\omega)(\dd r,y)]
  =
  \delta_0(s-r)Q(x-y)\dd s\dd r.
\]
Thus \(\tau_v\omega\) has the same law as \(\omega\), and
\begin{equation}
  W_t^{(v)}(\omega)=W_t^{(0)}(\tau_v\omega).
  \label{eq:drift-shear-identity}
\end{equation}
We define \(W_\infty^{(v)}=W_\infty^{(0)}\circ\tau_v\).  It follows from
\eqref{eq:weak-L1} and \eqref{eq:drift-shear-identity} that
\begin{equation}
  \sup_{v\in\R^d}
  \E\big|W_t^{(v)}-W_\infty^{(v)}\big|
  =
  \E\big|W_t^{(0)}-W_\infty^{(0)}\big|
  \longrightarrow0.
  \label{eq:uniform-drift-L1-tail}
\end{equation}

We next record continuity in the drift.  For every fixed \(T<\infty\),
\begin{equation}
  \lim_{v\to0}\E\big|W_T^{(v)}-W_T^{(0)}\big|=0.
  \label{eq:finite-drift-continuity}
\end{equation}
Indeed, under \(P^{(v)}\) write the path as \(B_s+vs\). For a fixed Brownian
path, put
\[
  X_v:=H_T(B+v\,\cdot),
  \qquad X_0:=H_T(B).
\]
Then
\[
  \Var(X_v)=\Var(X_0)=T,
  \qquad
  \Cov(X_v,X_0)=c_T(v):=\int_0^TQ(vs)\dd s.
\]
By Cauchy–Schwarz inequality and Jensen’s inequality, consequently,
\begin{align*}
  \E|W_T^{(v)} - W_T^{(0)}|
  &\le \left( \E|W_T^{(v)} - W_T^{(0)}|^2 \right)^{1/2} \\
  &\le \left( \E\left[ \left( e^{\beta X_v - \beta^2 T/2} - e^{\beta X_0 - \beta^2 T/2} \right)^2 \right] \right)^{1/2} \\
  &= \left( 2e^{\beta^2 T} - 2e^{\beta^2 c_T(v)} \right)^{1/2} \longrightarrow 0,
\end{align*}
because \(Q\) is continuous at zero and \(Q(0)=1\).  This proves
\eqref{eq:finite-drift-continuity}.

Equations \eqref{eq:uniform-drift-L1-tail} and
\eqref{eq:finite-drift-continuity} imply continuity of the terminal values, i.e.,
\begin{equation}
  \lim_{v\to0}
  \E\big|W_\infty^{(v)}-W_\infty^{(0)}\big|=0.
  \label{eq:infinite-drift-continuity}
\end{equation}
Combining
\eqref{eq:uniform-drift-L1-tail} and
\eqref{eq:infinite-drift-continuity}, for every deterministic \(v_t\to0\), when \(t\to\infty\),
\begin{equation}
  \E\big|W_t^{(v_t)}-W_t^{(0)}\big|\longrightarrow0.
  \label{eq:vanishing-drift-comparison}
\end{equation}

\begin{proposition}
\label{prop:weak-endpoint-CLT}
Assume the weak disorder holds.  Recall the definition \eqref{terminal distribution} of the measure $\nu_t^\omega$. Then the random probability measures
\begin{equation}
  \widehat\nu_t^\omega(A)
  :=\nu_t^\omega(\sqrt t\,A), 
  \label{eq:scaled-endpoint-law}
\end{equation}
converge (\(\Pp\)) as $t\to\infty$, for the topology of weak convergence, to
the standard Gaussian law \(\gamma_d=N(0,I_d)\).
\end{proposition}

\begin{proof}
Let \(\lambda\in\R^d\), and put \(v_t=\lambda/\sqrt t\). Using \eqref{eq:polymer-measure} and
\eqref{eq:drifted-partition}, by the Brownian Cameron--Martin formula,
\[
\int_{\R^d} e^{\lambda\cdot x}\widehat\nu_t^\omega(\dd x)
=
\frac{
  P\left[
    \exp\left\{
      \beta H_t(B)-\frac{\beta^2}{2}t
    \right\}
    \exp\left\{\lambda\cdot \frac{B_t}{\sqrt t}\right\}
  \right]}
  {W_t^{(0)}}
=
e^{|\lambda|^2/2}\,
\frac{W_t^{(v_t)}}{W_t^{(0)}} .
\]
By \eqref{eq:vanishing-drift-comparison},
\[
  W_t^{(v_t)}-W_t^{(0)}
  \longrightarrow 0
  \qquad\text{in }L^1(\Pp),
\]
hence also in \(\Pp\)-probability.  Since weak disorder gives
\[
  W_t^{(0)}\longrightarrow W_\infty^{(0)}>0
  \qquad\text{in }\Pp\text{-probability},
\]
we have
\[
  \frac{W_t^{(v_t)}}{W_t^{(0)}}
  =
  1+\frac{W_t^{(v_t)}-W_t^{(0)}}{W_t^{(0)}}
  \longrightarrow 1
  \qquad\text{in }\Pp\text{-probability}.
\]
Indeed, for any \(\varepsilon,\eta>0\),
\[
  \Pp\left(
    \left|
      \frac{W_t^{(v_t)}-W_t^{(0)}}{W_t^{(0)}}
    \right|>\varepsilon
  \right)
  \le
  \Pp(W_t^{(0)}\le\eta)
  +
  \Pp(|W_t^{(v_t)}-W_t^{(0)}|>\varepsilon\eta),
\]
then the right hand to \(0\) as \(\eta\downarrow0\) and \(t\to\infty\) .
Therefore
\begin{equation}
  \int_{\R^d} e^{\lambda\cdot x}\widehat\nu_t^\omega(\dd x)
  \longrightarrow
  e^{|\lambda|^2/2}
  \qquad\text{in }\Pp\text{-probability as }t\to\infty.
  \label{eq:quenched-Laplace-limit}
\end{equation}

It remains only to pass from the Laplace transforms to weak convergence.  We use
a standard subsequence argument.  Take
\(\mathcal D=\mathbb Q^d\), which is dense in \(\mathbb R^d\). From \eqref{eq:quenched-Laplace-limit}, every sequence
\(t_n\to\infty\) contains a subsequence, still denoted \(t_n\), such that
for every \(\lambda\in\mathcal D\),
\[
  \int_{\R^d} e^{\lambda\cdot x}\widehat\nu_{t_n}^\omega(\dd x)
  \longrightarrow e^{|\lambda|^2/2}
  \qquad\text{almost surely}.
\]
Fix an environment \(\omega\) in this almost sure event and write \(\nu_n:=\widehat\nu_{t_n}^\omega\). For every coordinate vector \(e_j\), there exists \(n \ge n_0(\omega)\) such that the convergence above with \(\lambda=\pm e_j\) gives
\[
  \sup_{n\ge n_0(\omega)}
  \int_{\mathbb R^d}e^{\pm x_j}\nu_n(\dd x)
  \le C(\omega).
\]
Consequently, for every \(R>0\),
\[
  \nu_{n}(|x_j|>R)
  \le
  e^{-R}\int e^{x_j}\nu_{n}(\dd x)
  +
  e^{-R}\int e^{-x_j}\nu_{n}(\dd x),
\]
and the right-hand side is bounded by \(C(\omega)e^{-R}\) for all large
\(n\), hence \((\nu_n)\) is tight. Let \(\nu_{n_k}\) be any further weakly convergent subsequence. Fix \(\lambda\in \mathcal D\). Since \(2\lambda\in \mathcal D\), \eqref{eq:quenched-Laplace-limit} implies 
\[
 \sup_k \int e^{2\lambda \cdot x} \nu_{n_k}(\dd x) \le C_{\lambda}(\omega)<\infty.
\]
Moreover, for every \(M>0\), the function \(e^{2\lambda \cdot x}\wedge M\) is bounded and continuous. Therefore, by weak convergence,
\[
\int (e^{2\lambda \cdot x}\wedge M) \rho(\dd x)=\lim_{k}\int (e^{2\lambda \cdot x}\wedge M)\nu_{n_k}(\dd x)\le C_{\lambda}(\omega).
\]
Letting \(M\rightarrow\infty\) and using monotone convergence theorem gives
\[
\int e^{2\lambda \cdot x}\rho(\dd x)\le C_{\lambda}(\omega).
\]
Thus \(e^{\lambda \cdot x}\) is uniformly integrable w.r.t \((\nu_{n_k})_k\) and \(\rho\). Indeed, for \(R>0\), 
\[
\sup_k\int_{\{e^{\lambda\cdot x}>R\}} e^{\lambda \cdot x}\nu_{n_k}(\dd x)\le \frac{C_{\lambda}(\omega)}{R}.
\]
and the same bound holds with \(\rho\) in place of \(\nu_{n_k}\).

Let \(f_R(x):=e^{\lambda \cdot x}\wedge R\). Since \(f_R\) is bounded and continuous, we have \(\int f_R \dd \nu_{n_k}\rightarrow \int f_R \dd\rho \). The uniform-integrability allows us to let \(R\rightarrow\infty\), and yields
\[
  \int_{\mathbb R^d}e^{\lambda\cdot x}\rho(\dd x)
  =
  \lim_{k\to\infty}
  \int_{\mathbb R^d}e^{\lambda\cdot x}\nu_{n_k}(\dd x)
  =
  e^{|\lambda|^2/2},
  \qquad \lambda\in\mathcal D.
\]
By continuity in \(\lambda\), this identity holds for all \(\lambda\in\R^d\).
The Gaussian law is determined by its Laplace transform, so
\(\rho=\gamma_d\).  Therefore
\begin{equation}
  \widehat\nu_{t_n}^\omega \Rightarrow \gamma_d
  \qquad\text{almost surely}.
  \label{susequent_weak_converg}
\end{equation}
Since every sequence \(t_n\to\infty\) has a subsequence along which this
almost sure convergence holds, the full family converges to \(\gamma_d\) in
\(\Pp\)-probability for the topology of weak convergence.
\end{proof}

\subsection{Proof of \(\beta_c=0\) in regime \(d\geq3,\kappa\leq2\)}

\begin{proposition}
\label{prop:strong-kappa-le-two}
Assume \eqref{eq:product-tail} and \(\kappa\le2\).  Then
\[
  W_\infty(\beta)=0
  \qquad\text{almost surely for every }\beta>0.
\]
\end{proposition}

\begin{proof}
Suppose, to the contrary, that weak disorder holds at some \(\beta>0\).
Let
\[
  A_+:=[1,2]^d,
  \qquad
  A_-:=[-2,-1]^d,
  \qquad
  a:=\gamma_d(A_+)=\gamma_d(A_-)>0.
\]
Both boxes are continuity sets for \(\gamma_d\).  By
Proposition~\ref{prop:weak-endpoint-CLT},
\begin{equation}
  \Pp\left(
    \widehat\nu_t^\omega(A_+)\ge\frac a2,\quad
    \widehat\nu_t^\omega(A_-)\ge\frac a2
  \right)
  \longrightarrow1
  \qquad\text{as }t\to\infty.
  \label{eq:two-box-mass}
\end{equation}
If \(x\in\sqrt t\,A_+\), \(y\in\sqrt t\,A_-\), and \(t\ge1\), then
\(2\sqrt t\le x_j-y_j\le4\sqrt t\) for every \(j\).  Therefore
\begin{equation}
  Q(x-y)
  \ge
  c_Q\prod_{j=1}^d(1+4\sqrt t)^{-\alpha_j}
  \ge c\,t^{-\kappa/2}.
  \label{eq:two-box-Q-lower}
\end{equation}
Restricting the double integral in \eqref{eq:instantaneous-overlap} to
\(\sqrt t\,A_+\times\sqrt t\,A_-\), equations
\eqref{eq:two-box-mass} and \eqref{eq:two-box-Q-lower} imply
\begin{equation}
  \mathbb P\left(I_t\ge c_0t^{-\kappa/2}\right)\longrightarrow1
  \qquad\text{for some }c_0>0  \qquad\text{as }t\to\infty.
  \label{eq:overlap-lower-probability}
\end{equation}

For \(\kappa\le2\), we have \(c_0t^{-\kappa/2}\ge c_0t^{-1}\) for \(t\ge1\). Define
\[
X_T:=\int_T^{2T}\mathbf{1}_{\{sI_s<c_0\}}\frac{ds}{s}.
\]
By Fubini's theorem and \(s=Tu\),
\[
\mathbb E X_T=\int_1^2\mathbb P(Tu I_{Tu}<c_0)\frac{du}{u}\to0,
\quad\text{hence }X_T\xrightarrow{\mathbb P}0.
\]
by \eqref{eq:overlap-lower-probability} and dominated convergence. We notice that:
\[
\int_T^{2T}I_s\,ds
\ge
\int_{T}^{2T} \mathbf{1}_{\{sI_s\ge c_0\}}\frac{c_0}{s}\,ds
=
c_0\left(\int_T^{2T}\frac{ds}{s}-\int_T^{2T}\mathbf{1}_{\{sI_s<c_0\}}\frac{ds}{s}\right)
=
c_0(\log2-X_T).
\]
Since \(X_T\to0\) in probability,
\begin{equation}
  \mathbb P\left(\int_T^{2T}I_s\,ds\ge \frac{c_0\log2}{2}\right)\to1
  \qquad\text{as }T\to\infty.
  \label{eq:dyadic-overlap-positive}
\end{equation}
But weak disorder and \eqref{prop:overlap-criterion} give \(\int_0^\infty I_s\,ds<\infty\) a.s., hence \(\int_T^{2T}I_s\,ds\to0\) a.s., contradicting \eqref{eq:dyadic-overlap-positive}. Thus weak disorder is impossible.
\end{proof}

\begin{remark}
    Note that from \eqref{eq:instantaneous-overlap} -- \eqref{susequent_weak_converg} in this section, no specific shape of the covariance is assumed; hence the same arguments remain valid when the covariance is taken to be the radial potential \eqref{eq:radial-tail-class}. Therefore, it suffices to verify the steps from \eqref{eq:two-box-mass} -- \eqref{eq:overlap-lower-probability} in order to prove \(\beta_c=0\) in Theorem~\ref{thm:radial-regimes}(i). This is straightforward, and we omit the details.
\end{remark}
\section{Two-dimensional integrable spatial correlations}
\label{sec:two-dimensional-integrable}

In this section we complete the proof of Theorem~\ref{thm:radial-regimes}(ii). Throughout the section, \(d=2,\vartheta>2\), which means \(Q\in L^1(\R^2)\), and \(p_t\) denotes the Brownian heat kernel
\begin{equation}
  p_t(x):=\frac{1}{2\pi t}\exp\left\{-\frac{|x|^2}{2t}\right\}.
  \label{eq:two-d-heat-kernel}
\end{equation}
We use the multilinear change-of-measure strategy of Berger and Lacoin\cite{BergerLacoin2017}, implemented here through an anchored continuous heat-kernel chain in a spatially correlated Gaussian environment.

\subsection{Some estimates used in Section~\ref{proof of Lemma 5.3}}
We shall use the following elementary facts, which follow from \(Q\in L^1(\R^2)\cap L^\infty(\R^2)\). Since \(0\le Q\le1\),
\begin{equation}
  \sup_z(p_t*Q)(z)
  \le \min\left\{\|Q\|_\infty,\|p_t\|_\infty\|Q\|_1\right\}
  \le \min\left\{1,\frac{\|Q\|_1}{2\pi t}\right\}
  \le \frac{C_Q}{1+t}.
  \label{eq:uniform-pQ}
\end{equation}

\begin{lemma}
\label{lem:gaussian-bridge-insertion}
There is \(C<\infty\) such that, for \(a,b>0\) and
\(x,y,w\in\R^2\),
\begin{equation}
  \int_{\R^2}
  p_a(x-z)Q(z-w)p_b(z-y)\dd z
  \le
  \frac{C}{1+\min(a,b)}p_{a+b}(x-y).
  \label{eq:bridge-insertion}
\end{equation}
Consequently, for \(k\ge 0\), \(L>0\), and
\(x,y,w_1,\ldots,w_k\in\R^2\), define
\[
  \Delta_k(L):=\{0<s_1<\cdots<s_k<L\},
\]
with the convention \(s_0=0\), \(s_{k+1}=L\), \(z_0=x\), and
\(z_{k+1}=y\).  Then
\begin{align}
  &\int_{\Delta_k(L)}
    \int_{(\R^2)^k}
    \prod_{i=0}^{k}
      p_{s_{i+1}-s_i}(z_i-z_{i+1})
    \prod_{i=1}^{k}
      Q(z_i-w_i)
    \dd z_1\cdots \dd z_k\,
    \dd s_1\cdots \dd s_k
  \notag\\
  &\hspace{3cm}
  \le
  \big[C\log(2+L)\big]^k p_L(x-y).
  \label{eq:ordered-insertion-cost}
\end{align}
\end{lemma}

\begin{proof}
The Brownian bridge identity gives
\begin{equation}
  p_a(x-z)p_b(z-y)
  =
  p_{a+b}(x-y)p_\tau(z-\mu),
  \qquad
  \tau=\frac{ab}{a+b},\quad
  \mu=\frac{bx+ay}{a+b}.
  \label{eq:gaussian-bridge-identity}
\end{equation}
Since
\(
\tau\ge\frac12\min(a,b)
\),
Equation~\eqref{eq:uniform-pQ} implies
\[
  (p_\tau*Q)(\mu-w)
  \le\frac{C}{1+\tau}
  \le\frac{C'}{1+\min(a,b)}.
\]
Multiplication by \(p_{a+b}(x-y)\) proves
\eqref{eq:bridge-insertion}.

Indeed, remove the inserted vertices one at a time for \eqref{eq:ordered-insertion-cost}.  If \(z_i\) lies between
two heat kernels of lengths \(a=s_i-s_{i-1}\) and \(b=s_{i+1}-s_i\), then
\eqref{eq:bridge-insertion} replaces the three factors involving \(z_i\) by
\[
  \frac{C}{1+\min(a,b)}p_{a+b}(z_{i-1}-z_{i+1}).
\]
Integrating the corresponding time variable over
\((s_{i-1},s_{i+1})\) gives
\[
  \int_{s_{i-1}}^{s_{i+1}}
  \frac{\dd s_i}
       {1+\min(s_i-s_{i-1},s_{i+1}-s_i)}
  \le C\log(2+s_{i+1}-s_{i-1})
  \le C\log(2+L).
\]
Repeating this \(k\) times leaves only \(p_L(x-y)\), proves
\eqref{eq:ordered-insertion-cost}.
\end{proof}

\subsection{An abstract fractional-moment reduction}
\label{subsec:abstract-fractional-reduction}

We first record an abstract fractional-moment reduction in the notation of
Section~\ref{subsec:coarse-graining}.  Recall the boxes \(I_y^T\) and the
coarse-path partition functions \(W_{mT}(Y)\) in
\eqref{eq:terminal-box} and \eqref{eq:coarse-path-partition}.

Fix \(T>1\).  Let \(g\) be a measurable functional of the environment
restricted to the time interval \([0,T]\), satisfying
\begin{equation*}
  0<g\le1,
  \qquad
  C_g:=\E[g^{-1}]<\infty.
\end{equation*}

For a fixed path segment \(B=(B_s)_{0\le s\le T}\), define the planted
environment measure \(\Pp_B^\beta\) by
\begin{equation}
  \frac{\dd\Pp_B^\beta}{\dd\Pp}
  :=
  \exp\left\{
    \beta H_T(B)-\frac{\beta^2}{2}T
  \right\},
  \qquad
  \E_B^\beta[\cdot]:=\E_{\Pp_B^\beta}[\cdot].
  \label{eq:abstract-planted-measure}
\end{equation}
The right-hand side of \eqref{eq:abstract-planted-measure} has expectation
one because \(H_T(B)\) is centered Gaussian with variance \(T\).  Set
\[
  \Phi_g(B):=\E_B^\beta[g]
\]
and, for \(z\in\Z^2\),
\begin{equation}
  a_{T,g}(z)
  :=
  \sup_{x\in I_0^T}
  P_x\left[
    \1_{\{B_T\in I_z^T\}}\Phi_g(B)
  \right].
  \label{eq:abstract-one-block-coefficient}
\end{equation}
Finally, define
\begin{equation*}
  \rho_{T,g}
  :=
  C_g^{1/2}
  \sum_{z\in\Z^2}a_{T,g}(z)^{1/2}.
\end{equation*}

For \(m\in\N\) and
\[
  Y=(y_1,\ldots,y_m)\in(\Z^2)^m,
  \qquad y_0:=0,
\]
put
\[
  z_k:=y_k-y_{k-1},
  \qquad 1\le k\le m.
\]
For \(\varphi\in C_c^\infty((0,T)\times\R^2)\), define the environment
in the \(k\)-th block, expressed in coordinates, by
\[
  \omega_{k,Y}(\varphi)
  :=
  \omega(\varphi_{k,Y}),\qquad
  \varphi_{k,Y}(s,x)
  :=
  \varphi\left(
    s-(k-1)T,\,
    x-\sqrt T\,y_{k-1}
  \right).
\]
The corresponding translated penalty and global penalty are
\[
  g_{k,Y}:=g(\omega_{k,Y}),
  \qquad
  G_Y:=\prod_{k=1}^m g_{k,Y}.
\]
For the Brownian path, define the \(k\)-th segment \(B^{k,Y}\) to be
the translation of the original segment \(B_{[(k-1)T,kT]}\) by
\(-\sqrt T\,y_{k-1}\):
\[
  B_s^{k,Y}
  :=
  B_{(k-1)T+s}-\sqrt T\,y_{k-1},
  \qquad 0\le s\le T.
\]

\begin{lemma}
\label{lem:integrable-fractional-reduction}
If
\begin{equation}
  \rho_{T,g}<1,
  \label{eq:abstract-contraction-condition}
\end{equation}
then, for every \(m\ge1\),
\begin{equation}
  \E\left[W_{mT}(\beta)^{1/2}\right]
  \le \rho_{T,g}^{\,m}.
  \label{eq:abstract-fractional-decay}
\end{equation}
Consequently,
\begin{equation}
  p(\beta)
  \le
  \frac{2\log\rho_{T,g}}{T}
  <0.
  \label{eq:abstract-free-energy-upper}
\end{equation}
\end{lemma}

\begin{proof}
The fields \(\omega_{k,Y}\), \(1\le k\le m\), depend on disjoint time
intervals and have the same law as the environment on \([0,T]\).
Therefore,
\[
  \E[G_Y^{-1}]
  =
  \prod_{k=1}^m\E[g_{k,Y}^{-1}]
  =
  C_g^m.
\]
Hence Cauchy--Schwarz inequality gives
\begin{equation}\label{fract12moment}
  \E\left[W_{mT}(Y)^{1/2}\right]
  =
  \E\left[
    G_Y^{-1/2}\bigl(G_YW_{mT}(Y)\bigr)^{1/2}
  \right]
  \le
  C_g^{m/2}
  \bigl(\E[G_YW_{mT}(Y)]\bigr)^{1/2}.
\end{equation}

For a fixed Brownian path, independence of the environment on disjoint
time intervals and \eqref{eq:abstract-planted-measure} give
\begin{equation}
  \E[G_YW_{mT}(Y)]
  =
  P\left[
    \prod_{k=1}^m
    \1_{\{B_{kT}\in I_{y_k}^T\}}
    \Phi_g(B^{k,Y})
  \right].
  \label{eq:abstract-block-environment-average}
\end{equation}
Let \(\mathcal B_t:=\sigma(B_s:0\le s\le t)\). The Markov property and
spatial translation invariance yield
\[
  P\left[
    \1_{\{B_{kT}\in I_{y_k}^T\}}
    \Phi_g(B^{k,Y})
    \,\middle|\,
    \mathcal B_{(k-1)T}
  \right]
  \le
  \sup_{x\in I_0^T}
  P_x\left[
    \1_{\{B_T\in I_{z_k}^T\}}\Phi_g(B)
  \right]
  =
  a_{T,g}(z_k).
\]
Iterating this inequality backward from \(k=m\) to
\(k=1\) in \eqref{eq:abstract-block-environment-average} gives
\begin{equation}
  \E[G_YW_{mT}(Y)]
  \le
  \prod_{k=1}^m a_{T,g}(z_k).
  \label{eq:abstract-global-Markov}
\end{equation}

Using the above estimates \eqref{eq:abstract-global-Markov}, \eqref{fract12moment} and \(\sqrt{\sum a} \le \sum \sqrt{a}\), we obtain
\[
  \E[W_{mT}(\beta)^{1/2}]
  \le
  C_g^{m/2}
  \sum_{Y\in(\Z^2)^m}
  \prod_{k=1}^m a_{T,g}(y_k-y_{k-1})^{1/2}
  =
  \left[
    C_g^{1/2}
    \sum_{z\in\Z^2}a_{T,g}(z)^{1/2}
  \right]^m
  =
  \rho_{T,g}^{\,m}.
\]
This proves \eqref{eq:abstract-fractional-decay}.  Finally,
\[
  \frac1{mT}\E\log W_{mT}(\beta)
  =
  \frac2{mT}\E\log W_{mT}(\beta)^{1/2}
  \le
  \frac2{mT}
  \log\E[W_{mT}(\beta)^{1/2}]
  \le
  \frac{2\log\rho_{T,g}}{T}.
\]
Letting \(m\to\infty\) proves
\eqref{eq:abstract-free-energy-upper}.
\end{proof}

\subsection{Some motivations}
\label{subsec:integrable-penalty}

By Lemma~\ref{lem:integrable-fractional-reduction}, one needs to construct a penalty \(g\), at a
block length \(T=T_\beta\), for which
\begin{equation}
  \rho_{T,g}
  :=
  \E[g^{-1}]^{1/2}
  \sum_{z\in\Z^2}a_{T,g}(z)^{1/2}
  <1,
  \label{eq:target-rho}
\end{equation}
where \(a_{T,g}\) is defined in
\eqref{eq:abstract-one-block-coefficient}. 
The two factors in \eqref{eq:target-rho} impose opposite requirements on
\(g\).  Its cost \(\E[g^{-1}]\), computed under the original environment
law, must remain bounded, whereas its expectation under the planted law
\(\Pp_B^\beta\) defined in \eqref{eq:abstract-planted-measure} must be
small.

For this purpose, let \(\mathcal A\) be an environment
event and, for \(K>1\), set
\[
  g_{K,\mathcal A}
  :=
  e^{-K}\1_{\mathcal A}+\1_{\mathcal A^c}
  =
  \exp\{-K\1_{\mathcal A}\}.
\]
Then
\begin{align*}
  \E[g_{K,\mathcal A}^{-1}]
  &=
  1+(e^K-1)\Pp(\mathcal A),\\
  \E_B^\beta[g_{K,\mathcal A}]
  &=
  e^{-K}\Pp_B^\beta(\mathcal A)
  +\Pp_B^\beta(\mathcal A^c)
  \le
  e^{-K}+\Pp_B^\beta(\mathcal A^c).
\end{align*}
Thus it is enough to find an event which is rare under \(\Pp\) but has
probability close to one under \(\Pp_B^\beta\).

We construct such an event from an environment statistic \(X\).  Suppose
that
\[
  \E[X]=0,
  \qquad
  \Var(X)=1,
\]
while, under the planted law,
\[
  \mu_B:=\E_B^\beta[X]\longrightarrow\infty,
  \qquad
  v_B:=\Var_B^\beta(X)\le C
\]
For
\[
  \mathcal A_X:=\{X>e^{K^2}\},
\]
The Chebyshev's inequality under the original environment law gives
\[
  \Pp(\mathcal A_X)
  \le
  e^{-2K^2}.
\]
On the other hand, whenever \(\mu_B>2e^{K^2}\),
\[
  \Pp_B^\beta(\mathcal A_X^c)
  =
  \Pp_B^\beta(X\le e^{K^2})
  \le
  \frac{v_B}{(\mu_B-e^{K^2})^2}
  \le
  4\frac{C}{\mu_B^2}.
\]
Consequently, the penalty
\begin{equation}
  \mathfrak g_K(X)
  :=
  \exp\left\{-K\1_{\{X>e^{K^2}\}}\right\}
  \label{eq:abstract-threshold-penalty}
\end{equation}
satisfies
\[
  \E[\mathfrak g_K(X)^{-1}]
  \le
  1+(e^K-1)e^{-2K^2},
  \qquad
  \E_B^\beta[\mathfrak g_K(X)]
  \le
  e^{-K}+4\frac{C}{\mu_B^2}.
\]
The remaining task is therefore to construct a statistic \(X=X_{q,T}\)
whose variance is normalized under \(\Pp\), whose planted mean diverges,
and whose planted variance remains bounded.  We obtain such a statistic
from an ordered Wiener chaos.

\subsection{Construction of the detector and one-block contraction}
\label{subsec:detector-construction}

Fix \(a\in(0,1)\).  Let \(A>0\) be chosen later (see \eqref{eq:specialized-planted-mean}) and, for small \(\beta\),
set
\begin{equation}
  T=T_\beta:=\exp\left\{\frac{A}{\beta^2}\right\},
  \qquad
  u:=T^a,
  \qquad
  q:=\left\lceil(\log\log T)^2\right\rceil.
  \label{eq:integrable-scales}
\end{equation}
Then
\begin{equation*}
  \beta^2\log u=aA,
  \qquad
  \frac{qu}{T}\longrightarrow0,
  \qquad
  q\log(q+1)=o(\log T).
\end{equation*}
Here \(T\) is the coarse-graining scale, \(u\) is the maximal time
separation between successive vertices of the detector, and \(q+1\) is its chaos order.

We next introduce the lower cutoff on the time gaps.  Choose \(L\ge1\)
and set
\begin{equation}
  D_L:=\{x\in\R^2:|x|\le L\},
  \qquad
 \underline{r}_q := 8L^2 q.
  \label{eq:local-mass-and-cutoff}
\end{equation}

Recall that \(I_0^T=\sqrt T[0,1)^2\).  For \(R>4\), let
\begin{equation*}
  \mathcal C_T
  :=
  \{x:\operatorname{dist}_\infty(x,I_0^T)\le R\sqrt T\}
  =
  \sqrt T[-R,R+1]^2.
\end{equation*}
The ordered time region of the detector is
\begin{equation}
  \Delta_{q,T}
  :=
  \left\{
    (t_0,\ldots,t_q):
    0<t_0<T/4,\quad
    \underline{r}_q<t_j-t_{j-1}<u,\quad 1\le j\le q
  \right\}.
  \label{Delta_q,T}
\end{equation}
For small \(\beta\), one has \(\underline{r}_q<u\) and \(qu<T/4\); hence
\(t_q<T/2\) for every \(\mathbf t\in\Delta_{q,T}\).

We now use the isonormal realization introduced in
\eqref{eq:e-x-inner-product}--\eqref{eq:H-as-isonormal}.  Thus
\(\mathfrak H_Q\) is generated by \((e_x)_{x\in\R^2}\) with
\[
  \langle e_x,e_y\rangle_{\mathfrak H_Q}=Q(x-y),
\]
and \(W\) denotes the isonormal Gaussian process over
\(L^2(\R_+;\mathfrak H_Q)\).

For \(\mathbf t=(t_0,\ldots,t_q)\in\Delta_{q,T}\), put
\(r_j=t_j-t_{j-1}\) and define
\begin{equation}
  F_{\mathbf t}
  :=
  \int_{\mathcal C_T\times(\R^2)^q}
  \left(
    \prod_{j=1}^q p_{r_j}(x_j-x_{j-1})
  \right)
  e_{x_0}\otimes\cdots\otimes e_{x_q}
  \dd x_0\cdots\dd x_q.
  \label{definition of F_t}
\end{equation}
Thus \(F_{\mathbf t}\in\mathfrak H_Q^{\otimes(q+1)}\).  The product of
heat kernels assigns to each spatial chain
\((x_0,\ldots,x_q)\) a Brownian transition weight.  Only \(x_0\) is
restricted to \(\mathcal C_T\); the remaining vertices are localized
through the heat kernels.

Let \(\mathcal I_{q+1}(F)\) denote the iterated Wiener integral with
respect to \(W\), over the ordered region \(\Delta_{q,T}\).  Define
\begin{equation}
  \sigma_{q,T}^2
  :=
  \int_{\Delta_{q,T}}
  \norm{F_{\mathbf t}}_{\mathfrak H_Q^{\otimes(q+1)}}^2
  \dd\mathbf t,
  \qquad
  X_{q,T}
  :=
  \frac{\mathcal I_{q+1}(F)}{\sigma_{q,T}}.
  \label{eq:detector-definition}
\end{equation}
By It\^o isometry,
\[
  \E[X_{q,T}]=0,
  \qquad
  \Var(X_{q,T})=1.
\]

For a fixed path \(B\), set
\begin{equation*}
  \mu_{q,T}(B)
  :=
  \E[X_{q,T}(W+\beta h_B)],
  \qquad
  v_{q,T}(B)
  :=
  \Var(X_{q,T}(W+\beta h_B)).
\end{equation*}
Here \(h_B(s)=\1_{[0,T]}(s)e_{B_s}\), as in
\eqref{eq:H-as-isonormal}, and \(W+\beta h_B\) denotes the
Cameron--Martin shift
\[
  W(f)\longmapsto
  W(f)+\beta\langle h_B,f\rangle_{L^2(\R_+;\mathfrak H_Q)}.
\]

The following three estimates, \eqref{eq:sigma-qT-bounds}--\eqref{eq:shifted-chaos-variance}, will be used to verify \eqref{eq:target-rho}. Their proofs are postponed to Section~\ref{proof of Lemma 5.3}.

\begin{lemma}
\label{lem:detector-estimates}
There are constants \(0<c_0<C_0<\infty\), independent of \(q,T\), such
that
\begin{equation}
  c_0^{q+1}T^2(\log u)^q
  \le\sigma_{q,T}^2\le
  C_0^{q+1}T^2(\log u)^q.
  \label{eq:sigma-qT-bounds}
\end{equation}

For every \(\delta>0\), \(R\) may be chosen so that, uniformly in
\(x\in I_0^T\),
\begin{equation}
  P_x\left(
    \mu_{q,T}(B)
    \ge
    \beta\big(c_1\beta^2\log u\big)^{q/2}
  \right)
  \ge1-\delta-o_\beta(1),
  \label{eq:planted-mean-high-probability}
\end{equation}
where \(c_1>0\) is independent of \(q,T\).

Finally, with \(L_T:=\log(2+qu)\),
\begin{equation}
  \sup_B v_{q,T}(B)
  \le
  1+
  e^{Cq\log(q+1)}\frac{qu}{T}
  \sum_{r=1}^q\beta^{2r}L_T^{r-1}.
  \label{eq:shifted-chaos-variance}
\end{equation}
For the scales in \eqref{eq:integrable-scales}, the second term on the
right-hand side tends to zero.
\end{lemma}
\begin{remark}
    It gives a successful example \(X_{q,T}\) that satisfies the framework in Section~\ref{subsec:integrable-penalty}. The above three estimates describe, respectively, the normalization of the detector, the amplification of its mean and the stability of its fluctuations under the planted law \(\Pp_B^\beta\). The main technical point is \eqref{eq:shifted-chaos-variance}. When \(Q\) is integrable, the paired heat-kernel chains can be contracted after integrating out the retained anchor \( Q(z-B_{\widehat t})\) in the whole space \(R^2\), as \eqref{eq:integrable-anchor-chain}. When \(Q\) is non-integrable, the same strategy requires retaining the anchor together with the \(\1_{\{x_0\in\mathcal C_T\}}\1_{\{x'_0\in\mathcal C_T\}}\), which produces an extra cost of \(\Lambda(T)^2\); this is the additional role of Lemma~\ref{lem:critical-detector-estimates}.
    In both regimes, the estimates for the shifted variances \(v_{q,T}(B)\) and \(v^{cr}_{q,T}(B)\) are stronger than what is required by \cite{BergerLacoin2017}; this shows that the variances under the planted law \(\Pp_B^\beta\) converge uniformly to \(1\).
\end{remark}

Now fix \(K>1\) and set
\begin{equation}
  g_\beta:=\mathfrak g_K(X_{q,T}).
  \label{eq:chaos-penalty}
\end{equation}
See the definition of $\mathfrak g_K$ in \eqref{eq:abstract-threshold-penalty}. Because \(X_{q,T}\) is centered with variance one, for sufficiently large \(K\ge1\)
\begin{equation*}
   C_K:=\E[g_\beta^{-1}]
  \le
  1+(e^K-1)e^{-2K^2}\leq2
\end{equation*}

For every fixed path \(B\), the Gaussian shift identity gives
\begin{equation}
  \E_B^\beta[G(W)]=\E\left[
    e^{\beta H_T(B)-\beta^2T/2}G(W)
  \right]
  =
  \E[G(W+\beta h_B)].
  \label{eq:path-size-bias}
\end{equation}

For \(z\in\Z^2\), define
\begin{equation*}
  a_\beta(z)
  :=
  \sup_{x\in I_0^T}
  P_x\left[
    \1_{\{B_T\in I_z^T\}}
    \E\big[g_\beta(W+\beta h_B)\big]
  \right]
\end{equation*}
and
\begin{equation*}
  \rho_\beta
  :=C_K^{1/2}\sum_{z\in\Z^2}a_\beta(z)^{1/2}.
\end{equation*}

\begin{proposition}[One-block contraction]
  \label{eq:integrable-one-block-rho}
The constants \(K,R,A\) can be chosen, independently of \(\beta\), so
that there exist \(\rho<1\) and \(\beta_0>0\) satisfying
\begin{equation*}
  \rho_\beta\le\rho,
  \qquad 0<\beta\le\beta_0.
\end{equation*}
\end{proposition}

\begin{proof}
We split the quantity to be estimated as
\begin{align*}
  \rho_\beta
  &=
  C_K^{1/2}
  \sum_{|z|_\infty\le M}a_\beta(z)^{1/2}
  +
  C_K^{1/2}
  \sum_{|z|_\infty>M}a_\beta(z)^{1/2}.
\end{align*}

Since \(0<g_\beta\le1\),
\[
  a_\beta(z)
  \le
  \sup_{x\in I_0^T}P_x(B_T\in I_z^T).
\]
Moreover, \(C_K\le2\).  Brownian scaling and Gaussian tails allow us to choose \(M\) so large that, uniformly in \(T\ge1\),
\begin{equation}
  C_K^{1/2}
  \sum_{|z|_\infty>M}a_\beta(z)^{1/2}
  \le
  \sqrt2
  \sum_{|z|_\infty>M}
  \left(
    \sup_{x\in I_0^T}P_x(B_T\in I_z^T)
  \right)^{1/2}<\frac14.
  \label{eq:far-box-contribution}
\end{equation}

It remains to control the finitely many terms with
\(|z|_\infty\le M\).  For such \(z\), we discard the endpoint constraint
in the definition of \(a_\beta(z)\) and obtain
\begin{equation*}
  a_\beta(z)
  \le
  \sup_{x\in I_0^T}
  P_x\left[\E[g_\beta(W+\beta h_B)]\right].
\end{equation*}
Thus it suffices to show the planted expectation on the right-hand side
uniformly small.

For
\(\delta>0\).  Choose \(R\) in
\eqref{eq:planted-mean-high-probability} for this value of \(\delta\),
and then choose \(A\) so large that  \(c_1aA>4\).  Since
\(\beta^2\log u=aA\), that estimate gives
\begin{equation}
  \inf_{x\in I_0^T}
  P_x\left(
    \mu_{q,T}(B)
    \ge
    \beta(c_1aA)^{q/2}
  \right)
  \ge1-\delta-o_\beta(1).
  \label{eq:specialized-planted-mean}
\end{equation}

We next substitute the chosen scales into
\eqref{eq:shifted-chaos-variance}.  Since
\(L_T=\log(2+qu)\le C_A\beta^{-2}\),
\begin{equation*}
  \sum_{r=1}^q\beta^{2r}L_T^{r-1}
  \le
  \beta^2q\max\{1,C_A^q\},
\end{equation*}
and hence
\begin{align*}
  e^{Cq\log(q+1)}
  \frac{qu}{T}
  \sum_{r=1}^q\beta^{2r}L_T^{r-1}\le
  \beta^2q^2T^{a-1}
  e^{Cq\log(q+1)}
  \max\{1,C_A^q\}
  \longrightarrow0.
\end{align*}
Consequently, for all sufficiently small $\beta$,
\[
\sup_B v_{q,T}(B) \le 2.
\]

If
\(
\mu_{q,T}(B)\ge\beta(c_1aA)^{q/2},
\)
The Chebyshev's inequality gives
\begin{equation}
  \Pp\left(
    X_{q,T}(W+\beta h_B)\le e^{K^2}
  \right)
  \le
  \frac{2}
  {\left(\beta(c_1aA)^{q/2}-e^{K^2}\right)^2}
  =
  o_\beta(1).
  \label{eq:detector-lower-tail}
\end{equation}
Combining \eqref{eq:abstract-threshold-penalty}, \eqref{eq:integrable-scales}, \eqref{eq:chaos-penalty}, \eqref{eq:specialized-planted-mean}, and \eqref{eq:detector-lower-tail}, we obtain
\begin{align*}
  &\sup_{x\in I_0^T}
  P_x\left[\E[g_\beta(W+\beta h_B)]\right]\\
  &\qquad\le
  e^{-K}
  +
  \sup_{x\in I_0^T}
  P_x\left(
    \mu_{q,T}(B)<\beta(c_1aA)^{q/2}
  \right)
  +
  \frac{2}
  {\left(\beta(c_1aA)^{q/2}-e^{K^2}\right)^2}\\
  &\qquad\le e^{-K}+\delta+o_\beta(1).
\end{align*}

We now choose \(K\) large and then \(\delta\) small enough that
\begin{equation*}
  \sqrt2\,
  \#\{z\in\mathbb Z^2:|z|_\infty\le M\}
  (e^{-K}+2\delta)^{1/2}
  <\frac14.
\end{equation*}
With the corresponding choices of \(R\) and \(A\), the preceding estimate
implies, for all sufficiently small \(\beta\),
\begin{align*}
  C_K^{1/2}
  \sum_{|z|_\infty\le M}a_\beta(z)^{1/2}
  &\le
  \sqrt2\,
  \#\{z\in\mathbb Z^2:|z|_\infty\le M\}
  (e^{-K}+2\delta)^{1/2}<\frac14.
\end{align*}
Combining this with \eqref{eq:far-box-contribution} yields
\[
  \rho_\beta<\rho=\frac12
\]
for all sufficiently small \(\beta\), which proves the proposition.
\end{proof}

The identity \eqref{eq:path-size-bias} shows that
\(a_\beta(z)=a_{T,g_\beta}(z)\) in the notation of
\eqref{eq:abstract-one-block-coefficient}. By Proposition~\ref{eq:integrable-one-block-rho} and Lemma~\ref{lem:integrable-fractional-reduction}, consequently,
\begin{equation*}
  \E[W_{mT_\beta}(\beta)^{1/2}]
  \le\rho^m,
  \qquad m\ge1,
\end{equation*}
and
\begin{equation}
  p(\beta)
  \le
  \frac{2\log\rho}{T_\beta}
  \le
  -\exp\left\{-\frac{C}{\beta^2}\right\}.
  \label{eq:integrable-free-energy-upper}
\end{equation}
This proves the required free-energy upper bound. It remains only to
justify Lemma~\ref{lem:detector-estimates}.

\subsection{Proof of Lemma~\ref{lem:detector-estimates}}
\label{proof of Lemma 5.3}

\paragraph{Proof of \eqref{eq:sigma-qT-bounds}.}
For \(y\in\R^2\), let
\[
  A_T(y):=
  |\mathcal C_T\cap(\mathcal C_T+y)|.
\]
Expanding the tensor norm in \eqref{eq:detector-definition}, using
\(\langle e_x,e_{x'}\rangle_{\mathfrak H_Q}=Q(x-x')\), and introducing
the relative coordinates \(y_j=x_j-x_j'\), we obtain
\begin{equation}
  \norm{F_{\mathbf t}}^2
  =
  \int_{(\R^2)^{q+1}}
  A_T(y_0)Q(y_0)
  \prod_{j=1}^q
  p_{2r_j}(y_j-y_{j-1})Q(y_j)
  \dd y_0\cdots\dd y_q.
  \label{eq:kernel-square-relative}
\end{equation}
Here the heat kernel \(p_{2r_j}\) results from
\begin{equation}
  \int_{\mathbb R^2}
  p_{r_j}(z)p_{r_j}(z-w)\,dz
  =
  p_{2r_j}(w),
  \label{heat-semigroup}
\end{equation}
Since \(A_T(y)\le|\mathcal C_T|\le CT\), successive applications of
\eqref{eq:uniform-pQ} give
\begin{equation}
  \norm{F_{\mathbf t}}^2
  \le
  C^{q+1}TC_Q\prod_{j=1}^qr_j^{-1}.
  \label{eq:kernel-square-upper}
\end{equation}

We now prove the reverse inequality. Recall \(D_L\) in \eqref{eq:local-mass-and-cutoff} 
and restrict every \(y_j\) in
\eqref{eq:kernel-square-relative} to \(D_L\).  Since \(D_L\) is fixed
while the diameter of \(\mathcal C_T\) is of order \(\sqrt T\),
\[
  \inf_{y\in D_L}A_T(y)\ge [(2R+1)\sqrt{T}-L]^2 \ge cT
\]
for all sufficiently large \(T\).
Moreover, for \(y,y'\in D_L\) and \(r\ge \underline{r}_q=8L^2q\), recalling \eqref{definition of F_t},
\begin{equation}
  4\pi r\,p_{2r}(y-y')=
  \exp\left\{-\frac{|y-y'|^2}{4r}\right\}\ge
  \exp\left\{-\frac{4L^2}{4\underline{r}_q}\right\}
  =e^{-1/(8q)}.
  \label{heat-kernel-density-lower-bound}
\end{equation}
Consequently,
\begin{equation}
  \norm{F_{\mathbf t}}^2
  \ge
  c^{q+1}T\prod_{j=1}^qr_j^{-1}.
  \label{eq:kernel-square-lower}
\end{equation}

Under the change of variables
\[
  (t_0,\ldots,t_q)
  \longmapsto
  (t_0,r_1,\ldots,r_q),
  \qquad r_j=t_j-t_{j-1},
\]
the region \(\Delta_{q,T}\) becomes
\[
  (0,T/4)\times(\underline{r}_q,u)^q.
\]
Its Jacobian is one, and hence
\begin{equation}
  \int_{\Delta_{q,T}}
  \prod_{j=1}^qr_j^{-1}\dd\mathbf t
  =
  \frac T4
  \left(\log\frac{u}{\underline{r}_q}\right)^q\sim\frac T4(\log u)^q.
  \label{eq:time-simplex-logarithm}
\end{equation}
Since
\[
  \log \underline{r}_q
  =
  \log(8L^2q)
  =
  O(\log q)
  =
  o(\log u),
\]
Integrating \eqref{eq:kernel-square-upper} and
\eqref{eq:kernel-square-lower} over \(\Delta_{q,T}\) proves
\eqref{eq:sigma-qT-bounds}.

\paragraph{Proof of \eqref{eq:planted-mean-high-probability}.}
For a fixed path \(B\), Applying \eqref{eq:path-size-bias} to \(X_{q,T}\) yields
\begin{equation}
  \mu_{q,T}(B)
  =
  \frac{\beta^{q+1}}{\sigma_{q,T}}\mathcal M_{q,T}(B),
  \label{mu_q,T-other-form}
\end{equation}
where
\begin{equation*}
  \mathcal M_{q,T}(B)
  :=
  \int_{\Delta_{q,T}}
  \int_{\mathcal C_T\times(\R^2)^q}
  \prod_{j=1}^q
  p_{r_j}(x_j-x_{j-1})
  \prod_{j=0}^qQ(x_j-B_{t_j})
  \dd\mathbf x\,\dd\mathbf t,
  \qquad
  r_j:=t_j-t_{j-1}.
\end{equation*}
Let
\begin{equation*}
  \mathcal G_{R,T}
  :=
  \left\{
    \sup_{0\le s\le T/2}|B_s-B_0|_\infty
    \le(R-2)\sqrt T
  \right\}.
\end{equation*}
By Brownian scaling and the maximal inequality, we have, uniformly in \(x\in I_0^T\),
\begin{equation}
\lim_{R\to\infty}\sup_{T\ge1}P_x(\mathcal G_{R,T}^c)=0
\label{restrction-corridor}
\end{equation}.

We next work on \(\mathcal G_{R,T}\). Substitute \(x_j=B_{t_j}+y_j\) and restrict
every \(y_j\) to \(D_L\). We obtain
\begin{equation}
  \mathcal M_{q,T}(B)
  \ge
  \int_0^{T/4}\Phi_s(B)\dd s,
  \label{eq:M-lower-Phi}
\end{equation}
where
\begin{align}
  \Phi_s(B)
  :=
  \int_{(\underline{r}_q,u)^q}\int_{D_L^{q+1}}
  \prod_{j=1}^q
  p_{r_j}\big(
    B_{s+r_1+\cdots+r_j}
    -B_{s+r_1+\cdots+r_{j-1}}
    +y_j-y_{j-1}
  \big)\prod_{j=0}^qQ(y_j)
  \dd\mathbf y\dd\mathbf r.
  \label{eq:Phi-s-definition}
\end{align}
For fixed \(r_j\), the Brownian increments in \eqref{eq:Phi-s-definition} are
independent and have laws \(N(0,r_jI_2)\). Therefore,
\begin{align*}
  &P_x\left[
    p_{r_j}\Big(
      B_{s+r_1+\cdots+r_j}
      -
      B_{s+r_1+\cdots+r_{j-1}}
      +y_j-y_{j-1}
    \Big)
  \right] \\
  &\qquad=
  \int_{\R^2}
  p_{r_j}(z)
  p_{r_j}(z+y_j-y_{j-1})
  \dd z
  =
  p_{2r_j}(y_j-y_{j-1}).
\end{align*}
It follows that
\begin{align*}
  P_x[\Phi_s]
  &=
  \int_{(\underline r_q,u)^q}
  \int_{D_L^{q+1}}
  \prod_{j=1}^q
  p_{2r_j}(y_j-y_{j-1})
  \prod_{j=0}^qQ(y_j)
  \dd\mathbf y\,\dd\boldsymbol r.
\end{align*}
Recalling \eqref{heat-kernel-density-lower-bound}, there exists \(c>0\), independent of \(q,T\), such that
\begin{equation*}
  P_x[\Phi_s]
  \ge
  c^{q+1}(\log u)^q
\end{equation*}
for all sufficiently small \(\beta\). Moreover, \(p_r\le(2\pi r)^{-1}\) implies
\begin{equation}
  0\le\Phi_s(B)\le
  C^{q+1}(\log u)^q.
  \label{Phi-upper-bound}
\end{equation}

Put
\[
  Z_T:=\int_0^{T/4}\Phi_s\dd s.
\]
Recalling the definition of \(\Delta_{q,T}\) in \eqref{Delta_q,T}, the variable \(\Phi_s\) depends only on Brownian increments in
\([s,s+qu]\).  Hence
\[
  \operatorname{Cov}_{P_x}(\Phi_s,\Phi_{s'})=0
  \qquad\text{whenever }|s-s'|>qu.
\]
Using \eqref{Phi-upper-bound} and integrating only over \(\{|s-s'|\le qu\}\), whose area is at most \(CTqu\), gives
\begin{align*}
  \operatorname{Var}_{P_x}(Z_T)
  &=
  \int_0^{T/4}\int_0^{T/4}
  \operatorname{Cov}_{P_x}(\Phi_s,\Phi_{s'})
  \dd s\,\dd s' \\
  &\le
  C^{q+1}Tqu(\log u)^{2q}.
\end{align*}
On the other hand,
\[
  P_x[Z_T]
  =
  \int_0^{T/4}P_x[\Phi_s]\dd s
  \ge
  c^{q+1}T(\log u)^q,
\]
Chebyshev's inequality therefore yields
\begin{align*}
  P_x\left(
    Z_T<\frac12c^{q+1}T(\log u)^q
  \right)
  &\le
  C^{q+1}\frac{qu}{T} \longrightarrow 0
  \qquad \text{as } \beta \downarrow 0.
\end{align*}
Thus, uniformly in \(x\in I_0^T\), for a constant \(c_2>0\),
\begin{equation}
  P_x\left(
    Z_T<c_2^{q+1}T(\log u)^q
  \right)
  =
  o_\beta(1).
  \label{concentration-inequality-Z_T}
\end{equation}

We can now combine \eqref{restrction-corridor} and \eqref{concentration-inequality-Z_T}.  Since
\(
\mathcal M_{q,T}(B)\ge Z_T
\)
on \(\mathcal G_{R,T}\),
\begin{equation}
  \begin{aligned}
    &P_x\left(
      \mathcal M_{q,T}(B)
      <
      c_2^{q+1}T(\log u)^q
    \right) \\
    &\qquad\le
    P_x(\mathcal G_{R,T}^{\,c})
    +
    P_x\left(
      Z_T<c_2^{q+1}T(\log u)^q
    \right)
    \le
    \delta+o_\beta(1).
  \end{aligned}
  \label{M_q,T-concentration-inequality}
\end{equation}

Finally, the upper bound in \eqref{eq:sigma-qT-bounds} gives
\[
  \sigma_{q,T}
  \le
  C_0^{(q+1)/2}T(\log u)^{q/2}.
\]
Therefore, choosing \(c_1>0\) sufficiently small, \eqref{mu_q,T-other-form} and \eqref{M_q,T-concentration-inequality} give \eqref{eq:planted-mean-high-probability}, i.e.,
\[
  \inf_{x\in I_0^T}
  P_x\left(
    \mu_{q,T}(B)
    \ge
    \beta\bigl(c_1\beta^2\log u\bigr)^{q/2}
  \right)
  \ge
  1-\delta-o_\beta(1),
\].

\paragraph{Proof of \eqref{eq:shifted-chaos-variance}.} \label{sec:proof-shifted-chaos}
Fix a path \(B\).  Under the planted environment law \(\Pp_B^\beta\), \(\E_B^\beta[X_{q,T}(W)]=\E[X_{q,T}(W+\beta h_B)]\). Since
\[
  h_B(t)=\1_{[0,T]}(t)e_{B_t},
  \qquad
  \langle e_{B_t},e_x\rangle_{\mathfrak H_Q}
  =Q(x-B_t),
\]
each Gaussian integrator in \(X_{q,T}(W+\beta h_B)\) is transformed into
\begin{equation*}
  (W+\beta h_B)(\dd t,e_x)
  =
  W(\dd t,e_x)+\beta Q(x-B_t)\dd t.
\end{equation*}

Let
\[
  J_q:=\{0,\ldots,q\}.
\]
For \(A\subseteq J_q\), define
\[
  \dd\Xi_j^A(t_j,x_j;B)
  :=
  \begin{cases}
    Q(x_j-B_{t_j})\dd t_j,
      & j\in A,\\[2mm]
    W(\dd t_j,e_{x_j}),
      & j\notin A.
  \end{cases}
\]
We then set
\begin{align*}
  \mathcal Y_A(B)
  &:=
  \int_{\Delta_{q,T}}
  \int_{\mathcal C_T\times(\R^2)^q}
  \left(
    \prod_{j=1}^q
    p_{t_j-t_{j-1}}(x_j-x_{j-1})
  \right)                                      \\
  &\hspace{25mm}\times
  \dd\Xi_0^A(t_0,x_0;B)
  \cdots
  \dd\Xi_q^A(t_q,x_q;B)
  \dd\mathbf x.
\end{align*}
Expanding \(\prod_{j=0}^q ( W(dt_j,e_{x_j}) + \beta Q(x_j-B_{t_j})\,dt_j )\) therefore gives the identity
\begin{equation}
  X_{q,T}(W+\beta h_B)
  =
  \frac1{\sigma_{q,T}}
  \sum_{A\subseteq J_q}
  \beta^{|A|}\mathcal Y_A(B).
  \label{eq:shifted-chaos-expansion}
\end{equation}

We give two special examples. Namely,
\begin{equation*}
  \mathcal Y_{\varnothing}(B)
  =
  \sigma_{q,T}X_{q,T}(W),
  \qquad
  \mathcal Y_{J_q}(B)
  =
  \mathcal M_{q,T}(B).
\end{equation*}
More generally, if \(|A|=k\), then \(\mathcal Y_A(B)\) contains
\(q+1-k\) Gaussian integrators and hence belongs to the
\((q+1-k)\)-th Wiener chaos.  Put
\begin{equation}
  Z_k(B)
  :=
  \sum_{\substack{A\subseteq J_q\\ |A|=k}}
  \mathcal Y_A(B),
  \qquad 0\le k\le q+1.
  \label{definition of Z_k}
\end{equation}
Equation \eqref{eq:shifted-chaos-expansion} becomes
\begin{equation*}
  X_{q,T}(W+\beta h_B)
  =
  \frac1{\sigma_{q,T}}
  \sum_{k=0}^{q+1}\beta^k Z_k(B).
\end{equation*}
For fixed \(B\), the variables \(Z_k(B)\) belong to mutually orthogonal
Wiener chaoses.  The term \(Z_{q+1}(B)\) is deterministic and therefore
does not contribute to the variance.  Consequently,
\begin{align}
  v_{q,T}(B)
  &=
  \frac1{\sigma_{q,T}^2}
  \sum_{k=0}^{q}
  \beta^{2k}\E\big[Z_k(B)^2\big] \notag\\
  &=
  1+
  \frac1{\sigma_{q,T}^2}
  \sum_{k=1}^{q}\beta^{2k}
  \sum_{\substack{A,A'\subseteq J_q\\|A|=|A'|=k}}
  \E\big[\mathcal Y_A(B)\mathcal Y_{A'}(B)\big].
  \label{eq:shifted-variance-contractions}
\end{align}
Here we used
\(
Z_0(B)=\sigma_{q,T}X_{q,T}(W)
\)
and
\(
\E[X_{q,T}(W)^2]=1
\).

It remains to estimate the same-order covariance terms in
\eqref{eq:shifted-variance-contractions}.  For
\(|A|=|A'|=k\), write
\begin{equation*}
  \mathcal D_{A,A'}(B)
  :=
    \E\big[\mathcal Y_A(B)\mathcal Y_{A'}(B)\big].
\end{equation*}
We shall prove that
\begin{equation}
  \mathcal D_{A,A'}(B)
  \le
  C^{q+1}T(qu)L_T^{q+k-1},
  \qquad
  L_T:=\log(2+qu),
  \label{eq:single-contraction-bound}
\end{equation}
uniformly in \(B,A,A'\).

We first write this covariance as an integral.
For \(\mathbf t=(t_0,\ldots,t_q)\in\Delta_{q,T}\), put
\[
  K_{\mathbf t}(\mathbf x)
  :=
  \1_{\{x_0\in\mathcal C_T\}}
  \prod_{j=1}^q
  p_{t_j-t_{j-1}}(x_j-x_{j-1}).
\]
Write
\[
  A^c=\{a_0,\cdots, a_{q-k}\},
  \qquad
  (A')^c=\{a'_0,\cdots,a'_{q-k}\},
\]
ordered increasingly.
For \(\mathbf s=(s_0,\ldots,s_{q-k})\), define
\[
  \mathcal T_A(\mathbf s)
  :=
  \left\{
    (t_j)_{j\in A}:
    \mathbf t\in\Delta_{q,T},
    \quad t_{a_\ell}=s_\ell,\ 0\le\ell\le q-k
  \right\},
\]
where the components of \(\mathbf t\) indexed by \(A^c\) are fixed by
\(t_{a_\ell}=s_\ell\).  Define
\(\mathcal T_{A'}(\mathbf s)\) analogously.

The It\^o isometry pairs the remaining Gaussian integrators in their natural order:
\[
  W(\dd t_{a_\ell},e_{x_{a_\ell}})
  \quad\text{with}\quad
  W(\dd t'_{a'_\ell},e_{x'_{a'_\ell}}).
\]
Their covariance is
\[
  \E\left[
    W(\dd s_\ell,e_{x_{a_\ell}})
    W(\dd s_\ell,e_{x'_{a'_\ell}})
  \right]
  =
  Q(x_{a_\ell}-x'_{a'_\ell})\dd s_\ell.
\]
Consequently,
\begin{align}
  \mathcal D_{A,A'}(B)
  &=
  \int_{0<s_0<\cdots<s_{q-k}<T}
  \int_{\mathcal T_A(\mathbf s)}
  \int_{\mathcal T_{A'}(\mathbf s)}
  \int_{(\R^2)^{2(q+1)}}
  K_{\mathbf t}(\mathbf x)
  K_{\mathbf t'}(\mathbf x')
  \notag\\
  &\quad\times
  \prod_{j\in A}Q(x_j-B_{t_j})
  \prod_{j\in A'}Q(x'_j-B_{t'_j})
  \prod_{\ell=0}^{q-k}
  Q(x_{a_\ell}-x'_{a'_\ell})
  \notag\\
  &\quad\times
  \dd\mathbf x\,\dd\mathbf x'\,
  \dd\mathbf t_A\,\dd\mathbf t'_{A'}\,\dd\mathbf s .
  \label{eq:fixed-contraction-integral}
\end{align}

We now estimate \(\mathcal D_{A,A'}(B)\) by first deleting the path
factors and then contracting the remaining spatial variables.

 Let \(j_*=\min \{A\}\) as anchor point,  and set
\(
  \widehat t:=t_{j_*}, z:=x_{j_*}.
\)
We retain the single factor
\(
  Q(z-B_{\widehat t})
\)
in the first chain, consisting of points without subscript prime,  and delete all  factors $Q(x_j-B_{t_j})$ in both chains using \eqref{eq:bridge-insertion}. The deletions are performed successively. A deleted factor at an original endpoint
is handled by the corresponding one-sided estimate.  After these
estimates, the corridor indicators are discarded using
\(\1_{\mathcal C_T}\le1\).  At most \(2k-1\) path factors are deleted.

For example, Consider first a nonendpoint path factor \(Q(x_j-B_{t_j})\) in the
first chain.  Let \(i<j<\ell\) be the nearest retained indices.  The
Gaussian bridge identity and \eqref{eq:uniform-pQ} give
\begin{align*}
  &\int_{t_i}^{t_\ell}\int_{\R^2}
  p_{t_j-t_i}(x_j-x_i)
  Q(x_j-B_{t_j})
  p_{t_\ell-t_j}(x_\ell-x_j)
  \,\dd x_j\dd t_j\\
  &\quad\le
  Cp_{t_\ell-t_i}(x_\ell-x_i)
  \int_{t_i}^{t_\ell}
  \frac{\dd t_j}
       {1+(t_j-t_i)(t_\ell-t_j)/(t_\ell-t_i)}\\
  &\quad\le
  CL_Tp_{t_\ell-t_i}(x_\ell-x_i),
  \qquad L_T:=\log(2+qu),
\end{align*}
where \(t_\ell-t_i\le qu\).  The same estimate applies to every
nonendpoint path factor in the second chain.

For a left-endpoint path factor, let \((t_\ell,x_\ell)\) be the next
retained vertex.  By \eqref{eq:uniform-pQ},
\begin{align*}
  &\int_{\substack{t_j<t_\ell\\t_\ell-t_j\le qu}}
  \int_{\R^2}
  Q(x_j-B_{t_j})
  p_{t_\ell-t_j}(x_\ell-x_j)
  \,\dd x_j\dd t_j\\
  &\quad\le
  C\int_0^{qu}\frac{\dd a}{1+a}
  \le CL_T.
\end{align*}
For a right-endpoint factor, with preceding retained vertex
\((t_i,x_i)\), the one-sided estimate is
\begin{align*}
  &\int_{\substack{t_i<t_j\\t_j-t_i\le qu}}
  \int_{\R^2}
  Q(x_j-B_{t_j})
  p_{t_j-t_i}(x_j-x_i)
  \,\dd x_j\dd t_j\\
  &\quad\le
  C\int_0^{qu}\frac{\dd a}{1+a}
  \le CL_T.
\end{align*}
Consequently, after the \(2k-1\) deletions, we have the enlarged residual
integral below (for $R_{A,A'}(B)$ see \eqref{middle-quantity-R})
\begin{equation}
  \mathcal D_{A,A'}(B)
  \le
  (CL_T)^{2k-1}\mathcal R_{A,A'}(B).
  \label{eq:integrable-path-reduction}
\end{equation}

Put \(n:=q-k\), and relabel the paired Gaussian vertices as
\[
  (s_0,x_0),\ldots,(s_n,x_n),
  \qquad
  (s_0,x'_0),\ldots,(s_n,x'_n),
\]
The symbol \(z\) retains the original anchor coordinate and is distinct
from these relabeled Gaussian variables; if \(j_*=0\), then \(z\) is the
original first-chain left endpoint.
Put
\[
  \mathscr S_{A,A'}:=
  \left\{\mathbf s:0<s_0<\cdots<s_n<T,\ 
  \mathcal T_A(\mathbf s)\ne\varnothing,\ 
  \mathcal T_{A'}(\mathbf s)\ne\varnothing\right\},
\]
let \(\kappa_*\in\{0,\ldots,n+1\}\) be determined by
\[
\kappa_*=
\begin{cases}
0,&j_*<a_0,\\
\kappa,&a_{\kappa-1}<j_*<a_\kappa,\quad 1\le\kappa\le n,\\
n+1,&j_*>a_n.
\end{cases}
\]
Since \(j_*\in A\) and \(a_\ell\in A^c\), the inequalities are strict
and \(\widehat t=t_{j_*}\) never equals a Gaussian time \(s_\ell\).
Define the single residual anchor-time fiber
\[
  G_0(\mathbf s):=(0,s_0),\qquad
  G_\kappa(\mathbf s):=(s_{\kappa-1},s_\kappa)\ (1\le\kappa\le n),
  \qquad
  G_{n+1}(\mathbf s):=(s_n,T),
\]
\[
  H_{A,j_*}(\mathbf s)
  :=
  \{t_{j_*}\in G_{\kappa_*}(\mathbf s) :(t_j)_{j\in A}\in\mathcal T_A(\mathbf s)\}.
\]
Moreover, since every original gap is
smaller than \(u\),
\[
  H_{A,j_*}(\mathbf s)\subset
  \begin{cases}
    (0,T)\cap(s_0-qu,s_0),&\kappa_*=0,\\
    (s_{\kappa_*-1},s_{\kappa_*}),&1\le\kappa_*\le n,\\
    (s_n,s_n+qu)\cap(0,T),&\kappa_*=n+1,
  \end{cases}
  \qquad
  |H_{A,j_*}(\mathbf s)|\le qu.
\]

For \(\widehat t\in H_{A,j_*}(\mathbf s)\), put
\[
  d_\ell:=s_\ell-s_{\ell-1},\qquad 1\le\ell\le n,
\]
\[
  \alpha:=
  \begin{cases}
    s_0-\widehat t,&\kappa_*=0,\\
    \widehat t-s_{\kappa_*-1},&1\le\kappa_*\le n,\\
    \widehat t-s_n,&\kappa_*=n+1.
  \end{cases}
\]
Define the retained first-chain kernels by
\begin{align}
\Pi^{(1)}_{0,\alpha}(z,\mathbf x)
&:=
p_\alpha(x_0-z)
\prod_{\ell=1}^{n}p_{d_\ell}(x_\ell-x_{\ell-1}),
\notag\\
\Pi^{(1)}_{\kappa,\alpha}(z,\mathbf x)
&:=
\left(\prod_{\ell=1}^{\kappa-1}
p_{d_\ell}(x_\ell-x_{\ell-1})\right)
 p_\alpha(z-x_{\kappa-1})
\notag\\[-1mm]
&\hspace{12mm}\times
p_{d_\kappa-\alpha}(x_\kappa-z)
\left(\prod_{\ell=\kappa+1}^{n}
p_{d_\ell}(x_\ell-x_{\ell-1})\right),
\qquad 1\le\kappa\le n,
\notag\\
\Pi^{(1)}_{n+1,\alpha}(z,\mathbf x)
&:=
\left(\prod_{\ell=1}^{n}
p_{d_\ell}(x_\ell-x_{\ell-1})\right)
p_\alpha(z-x_n).
\label{eq:integrable-first-anchor-kernels}
\end{align}
The second-chain kernel is
\begin{align}
\Pi^{(2)}_{\mathbf s}(\mathbf x')
&:=
\prod_{\ell=1}^{n}
p_{d_\ell}(x'_\ell-x'_{\ell-1}).
\label{eq:integrable-anchor-chain-kernels}
\end{align}
The indices \(0\), \(1\le\kappa\le n\), and \(n+1\) correspond,
respectively, to an anchor before \(s_0\), inside
\((s_{\kappa-1},s_\kappa)\), and after \(s_n\).  Set
\begin{align}
\mathcal K_{\mathbf s,\widehat t}(B)
:={}&
\int_{(\R^2)^{2n+3}}
Q(z-B_{\widehat t})
\Pi^{(1)}_{\kappa_*,\alpha}(z,\mathbf x)
\Pi^{(2)}_{\mathbf s}(\mathbf x')
\prod_{\ell=0}^{n}Q(x_\ell-x'_\ell)\,
\dd z\,\dd\mathbf x\,\dd\mathbf x'.
\label{eq:integrable-anchor-fiber}
\end{align}
The two corridor indicators have been removed here using
\(\1_{\mathcal C_T}\le1\); the deleted path times have already been
integrated in the preceding estimates.  Thus the enlarged residual
integral in \eqref{eq:integrable-path-reduction} is
\begin{equation}
  \mathcal R_{A,A'}(B)
  :=
  \int_{\mathscr S_{A,A'}}
  \int_{H_{A,j_*}(\mathbf s)}
  \mathcal K_{\mathbf s,\widehat t}(B)\,
  \dd\widehat t\,\dd\mathbf s .
  \label{middle-quantity-R}
\end{equation}

\begin{lemma}
\label{lem:integrable-anchor-chain}
Assume \(Q\in L^1(\R^2)\) and \eqref{eq:uniform-pQ}.  For the kernels
above, uniformly in \(B\),
\begin{align}
\mathcal K_{\mathbf s,\widehat t}(B)
&=
\norm{Q}_1
\int_{(\R^2)^{n+1}}
Q(y_0)
\prod_{\ell=1}^{n}
p_{2d_\ell}(y_\ell-y_{\ell-1})Q(y_\ell)\,
\dd\mathbf y
\notag\\
&\le
C^{n+1}\prod_{\ell=1}^{n}\frac1{1+d_\ell},
\qquad
y_\ell:=x_\ell-x'_\ell,\quad 0\le\ell\le n.
\label{eq:integrable-anchor-chain}
\end{align}
Every product over an empty index set is \(1\); in particular, the
right-hand side is \(C\) when \(n=0\).
\end{lemma}

\begin{proof}
Set \(u_\ell:=x_\ell-z\), \(v_\ell:=x'_\ell-z\), and
\(y_\ell:=u_\ell-v_\ell\). All heat kernels and all factors \(Q(x_\ell-x'_\ell)\) are invariant
under the common translation by \(z\).  Hence the only remaining
\(z\)-dependence is \(Q(z-B_{\widehat t})\), whose integral is
\(\norm{Q}_1\).

It remains to integrate the absolute coordinates \(u_\ell,v_\ell\).
Integrate the ordinary intervals from the two ends towards the interval
containing \(\widehat t\), using
\[
  \int_{\R^2}p_d(r+\delta)p_d(r)\,\dd r
  =p_{2d}(\delta).
\]
After all ordinary intervals have been contracted, the split interval,
when \(1\le\kappa_*\le n\), gives, with
\(\alpha=\widehat t-s_{\kappa_*-1}\),
\[
\begin{aligned}
&\int_{\R^2\times\R^2}
 p_a(v+y_{\kappa_*-1})
 p_{d_{\kappa_*}-\alpha}(w+y_{\kappa_*})
 p_{d_{\kappa_*}}(w-v)\,\dd v\,\dd w\\
&\hspace{35mm}
=p_{2d_{\kappa_*}}(y_{\kappa_*}-y_{\kappa_*-1}).
\end{aligned}
\]
When \(\kappa_*=0\) or \(\kappa_*=n+1\), the boundary factor integrates
to \(1\), and the same one-variable convolution is applied to all
\(n\) ordinary intervals.  This proves the equality in
\eqref{eq:integrable-anchor-chain} for all possible anchor positions.

Finally, \eqref{eq:uniform-pQ} gives
\[
  \sup_{v\in\R^2}
  \int_{\R^2}p_{2d}(y-v)Q(y)\,\dd y
  \le \frac{C}{1+d}.
\]
Applying this estimate successively to \(y_n,\ldots,y_1\), and then
integrating \(Q(y_0)\), proves the last inequality.
\end{proof}

The time constraints imply
\(0<d_\ell<qu\).  Applying
\eqref{eq:integrable-anchor-chain} on each fiber gives
\begin{align}
\mathcal R_{A,A'}(B)
&\le
C^{n+1}
\int_{\mathscr S_{A,A'}}
\bigl|H_{A,j_*}(\mathbf s)\bigr|
\prod_{\ell=1}^{n}\frac1{1+d_\ell}\,\dd\mathbf s \notag\\
&\le
C^{q+1}T(qu)
\prod_{\ell=1}^{n}
\int_0^{qu}\frac{\dd d_\ell}{1+d_\ell} \notag\\
&\le
C^{q+1}T(qu)L_T^{\,q-k}.
\label{R-upper-bound}
\end{align}
In the second line we used the change of variables
\((s_0,\ldots,s_n)\mapsto(s_0,d_1,\ldots,d_n)\), enlarged the domain. When \(n=0\), the \(d\)-integral and the product are empty and equal to \(1\).

Combining \eqref{R-upper-bound} with
\eqref{eq:integrable-path-reduction}, we obtain
\begin{align}
  \mathcal D_{A,A'}(B)
  &\le
  (CL_T)^{2k-1}
  C^{q+1}T(qu)L_T^{q-k} \notag\\
  &\le
  C^{q+1}T(qu)L_T^{q+k-1}.
  \label{eq:fixed-contraction-bound}
\end{align}

We now sum \eqref{eq:fixed-contraction-bound} over the subsets appearing
in \(Z_k(B)\).  By \eqref{definition of Z_k},
\begin{align*}
  \E[Z_k(B)^2]
  &=
  \sum_{\substack{A,A'\subseteq J_q\\ |A|=|A'|=k}}
  \mathcal D_{A,A'}(B)\\
  &\le
  \binom{q+1}{k}^2
  C^{q+1}T(qu)L_T^{q+k-1}.
\end{align*}
This estimate is uniform in \(B\).  Since
\(\binom{q+1}{k}^2\le4^{q+1}\) and \(L_T\le C\log u\), the lower bound in
\eqref{eq:sigma-qT-bounds} yields
\begin{align}
  v_{q,T}(B)
  &=\Var\bigl(X_{q,T}(W+\beta h_B)\bigr) \notag\\
  &\le
  1+
  e^{Cq\log(q+1)}
  \frac{qu}{T}
  \sum_{k=1}^{q}\beta^{2k}L_T^{k-1}.
  \label{shifted-chaos-variance-upper-bound}
\end{align}
This proves \eqref{eq:shifted-chaos-variance}.  Under the scale choices
in \eqref{eq:integrable-scales}, the second term tends to zero as
\(\beta\downarrow0\), and hence \(v_{q,T}(B)=1+o_\beta(1)\).

\subsection{The matching lower bound}
\label{subsec:integrable-matching-lower}

It remains to prove that the lower bound. By the replica comparison
\eqref{eq:replica-pinning-comparison}, it is enough to control the
two-dimensional pinning free energy for a bounded integrable potential.

\begin{lemma}
\label{lem:two-d-integrable-pinning}
Let \(V:\R^2\to[0,\infty)\) be bounded and integrable. There are
\(c,C,h_0>0\) such that
\begin{equation}
  F_V(h)\le C\exp\left\{-\frac{c}{h}\right\},
  \qquad 0<h\le h_0.
  \label{eq:two-d-pinning-upper}
\end{equation}
\end{lemma}

\begin{proof}
Set
\[
  \eta(T)
  :=
  \sup_{x\in\R^2}E_x\left[\int_0^T V(B_s)\dd s\right]
  =
  \sup_{x\in\R^2}\int_0^T(p_s*V)(x)\dd s.
\]
For \(0<s\le1\),
\[
  \sup_x(p_s*V)(x)\le\norm V_\infty,
\]
whereas, for \(s\ge1\),
\[
  \sup_x(p_s*V)(x)
  \le\norm{p_s}_\infty\norm V_1
  =\frac{\norm V_1}{2\pi s}.
\]
Consequently,
\begin{equation}
  \eta(T)
  \le
  \norm V_\infty+\frac{\norm V_1}{2\pi}\log T
  =:C_0+C_1\log T,
  \qquad T\ge1.
  \label{eq:two-d-khasminskii-log}
\end{equation}

Choose \(b>0\) so small that \(C_1b\le1/4\), and put
\[
  T_h:=\exp\{b/h\}.
\]
If \(h\le h_0\), with \(h_0C_0\le1/4\), then
\(h\eta(T_h)\le1/2\). 
Lemma~\ref{lem:upper}, gives
\begin{equation}
  \sup_{x\in\R^2}E_x\left[
    \exp\left\{h\int_0^{T_h}V(B_s)\dd s\right\}
  \right]
  \le2.
  \label{eq:two-d-one-block}
\end{equation}
For any $t = nT_h + r$ with $0 \le r < T_h$, partition the interval $[0,t]$ into blocks of length $T_h$. On each block, apply the uniform bound in \eqref{eq:two-d-one-block}; the final block of length $r$ is controlled by the same bound. By the Markov property at the block endpoints, iterating gives
\[
E_x\left[\exp\left\{h\int_0^t V(B_s)\,ds\right\}\right] \le 2^{n+1}.
\]
Taking logarithms and dividing by $t$,
\[
\frac1t \log E_x\left[\exp\left\{h\int_0^t V(B_s)\,ds\right\}\right]
\le \frac{(n+1)\log 2}{nT_h + r}.
\]
Let $t \to \infty$, therefore,
\[
F_V(h) = \limsup_{t\to\infty} \frac1t \log E_x\left[\exp\left\{h\int_0^t V(B_s)\,ds\right\}\right]
\le \frac{\log 2}{T_h}
= (\log 2)\exp\left\{-\frac{b}{h}\right\},
\]
which proves the lemma.
\end{proof}

Apply Lemma~\ref{lem:two-d-integrable-pinning} to
\[
  V(x)=Q(\sqrt2x).
\]
The replica comparison \eqref{eq:replica-pinning-comparison} gives
\begin{equation}
  p(\beta)
  \ge
  -C F_{Q(\sqrt2\,\cdot)}(2\beta^2)
  \ge
  -C\exp\left\{-\frac{c}{\beta^2}\right\}.
  \label{eq:integrable-free-energy-lower}
\end{equation}
Together with \eqref{eq:integrable-free-energy-upper}, this proves
\eqref{eq:two-d-integrable-free-energy}. This completes the proof of Theorem~\ref{thm:radial-regimes}(ii).

\section{The two-dimensional critical spatial correlation}
\label{sec:two-dimensional-critical}

In this section we prove Theorem~\ref{thm:radial-regimes}(iii). Thus
\(d=2\), and the radial tail assumption \eqref{eq:radial-tail-class} is with \(\vartheta=2\) throughout. We retain the heat kernel \(p_t\)
from \eqref{eq:two-d-heat-kernel} and put
\[
  \Lambda(t):=1+\log(2+t),
  \qquad t\ge0.
\]
The abstract fractional-moment reduction of
Subsection~\ref{subsec:abstract-fractional-reduction} and the penalty \eqref{eq:abstract-threshold-penalty} will be used. We follow the same order of
argument: construct the detector, prove the one-block contraction, and only
then verify the estimates on the detector. Unlike
Section~\ref{sec:two-dimensional-integrable}, we have
$Q\notin L^1(\R^2)$. This causes many difficulties in the estimation of the shifted variance \eqref{eq:critical-shifted-variance}, but the strategy is analogous: first delete the path factor, then integrate over the spatial variables at non-initial points, and finally extend the integration to a larger time region.

\subsection{Some estimates used in Section~\ref{subsec:critical-detector-estimates-proof}}
We first record the convolution bounds needed below. The tail assumption
\eqref{eq:radial-tail-class} with \(\vartheta=2\) implies
\begin{equation}
  \sup_{z\in\R^2}(p_t*Q)(z)
  \le C\frac{\Lambda(t)}{1+t},
  \qquad t>0.
  \label{eq:critical-pQ}
\end{equation}
Indeed, with \(w(x):=(1+|x|)^{-2}\), the upper tail bound and the fact
that the convolution of two radial decreasing functions is maximized
at the origin give, for \(t\ge2\),
\[
  \sup_z(p_t*Q)(z)
  \le C(p_t*w)(0)
  =\frac Ct\int_0^\infty
    e^{-r^2/(2t)}\frac{r}{(1+r)^2}\dd r
  \le C\frac{\log t}{t}.
\]
For \(t<2\), use \(0\le Q\le Q(0)=1\). We next estimate the convolution of the two \(Q\)-factors. Set
\(R:=|z|\). Since \(Q\in L^2(\R^2)\), the desired bound is immediate
when \(R\le2\). Suppose therefore that \(R>2\), and decompose the
convolution into
\[
  D_1:=\{|x|\le R/2\},
  \qquad
  D_2:=\{|z-x|\le R/2\},
\]
and their complement \(D_3\). On \(D_1\), we have
\(|z-x|\ge R/2\), and hence
\begin{align*}
  \int_{D_1}Q(x)Q(z-x)\dd x
  \le
  \frac{C}{(1+R)^2}
  \int_0^{R/2}\frac{r}{(1+r)^2}\dd r\le
  C\frac{\Lambda(R)}{(1+R)^2}.
\end{align*}
The contribution of \(D_2\) satisfies the same bound. On \(D_3\), split
further according to \(|x|\le2R\) or \(|x|>2R\). Using
\(|x|,|z-x|\ge R/2\) in the first region and
\(|z-x|\ge |x|/2\) in the second, we obtain
\[
  \int_{D_3}Q(x)Q(z-x)\dd x
  \le
  \frac{C}{R^4}R^2
  +C\int_{2R}^{\infty}\frac{\dd r}{r^3}
  \le\frac{C}{R^2}.
\]
Consequently,
\[
  (Q*Q)(z)
  \le
  C\frac{\Lambda(|z|)}{(1+|z|)^2}.
\]

It follows from polar coordinates that, for \(t\ge4\),
\begin{align*}
  \int_{\R^2}p_{2t}(z)(Q*Q)(z)\dd z
  &\le
  \frac Ct\int_0^\infty
  e^{-r^2/(4t)}
  \frac{\Lambda(r)r}{(1+r)^2}\dd r\\
  &\le
  \frac Ct\left(
    1+\int_2^{\sqrt t}\frac{\log r}{r}\dd r
    +\int_{\sqrt t}^{\infty}
      e^{-r^2/(4t)}\frac{\Lambda(r)}r\dd r
  \right)\\
  &\le
  C\frac{\Lambda(t)^2}{t}.
\end{align*}
Indeed, the middle integral is \(O((\log t)^2)\), while the change of
variables \(r=\sqrt t\,s\) shows that the last one is
\(O(\Lambda(t))\). The remaining range \(2\le t\le4\) is absorbed by
increasing the constant. Thus
\begin{equation}
  \int_{\R^2}p_{2t}(z)(Q*Q)(z)\dd z
  \le C\frac{\Lambda(t)^2}{t},
  \qquad t\ge2.
  \label{eq:critical-pQQ}
\end{equation}

\subsection{Construction of the detector and one-block contraction}
\label{subsec:critical-detector}

We use the isonormal realization
\eqref{eq:e-x-inner-product}--\eqref{eq:H-as-isonormal}. Fix
\(a\in(0,1)\), \(R>4\), and, for large \(T\), set
\[
  u:=T^a,
  \qquad
  q:=\left\lceil(\log\log T)^2\right\rceil,
  \qquad
  \underline r_q:=8q,
  \qquad
  \mathcal C_T:=\sqrt T[-R,R+1]^2.
\]
Define
\begin{equation}
  \Delta_{q,T}^{\mathrm{cr}}
  :=
  \left\{
    (t_0,\ldots,t_q):
    0<t_0<T/4,\quad
    \underline r_q<t_j-t_{j-1}<u,\ 1\le j\le q
  \right\}.
  \label{eq:critical-time-simplex}
\end{equation}
For \(\mathbf t\in\Delta_{q,T}^{\mathrm{cr}}\), put
\(r_j=t_j-t_{j-1}\) and
\[
  F_{\mathbf t}^{\mathrm{cr}}
  :=
  \int_{\mathcal C_T\times(\R^2)^q}
  \left(\prod_{j=1}^q p_{r_j}(x_j-x_{j-1})\right)
  e_{x_0}\otimes\cdots\otimes e_{x_q}\dd\mathbf x.
\]
With \(\mathcal I_{q+1}\) denoting the ordered Wiener integral, set
\begin{equation}
  (\sigma_{q,T}^{\mathrm{cr}})^2
  :=
  \int_{\Delta_{q,T}^{\mathrm{cr}}}
  \|F_{\mathbf t}^{\mathrm{cr}}\|^2\dd\mathbf t,
  \qquad
  X_{q,T}^{\mathrm{cr}}
  :=
  \frac{\mathcal I_{q+1}(F^{\mathrm{cr}})}
       {\sigma_{q,T}^{\mathrm{cr}}}.
  \label{eq:critical-detector-definition}
\end{equation}
Then
\[
  \E[X_{q,T}^{\mathrm{cr}}]=0,
  \qquad
  \Var(X_{q,T}^{\mathrm{cr}})=1.
\]
For a fixed path \(B\), write
\[
  \mu_{q,T}^{\mathrm{cr}}(B)
  :=\E[X_{q,T}^{\mathrm{cr}}(W+\beta h_B)],
  \qquad
  v_{q,T}^{\mathrm{cr}}(B)
  :=\Var(X_{q,T}^{\mathrm{cr}}(W+\beta h_B)).
\]

The following estimates are the analogue of
Lemma~\ref{lem:detector-estimates}. They are needed for the one-block contraction. Their proof is postponed to Subsection~\ref{subsec:critical-detector-estimates-proof}.

\begin{lemma}
\label{lem:critical-detector-estimates}
There are \(0<c<C<\infty\), independent of \(q,T\), such that
\begin{equation}
  c^{q+1}T^2\Lambda(u)^{2q+1}
  \le(\sigma_{q,T}^{\mathrm{cr}})^2
  \le C^{q+1}T^2\Lambda(u)^{2q+1}.
  \label{eq:critical-sigma-bounds}
\end{equation}
For every \(\delta>0\), \(R\) may be chosen so that, uniformly in
\(x\in I_0^T\),
\begin{equation}
  P_x\left(
    \mu_{q,T}^{\mathrm{cr}}(B)
    \ge
    \beta\sqrt{\Lambda(u)}
    \big(c\beta\Lambda(u)\big)^q
  \right)
  \ge1-\delta-o_T(1).
  \label{eq:critical-planted-mean-lower}
\end{equation}
Finally,
\begin{equation}
  \sup_Bv_{q,T}^{\mathrm{cr}}(B)
  \le
  1+e^{Cq\log(q+1)}\frac{qu}{T}
  \sum_{r=1}^q\beta^{2r}\Lambda(u)^{2r+3}.
  \label{eq:critical-shifted-variance}
\end{equation}
\end{lemma}

We complete the upper bound before proving the lemma. Choose
\begin{equation}
  T=T_\beta:=\exp\left\{\frac{A}{\beta}\right\},
  \qquad u:=T^a,
  \qquad q:=\left\lceil(\log\log T)^2\right\rceil.
  \label{eq:critical-block-scale}
\end{equation}
Then \(\beta\log u=aA\). Taking \(A\) so large that
\(caA>1\), the lower bound in
\eqref{eq:critical-planted-mean-lower} satisfies
\begin{equation}
  \beta\sqrt{\Lambda(u)}
  \big(c\beta\Lambda(u)\big)^q\longrightarrow\infty.
  \label{eq:critical-mean-diverges}
\end{equation}
Furthermore,
\begin{align*}
  \sum_{r=1}^q\beta^{2r}\Lambda(u)^{2r+3}
  &\le C\beta^{-3}q\max\{1,(C\beta\Lambda(u))^{2q}\},\\
  \frac{u}{T}
  &=\exp\left\{-\frac{(1-a)A}{\beta}\right\},\\
  q\log(q+1)&=o(\beta^{-1}).
\end{align*}
Hence \eqref{eq:critical-shifted-variance} gives
\begin{equation}
  \sup_Bv_{q,T}^{\mathrm{cr}}(B)\le1+o_\beta(1).
  \label{eq:critical-variance-bounded}
\end{equation}

Fix \(K>1\) and use the penalty from
\eqref{eq:abstract-threshold-penalty}:
\[
  g_\beta:=\mathfrak g_K(X_{q,T}^{\mathrm{cr}})
  =\exp\left\{-K\1_{\{X_{q,T}^{\mathrm{cr}}>e^{K^2}\}}\right\}.
\]
Since the detector is centered with variance one,
\[
  C_K:=\E[g_\beta^{-1}]
  \le1+(e^K-1)e^{-2K^2}\le2
\]
for \(K\) sufficiently large. By \eqref{eq:path-size-bias},
\[
  a_\beta(z)
  :=a_{T,g_\beta}(z)
  =\sup_{x\in I_0^T}P_x\left[
    \1_{\{B_T\in I_z^T\}}
    \E\big[g_\beta(W+\beta h_B)\big]
  \right].
\]

\begin{proposition}[one-block contraction]
\label{prop:critical-one-block-contraction}
The constants \(K,R,A\) can be chosen independently of \(\beta\) so
that, for some \(\rho<1\),
\begin{equation}
  C_K^{1/2}\sum_{z\in\Z^2}a_\beta(z)^{1/2}
  \le\rho
  \label{eq:critical-one-block-rho}
\end{equation}
for all sufficiently small \(\beta\).
\end{proposition}

\begin{proof}
This is the same split as in the proof of
Proposition~\ref{eq:integrable-one-block-rho}. For all sufficiently large
\(K\), the penalty cost satisfies \(C_K\le2\). Since \(0<g_\beta\le1\),
Brownian scaling and Gaussian tails give
\[
  \lim_{M\to\infty}\sup_{T\ge1}
  \sum_{|z|_\infty>M}
  \left(\sup_{x\in I_0^T}P_x(B_T\in I_z^T)\right)^{1/2}=0.
\]
Choose \(M\) so large that the contribution of \(|z|_\infty>M\) to the
left-hand side of \eqref{eq:critical-one-block-rho} is less than \(1/4\),
uniformly in \(\beta\).

There are now only finitely many boxes. For \(|z|_\infty\le M\), we
discard the endpoint indicator in the definition of \(a_\beta(z)\) and
obtain
\begin{equation}
  a_\beta(z)
  \le
  \sup_{x\in I_0^T}
  P_x\left[
    \E[g_\beta(W+\beta h_B)]
  \right].
  \label{eq:critical-finite-box-reduction}
\end{equation}
We estimate the right-hand side uniformly in the entrance point.

Fix \(K>1\) and \(\delta>0\). Choose \(R\) in
\eqref{eq:critical-planted-mean-lower} for this value of \(\delta\), and
then increase \(A\), if necessary, so that
\eqref{eq:critical-mean-diverges} holds. Since \(T=T_\beta\to\infty\),
\eqref{eq:critical-planted-mean-lower} gives
\begin{equation}
  \sup_{x\in I_0^T}
  P_x\left(
    \mu_{q,T}^{\mathrm{cr}}(B)
    <
    \beta\sqrt{\Lambda(u)}
    \bigl(c\beta\Lambda(u)\bigr)^q
  \right)
  \le \delta+o_\beta(1).
  \label{eq:critical-small-mean-probability}
\end{equation}
Moreover, by \eqref{eq:critical-mean-diverges},
\[
  \beta\sqrt{\Lambda(u)}
  \bigl(c\beta\Lambda(u)\bigr)^q
  >e^{K^2}
\]
for all sufficiently small \(\beta\).

Now fix a path \(B\) such that
\[
  \mu_{q,T}^{\mathrm{cr}}(B)
  \ge
  \beta\sqrt{\Lambda(u)}
  \bigl(c\beta\Lambda(u)\bigr)^q.
\]
For sufficiently small \(\beta\), the variance bound
\eqref{eq:critical-variance-bounded} gives
\(v_{q,T}^{\mathrm{cr}}(B)\le2\). Hence Chebyshev's inequality yields
\begin{align}
  &\Pp\left(
    X_{q,T}^{\mathrm{cr}}(W+\beta h_B)\le e^{K^2}
  \right) \le
  \frac{2}{
    \left[
      \beta\sqrt{\Lambda(u)}
      (c\beta\Lambda(u))^q-e^{K^2}
    \right]^2}
  =o_\beta(1),
  \label{eq:critical-detector-lower-tail}
\end{align}
uniformly over all such paths. By the definition of \(g_\beta\), for every fixed \(B\),
\begin{align*}
  \E[g_\beta(W+\beta h_B)]
  &=
  e^{-K}\,
  \Pp\left(
    X_{q,T}^{\mathrm{cr}}(W+\beta h_B)>e^{K^2}
  \right)\\
  &\quad+
  \Pp\left(
    X_{q,T}^{\mathrm{cr}}(W+\beta h_B)\le e^{K^2}
  \right)\\
  &\le
  e^{-K}
  +
  \Pp\left(
    X_{q,T}^{\mathrm{cr}}(W+\beta h_B)\le e^{K^2}
  \right).
\end{align*}
Using
\eqref{eq:critical-small-mean-probability} and
\eqref{eq:critical-detector-lower-tail} gives
\begin{align}
  &\sup_{x\in I_0^T}
  P_x\left[
    \E[g_\beta(W+\beta h_B)]
  \right] \notag\\
  &\qquad\le
  e^{-K}
  +
  \sup_{x\in I_0^T}
  P_x\left(
    \mu_{q,T}^{\mathrm{cr}}(B)
    <
    \beta\sqrt{\Lambda(u)}
    (c\beta\Lambda(u))^q
  \right)
  +o_\beta(1) \notag\\
  &\qquad\le e^{-K}+\delta+o_\beta(1).
  \label{eq:critical-penalty-expectation}
\end{align}

We now choose \(K\) large and then \(\delta>0\) small enough that
\[
  \sqrt2\,\#\{z:|z|_\infty\le M\}
  (e^{-K}+2\delta)^{1/2}<\frac14,
\]
and retain the corresponding choices of \(R\) and \(A\) above. Since
\(C_K\le2\), equations \eqref{eq:critical-finite-box-reduction} and
\eqref{eq:critical-penalty-expectation} imply, for all sufficiently
small \(\beta\),
\begin{align*}
  C_K^{1/2}
  \sum_{|z|_\infty\le M}a_\beta(z)^{1/2}
  &\le
  \sqrt2\,\#\{z:|z|_\infty\le M\}
  \bigl(e^{-K}+\delta+o_\beta(1)\bigr)^{1/2}\\
  &\le
  \sqrt2\,\#\{z:|z|_\infty\le M\}
  (e^{-K}+2\delta)^{1/2}
  <\frac14.
\end{align*}
The \(|z|_{\infty}>M\) contribution is also less than \(1/4\). Combining the two
bounds proves \eqref{eq:critical-one-block-rho}.
\end{proof}

Hence Lemma~\ref{lem:integrable-fractional-reduction}, applied with
\(T=T_\beta\) and \(g=g_\beta\), now gives
\[
  \E[W_{mT_\beta}(\beta)^{1/2}]\le\rho^m,
  \qquad m\ge1,
\]
and hence
\begin{equation}
  p(\beta)
  \le\frac{2\log\rho}{T_\beta}
  \le-\exp\left\{-\frac{C}{\beta}\right\}.
  \label{eq:critical-free-energy-upper}
\end{equation}
This proves the upper bound in \eqref{eq:two-d-critical-free-energy}. It
remains to justify Lemma~\ref{lem:critical-detector-estimates}.

\subsection{Proof of Lemma~\ref{lem:critical-detector-estimates}}
\label{subsec:critical-detector-estimates-proof}

\paragraph{Proof of \eqref{eq:critical-sigma-bounds}.}
For \(y\in\R^2\), let
\[
  A_T(y):=|\mathcal C_T\cap(\mathcal C_T+y)|.
\]
Similar to \eqref{eq:kernel-square-relative}, we obtain
\begin{equation}
  \|F_{\mathbf t}^{\mathrm{cr}}\|^2
  =
  \int_{(\R^2)^{q+1}}
  A_T(y_0)Q(y_0)
  \prod_{j=1}^q
  p_{2r_j}(y_j-y_{j-1})Q(y_j)\dd\mathbf y.
  \label{eq:critical-kernel-square}
\end{equation}
By \eqref{eq:critical-pQ},
\begin{equation}
  \sup_x\int_{\underline r_q}^{u}\int_{\R^2}
  p_{2r}(x-y)Q(y)\,dy\,dr
  \le C\Lambda(u)^2,
  \label{onePQ}
\end{equation}
and \(r> \underline{r}_q\ge2\), \eqref{eq:critical-pQQ} gives
\begin{equation}
  \int_{\underline r_q}^{u}\iint_{(\R^2)^2}
  Q(y_0)p_{2r}(y_1-y_0)Q(y_1)
  \,dy_0\,dy_1\,dr
  \le C\Lambda(u)^3,
  \label{PtwoQ}
\end{equation}
Since \(A_T\le CT\), integration over \(t_0\in(0,T/4)\), followed by
successive application of \eqref{onePQ} in \(y_q,\ldots,y_2\) and \eqref{PtwoQ} in \(y_1,y_0\), yields
\[
  (\sigma_{q,T}^{\mathrm{cr}})^2
  \le C^{q+1}T^2
  \Lambda(u)^3\Lambda(u)^{2(q-1)}
  =C^{q+1}T^2\Lambda(u)^{2q+1}.
\]

For the reverse bound, let
\[
  A_u:=\{|y|\le u^{1/4}\}.
\]
By the lower bound in \eqref{eq:radial-tail-class} and polar coordinates,
for all sufficiently large \(u\),
\[
  \int_{A_u}Q(y)\dd y
  \ge
  c\int_1^{u^{1/4}}\frac{\dd r}{r}
  \ge c\Lambda(u).
\]
Moreover, if \(x,y\in A_u\) and \(r\in[u^{1/2},u]\), then
\[
  |x-y|^2\le4u^{1/2}\le4r.
\]
Consequently,
\[
  p_{2r}(x-y)
  =
  \frac{1}{4\pi r}
  \exp\left\{-\frac{|x-y|^2}{4r}\right\}
  \ge\frac cr,
\]
and hence
\[
  \int_{u^{1/2}}^u p_{2r}(x-y)\dd r
  \ge
  c\int_{u^{1/2}}^u\frac{\dd r}{r}
  \ge c\Lambda(u).
\]

Since \(u^{1/4}=o(\sqrt T)\), the overlap of
\(\mathcal C_T\) and \(\mathcal C_T+y\) has area at least \(cT\),
uniformly in \(y\in A_u\). Thus
\[
  \inf_{y\in A_u}A_T(y)\ge[(2R+1)\sqrt{T}-u^{\frac{1}{4}}]^2 \ge cT
\]
for all sufficiently large \(T\). Also,
\(\underline r_q\le u^{1/2}\) for large \(T\), so the restrictions
\(y_j\in A_u\) and \(r_j\in[u^{1/2},u]\) are contained in the original
integration domain. So we obtain
\begin{align}
  (\sigma_{q,T}^{\mathrm{cr}})^2
  &\ge
  cT^2
  \int_{A_u^{q+1}}
  Q(y_0)
  \prod_{j=1}^q
  \left(
    \int_{u^{1/2}}^u
    p_{2r}(y_j-y_{j-1})\,\mathrm{d}r
  \right)
  Q(y_j)\,\mathrm{d}\mathbf{y} \nonumber \\
  &\ge
  cT^2
  \bigl(c\Lambda(u)\bigr)^q
  \left(\int_{A_u}Q(y)\,\mathrm{d}y\right)^{q+1} \nonumber \\
  &\ge
  c^{q+1}T^2\Lambda(u)^{2q+1}.
  \label{2q+1 lower bound}
\end{align}
This proves \eqref{eq:critical-sigma-bounds}.

\paragraph{Proof of \eqref{eq:critical-planted-mean-lower}.}
The Gaussian shift identity \eqref{eq:path-size-bias} gives
\begin{equation}
  \mu_{q,T}^{\mathrm{cr}}(B)
  =
  \frac{\beta^{q+1}}{\sigma_{q,T}^{\mathrm{cr}}}
  \mathcal M_{q,T}(B),
  \label{eq:critical-planted-mean-definition}
\end{equation}
where
\begin{equation*}
  \mathcal M_{q,T}(B)
  :=
  \int_{\Delta_{q,T}^{\mathrm{cr}}}
  \int_{\mathcal C_T\times(\R^2)^q}
  \prod_{j=1}^q p_{r_j}(x_j-x_{j-1})
  \prod_{j=0}^qQ(x_j-B_{t_j})
  \dd\mathbf x\dd\mathbf t.
\end{equation*}
Let
\[
  \mathcal G_{R,T}
  :=
  \left\{
    \sup_{0\le s\le T/2}|B_s-B_0|_\infty
    \le(R-2)\sqrt T
  \right\}.
\]
Similar to \eqref{restrction-corridor}, we have
\begin{equation}
  \lim_{R\to\infty}
  \sup_{T\ge1}\sup_{x\in I_0^T}
  P_x(\mathcal G_{R,T}^c)=0.
  \label{eq:critical-good-path-tail}
\end{equation}

On \(\mathcal G_{R,T}\), put
\[
  s:=t_0,
  \qquad
  r_j:=t_j-t_{j-1},
  \qquad
  x_j:=B_{t_j}+y_j.
\]
Since \(s<T/4\), \(qu<T/4\), and \(y_0\in A_u\),
the restriction
\[
  r_j\in[u^{1/2},u],
  \qquad
  y_j\in A_u,
  \qquad 0\le j\le q,
\]
is contained in the original domain for all sufficiently large \(T\).
Therefore
\begin{equation}
  \mathcal M_{q,T}(B)
  \ge\int_0^{T/4}\Phi_s(B)\dd s
  \qquad\text{on }\mathcal G_{R,T},
  \label{eq:critical-M-Phi}
\end{equation}
where
\begin{align*}
  \Phi_s(B)
  &:=
  \int_{[u^{1/2},u]^q}
  \int_{A_u^{q+1}}
  \prod_{j=1}^q
  p_{r_j}\left(
    B_{s+r_1+\cdots+r_j}
    -B_{s+r_1+\cdots+r_{j-1}}
    +y_j-y_{j-1}
  \right)
  \prod_{j=0}^qQ(y_j)
  \dd\mathbf y\dd\mathbf r.
\end{align*}
For fixed gaps, the Brownian increments in the product are independent
and have laws \(N(0,r_jI_2)\). Thus
\[
  E_x\left[
    p_{r_j}(B_{t_j}-B_{t_{j-1}}+y_j-y_{j-1})
  \right]
  =
  p_{2r_j}(y_j-y_{j-1}),
\]
and the same lower-bound calculation as \eqref{2q+1 lower bound} gives
\begin{equation}
  E_x[\Phi_s(B)]
  \ge c^{q+1}\Lambda(u)^{2q+1}.
  \label{eq:critical-Phi-mean}
\end{equation}
On the other hand,
\[
  p_r\le\frac1{2\pi r},
  \qquad
  \int_{A_u}Q(y)\dd y\le C\Lambda(u),
\]
so
\begin{equation}
  0\le\Phi_s(B)
  \le C^{q+1}\Lambda(u)^{2q+1}.
  \label{eq:critical-Phi-upper}
\end{equation}

Put \(Z_T:=\int_0^{T/4}\Phi_s(B)\dd s\). The variable
\(\Phi_s\) depends only on Brownian increments in \([s,s+qu]\);
hence
\[
  \Cov_{P_x}(\Phi_s,\Phi_{s'})=0
  \qquad\text{if }|s-s'|>qu.
\]
Using \eqref{eq:critical-Phi-upper},
\begin{align*}
  \Var_{P_x}(Z_T)
  &=
  \int_0^{T/4}\int_0^{T/4}
  \Cov_{P_x}(\Phi_s,\Phi_{s'})\dd s\dd s' \\
  &\le
  C^{q+1}Tqu\Lambda(u)^{4q+2}.
\end{align*}
By \eqref{eq:critical-Phi-mean},
\[
  E_x[Z_T]\ge c^{q+1}T\Lambda(u)^{2q+1}.
\]
Consequently,
\begin{equation*}
  P_x\left(
    Z_T<c^{q+1}T\Lambda(u)^{2q+1}
  \right)
  \le C^q\frac{qu}{T}
  =o_T(1),
\end{equation*}
uniformly in \(x\in I_0^T\). Similar to the proof of \eqref{eq:planted-mean-high-probability}, combining this estimate with
\eqref{eq:critical-good-path-tail},
\eqref{eq:critical-M-Phi}, and the upper bound in
\eqref{eq:critical-sigma-bounds} proves
\eqref{eq:critical-planted-mean-lower}.

\paragraph{Proof of \eqref{eq:critical-shifted-variance}.}
The argument follows the proof of \eqref{eq:shifted-chaos-variance},
with modifications. We now prove \eqref{eq:critical-shifted-variance}. 

Fix a path \(B\).
For \(A\subseteq J_q:=\{0,\ldots,q\}\), define
\[
  \dd\Xi_j^A(t_j,x_j;B)
  :=
  \begin{cases}
    Q(x_j-B_{t_j})\dd t_j, & j\in A,\\[1mm]
    W(\dd t_j,e_{x_j}), & j\notin A,
  \end{cases}
\]
and
\begin{align*}
  \mathcal Y_A(B)
  &:={}
  \int_{\Delta_{q,T}^{\mathrm{cr}}}
  \int_{\mathcal C_T\times(\R^2)^q}
  \left(\prod_{j=1}^q
    p_{t_j-t_{j-1}}(x_j-x_{j-1})\right)
  \prod_{j=0}^q\dd\Xi_j^A(t_j,x_j;B)
  \dd\mathbf x.
\end{align*}
Expanding the shifted integrators gives
\begin{equation}
  X_{q,T}^{\mathrm{cr}}(W+\beta h_B)
  =
  \frac1{\sigma_{q,T}^{\mathrm{cr}}}
  \sum_{A\subseteq J_q}
  \beta^{|A|}\mathcal Y_A(B).
  \label{eq:critical-shifted-expansion}
\end{equation}
If \(|A|=r\), then \(\mathcal Y_A(B)\) belongs to the
\((q+1-r)\)-th Wiener chaos. Hence different values of \(r\) are
orthogonal. The term \(r=q+1\) is deterministic, while the term
\(r=0\) has variance \((\sigma_{q,T}^{\mathrm{cr}})^2\). Thus
\begin{equation}
  v_{q,T}^{\mathrm{cr}}(B)
  =1+
  \frac1{(\sigma_{q,T}^{\mathrm{cr}})^2}
  \sum_{r=1}^q\beta^{2r}
  \sum_{\substack{A,A'\subseteq J_q\\|A|=|A'|=r}}
  \E[\mathcal Y_A(B)\mathcal Y_{A'}(B)].
  \label{eq:critical-variance-expansion}
\end{equation}

Fix \(A,A'\subseteq J_q\) with \(|A|=|A'|=r\), and write
\[
  A^c=\{a_0<\cdots<a_{q-r}\},
  \qquad
  (A')^c=\{a'_0<\cdots<a'_{q-r}\}.
\]
For \(\mathbf s=(s_0,\ldots,s_{q-r})\), let
\begin{equation}
  \mathcal T_A(\mathbf s)
  :=
  \left\{
    (t_j)_{j\in A}:
    \mathbf t\in\Delta_{q,T}^{\mathrm{cr}},
    \quad t_{a_\ell}=s_\ell,
    \ 0\le\ell\le q-r
  \right\},
  \label{time fiber}
\end{equation}
and define \(\mathcal T_{A'}(\mathbf s)\) analogously. The It\^o isometry pairs
\begin{equation}
  W(\dd t_{a_\ell},e_{x_{a_\ell}})
  \quad\text{with}\quad
  W(\dd t'_{a'_\ell},e_{x'_{a'_\ell}}),
  \qquad 0\le \ell\le q-r.
  \label{eq:critical-retained-gaussian-pairing}
\end{equation}
so that \(t_{a_\ell}=t'_{a'_\ell}=s_\ell\), and contributes
\(Q(x_{a_\ell}-x'_{a'_\ell})\). Therefore
\begin{align}
  \mathcal D_{A,A'}(B)
  &:={}
  \E[\mathcal Y_A(B)\mathcal Y_{A'}(B)]
  \notag\\
  &=
  \int_{0<s_0<\cdots<s_{q-r}<T}
  \int_{\mathcal T_A(\mathbf s)}
  \int_{\mathcal T_{A'}(\mathbf s)}
  \int_{(\R^2)^{2(q+1)}}
  K_{\mathbf t}(\mathbf x)K_{\mathbf t'}(\mathbf x')
  \notag\\
  &\quad\times
  \prod_{j\in A}Q(x_j-B_{t_j})
  \prod_{j\in A'}Q(x'_j-B_{t'_j})
  \prod_{\ell=0}^{q-r}
  Q(x_{a_\ell}-x'_{a'_\ell})
  \notag\\
  &\quad\times
  \dd\mathbf x\dd\mathbf x'
  \dd\mathbf t_A\dd\mathbf t'_{A'}\dd\mathbf s,
  \label{eq:critical-fixed-contraction-integral}
\end{align}
where
\[
  K_{\mathbf t}(\mathbf x)
  :=
  \1_{\{x_0\in\mathcal C_T\}}
  \prod_{j=1}^q
  p_{t_j-t_{j-1}}(x_j-x_{j-1}).
\]

We now estimate the integral in
\eqref{eq:critical-fixed-contraction-integral}. Write
\(\varepsilon':=\1_{\{0\in A'\}}\). Since \(r\ge1\), the set \(A\)
is nonempty. Choose
\[
  j_*\in A,
  \qquad j_*=0\ \text{if }0\in A,
  \qquad
  \widehat t:=t_{j_*},
  \qquad z:=x_{j_*},
\]
and retain
\[
  Q(z-B_{\widehat t})=Q(x_{j_*}-B_{t_{j_*}}).
\]
This is the single retained path factor in the first chain;
we call it the \emph{anchor}. Every other path factor in that chain is
deleted. In the second chain we also retain its original
left-endpoint path factor \(Q(\bar x'_0-B_{t'_0})\) when
\(\varepsilon'=1\), and delete its other path factors. The two original
variables (denoted \(\bar x_0,\bar x'_0\) below) are always kept in
their original positions, together with their indicators of
\(\mathcal C_T\). Throughout
this operation, the two indicators of \(\mathcal C_T\) remain attached
to the original variables \(\bar x_0,\bar x'_0\). At most \(2r-1\)
path factors are deleted.
The left-endpoint cases are included in this convention: when
\(0\in A\), the anchor is the first-chain left-endpoint factor, and,
when \(0\in A'\), the second-chain left-endpoint factor is the optional
factor just described. A path factor at an original right endpoint may
be deleted and is handled by the one-sided estimate below.

We use the following consequence of \eqref{eq:critical-pQ}. For
\(a,b>0\), \(x,y,c\in\R^2\), and
\[
  \tau:=\frac{ab}{a+b},
\]
the Brownian bridge identity \eqref{eq:gaussian-bridge-identity} gives
\begin{equation}
  \int_{\R^2}
  p_a(z-x)Q(z-c)p_b(y-z)\dd z
  \le
  C\frac{\Lambda(\tau)}{1+\tau}
  p_{a+b}(y-x).
  \label{eq:critical-two-kernel-insertion}
\end{equation}

First consider a path factor \(Q(x_j-B_{t_j})\) strictly between two
remaining vertices in the first chain, and let
\(t_i<t_j<t_\ell\) be their times. By
\eqref{eq:critical-two-kernel-insertion},
\begin{align}
  &\int_{t_i}^{t_\ell}\int_{\R^2}
  p_{t_j-t_i}(x_j-x_i)
  Q(x_j-B_{t_j})
  p_{t_\ell-t_j}(x_\ell-x_j)
  \dd x_j\dd t_j
  \notag\\
  &\quad\le
  Cp_{t_\ell-t_i}(x_\ell-x_i)
  \int_{t_i}^{t_\ell}
  \frac{\Lambda(\tau_j)}{1+\tau_j}\dd t_j,
  \qquad
  \tau_j:=
  \frac{(t_j-t_i)(t_\ell-t_j)}
       {t_\ell-t_i}.
  \label{eq:critical-delete-path-factor}
\end{align}
Writing \(D:=t_\ell-t_i\) and \(s:=t_j-t_i\), we have
\[
  \frac{s}{2}\le\frac{s(D-s)}D\le s,
  \qquad 0<s\le\frac D2.
\]
Consequently,
\begin{align}
  \int_{t_i}^{t_\ell}
  \frac{\Lambda(\tau_j)}{1+\tau_j}\dd t_j
  &\le
  C\int_0^{D/2}\frac{\Lambda(s)}{1+s}\dd s
  \le C\Lambda(D)^2
  \le C\Lambda(2qu)^2.
  \label{eq:critical-insertion-time-cost}
\end{align}
The same calculation applies to an interior path factor in the second
chain. The selected anchor and, when present, the second-chain
left-endpoint factor are not deleted.

If a deleted factor is at the original right endpoint (so that it has
only one neighboring retained vertex), the one-sided estimate
\eqref{eq:critical-pQ} gives
\begin{align*}
  &\int_{\substack{t_i<t_j\\t_j-t_i\le qu}}
  \int_{\R^2}
  Q(x_j-B_{t_j})
  p_{t_j-t_i}(x_j-x_i)
  \dd x_j\dd t_j\\
  &\quad\le
  C\int_0^{qu}\frac{\Lambda(a)}{1+a}\dd a
  \le C\Lambda(2qu)^2.
\end{align*}
Consequently, after deleting
the path factors specified above, we have
\begin{equation}
  \mathcal D_{A,A'}(B)
  \le
  \bigl[C\Lambda(2qu)^2\bigr]^{2r-1}
  \mathscr R_{A,A'}^{\mathrm{cr}}(B).
  \label{eq:critical-path-reduction}
\end{equation}

To estimate the right-hand side of
\eqref{eq:critical-path-reduction}, we use the following lemma. Here
\(\mathscr R_{A,A'}^{\mathrm{cr}}(B)\) denotes the nonnegative residual
integral left after the stated path-factor deletions; its explicit
coordinate representation is given below.

\begin{lemma}
\label{lem:critical-anchor-chain}
Set \(m:=q-r\). Let \(0<s_0<\cdots<s_m<T\), and

\[
  d_\ell:=s_\ell-s_{\ell-1}\quad(1\le\ell\le m),
  \qquad s_m-s_0\le qu,
  \qquad
  \lambda(t):=\frac{\Lambda(t)}{1+t}.
\]
Thus \(d_\ell\) is the transition duration between the
\((\ell-1)\)-st and \(\ell\)-th Gaussian vertices in each chain. After
relabeling the paired Gaussian variables as \(x_\ell,x'_\ell\), the
Gaussian vertices are \((s_\ell,x_\ell)\) and \((s_\ell,x'_\ell)\),
arising from the pairing in
\eqref{eq:critical-retained-gaussian-pairing}. Recall the anchor
\(Q(z-B_{\widehat t})\) selected above. Its original index is \(j_*\),
so \(\widehat t=t_{j_*}\); the symbol \(z\) continues to denote its
original spatial coordinate and is distinct from the relabeled
Gaussian variables \(x_\ell,x'_\ell\). When \(j_*=0\), one has
\(z=\bar x_0\), the original left endpoint. Assume
\(0<\widehat t<T\) and

\[
  |\widehat t-s_0|\le qu,
\]
and write \(\rho:=|\widehat t-s_0|\). Since \(j_*\in A\), whereas the
first-chain Gaussian indices lie in \(A^c\), one has
\(\widehat t\ne s_\ell\) for every \(\ell\). Hence exactly one of the
following cases occurs (empty products below are one):

\[
  \begin{array}{lll}
    k=0: & \alpha=s_0-\widehat t>0,\\
    1\le k\le m: & 0<\alpha=\widehat t-s_{k-1}<d_k,\\
    k=m+1: & \alpha=\widehat t-s_m>0.
  \end{array}
\]

Accordingly, define
\begin{align}
  \Gamma^{(1)}_{0,\alpha}(z,\mathbf x)
  &:=\1_{\{z\in\mathcal C_T\}}\,p_\alpha(x_0-z)
    \prod_{\ell=1}^{m}p_{d_\ell}(x_\ell-x_{\ell-1}), \\
  \Gamma^{(1)}_{k,\alpha}(z,\mathbf x)
  &:=\1_{\{x_0\in\mathcal C_T\}}
    \prod_{\ell=1}^{k-1}p_{d_\ell}(x_\ell-x_{\ell-1})
    p_\alpha(z-x_{k-1}) \notag\\[-1mm]
  &\hspace{18mm}\times
    p_{d_k-\alpha}(x_k-z)
    \prod_{\ell=k+1}^{m}p_{d_\ell}(x_\ell-x_{\ell-1}),
    \qquad 1\le k\le m, \\
  \Gamma^{(1)}_{m+1,\alpha}(z,\mathbf x)
  &:=\1_{\{x_0\in\mathcal C_T\}}
    \prod_{\ell=1}^{m}p_{d_\ell}(x_\ell-x_{\ell-1})
    p_\alpha(z-x_m).
  \label{eq:critical-anchor-chain-kernels}
\end{align}

The second-chain kernels are

\begin{align}
  \Gamma^{(2)}_0(\mathbf x')
  &:=\1_{\{x'_0\in\mathcal C_T\}}
    \prod_{\ell=1}^{m}p_{d_\ell}(x'_\ell-x'_{\ell-1}), \\ 
  \Gamma^{(2)}_1(a',\bar x'_0,\mathbf x')
  &:=\1_{\{\bar x'_0\in\mathcal C_T\}}
    p_{a'}(x'_0-\bar x'_0)
    \prod_{\ell=1}^{m}p_{d_\ell}(x'_\ell-x'_{\ell-1}),
  \qquad 0<a'<qu.
  \label{eq:critical-anchor-second-kernels}
\end{align}

For \(\varepsilon'=0\), set

\[
\begin{aligned}
  \mathfrak a^{(0)}_{k,\alpha}(B)
  &:={}
  \int Q(z-B_{\widehat t})\Gamma^{(1)}_{k,\alpha}(z,\mathbf x)
  \Gamma^{(2)}_0(\mathbf x')
  \prod_{\ell=0}^{m}Q(x_\ell-x'_\ell)\,
  \dd z\,\dd\mathbf x\,\dd\mathbf x',
\end{aligned}
\]
and, for any interval \(I'\subset(0,qu)\) satisfying
\(0<s_0-a'<T/4\) for every \(a'\in I'\), set

\[
\begin{aligned}
  \mathfrak a^{(1)}_{k,\alpha}(B)
  &:={}
  \int_{I'}\int Q(z-B_{\widehat t})
  Q(\bar x'_0-B_{s_0-a'})
  \Gamma^{(1)}_{k,\alpha}(z,\mathbf x)
  \Gamma^{(2)}_1(a',\bar x'_0,\mathbf x')\\[-1mm]
  &\hspace{20mm}\times
  \prod_{\ell=0}^{m}Q(x_\ell-x'_\ell),
  \dd z\,\dd\mathbf x\,\dd\mathbf x'
  \,\dd\bar x'_0\,\dd a'.
\end{aligned}
\]

Then, uniformly in \(B\), \(k\), \(\alpha\), and \(I'\),

\begin{equation}
  \mathfrak a^{(\varepsilon')}_{k,\alpha}(B)
  \le
  C^{m+1}\Lambda(T)^2\Lambda(2qu)^4
  \prod_{\ell=1}^{m}\lambda(d_\ell),
  \qquad \varepsilon'\in\{0,1\}.
  \label{eq:critical-anchor-chain}
\end{equation}
\end{lemma}

\begin{proof}
All other path factors have already been integrated out. The cases \(k=0\),
\(1\le k\le m\), and \(k=m+1\) mean respectively that this retained
factor is at the original left endpoint, strictly between two Gaussian
vertices, or after the last Gaussian vertex.

For \(D>0\),
\[
  \int_0^D\lambda(s)\,\dd s
  \le C\Lambda(D)^2.
\]
For \(\ell=m,\ldots,1\), first integrate \(x'_\ell\) in \(\mathfrak a^{(\varepsilon')}_{k,\alpha}(B)\). For
\(\ell\ge1\), the factor being integrated is the link
\(Q(x_\ell-x'_\ell)\) between the two Gaussian vertices at time
\(s_\ell\); we call it a relative link. The link at \(\ell=0\),
\(Q(x_0-x'_0)\), is the left-endpoint link and is not integrated at
this step. 

For \(1\le \ell\le m\), define
\[
  J_\ell(x_\ell,x'_{\ell-1})
  :=
  \int_{\R^2}
  p_{d_\ell}(x'_\ell-x'_{\ell-1})
  Q(x_\ell-x'_\ell)\,\dd x'_\ell .
\]
By the convolution estimate \eqref{eq:critical-pQ},
\begin{equation}
  J_\ell(x_\ell,x'_{\ell-1})
  \le
  C\frac{\Lambda(d_\ell)}{1+d_\ell}
  =
  C\lambda(d_\ell),
  \qquad
  1\le\ell\le m,
  \label{eq:critical-one-relative-link}
\end{equation}
uniformly in \(x_\ell,x'_{\ell-1}\). Thus the integration in
\(x'_\ell\) produces the factor \(C\lambda(d_\ell)\), while the
variable \(x_\ell\) is still integrated in the first chain.

Assume first that \(1\le k\le m\), so that
\[
  s_{k-1}<\widehat t<s_k,
  \qquad
  \alpha=\widehat t-s_{k-1},
  \qquad
  d_k-\alpha=s_k-\widehat t.
\]
After the integrations with indices \(\ell=m,\ldots,k+1\), the
right-hand part of the first chain has been removed by
\(\int p_t=1\). The factors involving the variables
\(x_k,x'_k\) are then
\[
  p_\alpha(z-x_{k-1})
  p_{d_k-\alpha}(x_k-z)
  p_{d_k}(x'_k-x'_{k-1})
  Q(x_k-x'_k).
\]
Consequently,
\begin{align*}
&\int_{\R^2}\!\dd x_k
  \int_{\R^2}\!\dd x'_k\,
  p_\alpha(z-x_{k-1})
  p_{d_k-\alpha}(x_k-z)
  p_{d_k}(x'_k-x'_{\ell-1})
  Q(x_k-x'_k)
\\
&\qquad
=
p_\alpha(z-x_{k-1})
\int_{\R^2}
p_{d_k-\alpha}(x_k-z)
J_k(x_k,x'_{k-1})\,\dd x_k
\\
&\qquad
\le
C\lambda(d_k)\,
p_\alpha(z-x_{k-1})
\int_{\R^2}
p_{d_k-\alpha}(x_k-z)\,\dd x_k
\\
&\qquad
=
C\lambda(d_k)\,p_\alpha(z-x_{k-1}).
\end{align*}
Thus the factor on the left of the anchor remains, whereas the
factor on the right of the anchor is integrated.

For \(1\le \ell\le k\), define the deterministic time length
\[
  \eta_\ell
  :=\widehat t-s_{\ell-1}
  =\alpha+\sum_{j=\ell}^{k-1}d_j,
\]
where the sum is empty when \(\ell=k\). Thus
\[
  \eta_k=\alpha,
  \qquad
  \eta_{\ell+1}=\widehat t-s_\ell,
  \qquad
  d_\ell+\eta_{\ell+1}=\eta_\ell.
\]
After the integrations with indices
\(k,k-1,\ldots,\ell+1\), the remaining first-chain factor
involving \(x_\ell\) is
\[
  p_{\eta_{\ell+1}}(z-x_\ell)
  =
  p_{\widehat t-s_\ell}(z-x_\ell).
\]
Therefore,
\begin{align*}
&\int_{\R^2}\!\dd x_\ell
  \int_{\R^2}\!\dd x'_\ell\,
  p_{d_\ell}(x_\ell-x_{\ell-1})
  p_{\eta_{\ell+1}}(z-x_\ell)
  p_{d_\ell}(x'_\ell-x'_{\ell-1})
  Q(x_\ell-x'_\ell)
\\
&\qquad\le
C\lambda(d_\ell)
\int_{\R^2}
p_{d_\ell}(x_\ell-x_{\ell-1})
p_{\eta_{\ell+1}}(z-x_\ell)\,\dd x_\ell
\\
&\qquad=
C\lambda(d_\ell)
p_{d_\ell+\eta_{\ell+1}}(z-x_{\ell-1})
\\
&\qquad=
C\lambda(d_\ell)
p_{\eta_\ell}(z-x_{\ell-1}).
\end{align*}
Finally, the remaining first-chain factor is \(p_{\hat t-s_0}(z-x_0)\).

For the case \(k=0\), the anchor lies before the first Gaussian
vertex. There is no split transition among the indices
\(\ell=1,\ldots,m\). Hence, for every \(\ell=m,\ldots,1\),
\begin{align*}
&\int_{\R^2}\!\dd x_\ell
  \int_{\R^2}\!\dd x'_\ell\,
  p_{d_\ell}(x_\ell-x_{\ell-1})
  p_{d_\ell}(x'_\ell-x'_{\ell-1})
  Q(x_\ell-x'_\ell)
\\
&\qquad
\le
C\lambda(d_\ell)
\int_{\R^2}
p_{d_\ell}(x_\ell-x_{\ell-1})\,\dd x_\ell
=
C\lambda(d_\ell).
\end{align*}
After all these integrations, the remaining first-chain factor is \(p_{s_0-\hat t}(x_0-z)\).

For the case \(k=m+1\), the anchor lies after the last Gaussian vertex.
Set
\[
  \eta_{m+1}:=\alpha.
\]
At the first step, \(\ell=m\), we obtain
\begin{align*}
&\int_{\R^2}\!\dd x_m
  \int_{\R^2}\!\dd x'_m\,
  p_{d_m}(x_m-x_{m-1})
  p_{\eta_{m+1}}(z-x_m)
  p_{d_m}(x'_m-x'_{m-1})
  Q(x_m-x'_m)
\\
&\qquad
\le
C\lambda(d_m)
\int_{\R^2}
p_{d_m}(x_m-x_{m-1})
p_{\eta_{m+1}}(z-x_m)\,\dd x_m
\\
&\qquad
=
C\lambda(d_m)
p_{d_m+\eta_{m+1}}(z-x_{m-1}).
\end{align*}
We then repeat the same argument for
\(\ell=m-1,\ldots,1\), defining
\[
  \eta_\ell:=d_\ell+\eta_{\ell+1}.
\]
At the end,
\[
  \eta_1
  =
  \alpha+d_1+\cdots+d_m
  =
  \widehat t-s_0,
\]
and the remaining first-chain factor is
\[
  p_{\eta_1}(z-x_0)=p_{\hat t-s_0}(z-x_0).
\]

Thus, in all three positions of the anchor, every index
\(\ell=1,\ldots,m\) contributes one factor
\(C\lambda(d_\ell)\). After the variables
\(x_\ell,x'_\ell\), \(1\le\ell\le m\), have been integrated, we give
\begin{equation}
  \mathfrak a^{(\varepsilon')}_{k,\alpha}(B)
  \le
  C^m\prod_{\ell=1}^{m}\lambda(d_\ell)\,
  \mathcal E^{(\varepsilon')}_{k,\alpha}(B).
  \label{eq:critical-anchor-elimination}
\end{equation}
where \(\mathcal E^{(\varepsilon')}_{k,\alpha}\) is the remaining
\emph{left-endpoint block}. Its first-chain kernel is
\(p_\rho(z-x_0)=p_{|\widehat t-s_0|}(z-x_0)\), obtained by the above computation.
The term ``left endpoint'' refers to the index-\(0\) variables and the
two original corridor constraints. Put \(\mathcal C:=\mathcal C_T\),
\(b:=B_{\widehat t}\), and \(b_{a'}:=B_{s_0-a'}\). The block is given
explicitly below.

For \(\varepsilon'=0\),
\begin{align*}
  \mathcal E^{(0)}_{0,\alpha}(B)
  &=\int_{\mathcal C}\!\dd z\int_{\mathcal C}\!\dd x'_0
    \int_{\R^2}\!\dd x_0\,
    Q(z-b)p_\rho(x_0-z)Q(x_0-x'_0),\\
  \mathcal E^{(0)}_{k,\alpha}(B)
  &=\int_{\mathcal C}\!\dd x_0\int_{\mathcal C}\!\dd x'_0
    \int_{\R^2}\!\dd z\,
    Q(z-b)p_\rho(z-x_0)Q(x_0-x'_0),
    \quad 1\le k\le m+1.
\end{align*}
For \(\varepsilon'=1\),
\begin{align*}
  \mathcal E^{(1)}_{0,\alpha}(B)
  &=\int_{I'}\!\dd a'\int_{\mathcal C}\!\dd z
    \int_{\mathcal C}\!\dd\bar x'_0
    \int_{(\R^2)^2}\!\dd x_0\dd x'_0\,
    Q(z-b)Q(\bar x'_0-b_{a'})\\[-1mm]
  &\hspace{28mm}\times p_\rho(x_0-z)
    p_{a'}(x'_0-\bar x'_0)Q(x_0-x'_0),\\
  \mathcal E^{(1)}_{k,\alpha}(B)
  &=\int_{I'}\!\dd a'\int_{\mathcal C}\!\dd x_0
    \int_{\mathcal C}\!\dd\bar x'_0
    \int_{(\R^2)^2}\!\dd z\dd x'_0\,
    Q(z-b)Q(\bar x'_0-b_{a'})\\[-1mm]
  &\hspace{28mm}\times p_\rho(z-x_0)
    p_{a'}(x'_0-\bar x'_0)Q(x_0-x'_0),
    \quad 1\le k\le m+1.
\end{align*}

The corridor estimate used below is

\begin{align}
  \sup_{v\in\R^2}\int_{v+\mathcal C}Q(y)\,\dd y
  &\le C\Lambda(T),
  \label{eq:critical-corridor-mass}\\
  \sup_{c\in\R^2}\int_{\mathcal C}(p_a*Q)(x-c)\,\dd x
  &\le C\Lambda(T),
  \qquad a\ge0.
  \label{eq:critical-endpoint-estimates}
\end{align}
The second estimate follows from the first by Gaussian translation and
Fubini; the case \(a=0\) is understood by direct substitution. Applying
these estimates to the four displayed formulas gives, when
\(\varepsilon'=0\),

\[
  \mathcal E^{(0)}_{k,\alpha}(B)\le C\Lambda(T)^2.
\]

If \(\varepsilon'=1\), the inner integrations in the two cases are
\begin{align*}
  &\int_{(\R^2)^2}p_\rho(x-z)Q(x-x')p_{a'}(x'-\bar x')
    \,\dd x\dd x'=(p_{\rho+a'}*Q)(z-\bar x'),\\
  &\int_{\R^2}p_{a'}(x'-\bar x')Q(x-x')\,\dd x'
    =(p_{a'}*Q)(x-\bar x').
\end{align*}
The first identity applies to \(k=0\), and the second to
\(1\le k\le m+1\). Therefore \eqref{eq:critical-pQ} and
\eqref{eq:critical-corridor-mass} give

\begin{align}
  \mathcal E^{(1)}_{k,\alpha}(B)
  &\le C\Lambda(T)^2
  \int_{I'}\bigl[\lambda(a')+\lambda(\rho+a')\bigr]\,\dd a' \le C\Lambda(T)^2\Lambda(2qu)^2.
  \label{eq:critical-optional-endpoint}
\end{align}

The formulas above cover all three positions
\(k=0,1,\ldots,m+1\). When \(m=0\), the product in
\eqref{eq:critical-anchor-elimination} is empty and the same displayed
left-endpoint formulas apply directly.
Since \(\Lambda\ge1\), \eqref{eq:critical-anchor-elimination} and
\eqref{eq:critical-optional-endpoint} imply
\eqref{eq:critical-anchor-chain}.
\end{proof}

Now, we return to \eqref{eq:critical-path-reduction} to apply
Lemma~\ref{lem:critical-anchor-chain} to its right-hand side. We first
give the expression of \(\mathscr R_{A,A'}^{\mathrm{cr}}(B)\).

Recall \(m=q-r\), and relabel the paired Gaussian vertices as
\((s_\ell,x_\ell)\) and \((s_\ell,x'_\ell)\), \(0\le\ell\le m\), as in
Lemma~\ref{lem:critical-anchor-chain}. The position of the anchor is
fixed by its original index \(j_*\):
\[
  k:=
  \begin{cases}
    0, & j_*<a_0,\\
    \ell, & a_{\ell-1}<j_*<a_\ell\quad(1\le\ell\le m),\\
    m+1, & j_*>a_m.
  \end{cases}
\]
There is no equality because \(j_*\in A\) and \(a_\ell\in A^c\). We define
\[
  \mathcal S_{A,A'}:=
  \left\{\mathbf s:0<s_0<\cdots<s_m<T,\ 
  \mathcal T_A(\mathbf s)\ne\varnothing,\ 
  \mathcal T_{A'}(\mathbf s)\ne\varnothing\right\}.
\]
For \(\mathbf s\in\mathcal S_{A,A'}\), see the definition of \(\mathcal T_A(\mathbf s)\) and \(\mathcal T_{A'}(\mathbf s)\) in \eqref{time fiber}, and put
\[
  H_{A,j_*}(\mathbf s)
  :=
  \{t_{j_*}:(t_j)_{j\in A}\in\mathcal T_A(\mathbf s)\}.
\]
Thus \(H_{A,j_*}(\mathbf s)\) is the remaining anchor-time range. If
\(\varepsilon'=1\), set
\[
  I'_{A'}(\mathbf s)
  :=
  \{s_0-t'_0:(t'_j)_{j\in A'}\in\mathcal T_{A'}(\mathbf s)\}.
\]
Moreover,
\[
  \begin{aligned}
  &H_{A,j_*}(\mathbf s)\subset
  \begin{cases}
    (s_0-qu,s_0), & k=0,\\
    (s_{k-1},s_k), & 1\le k\le m,\\
    (s_m,s_m+qu), & k=m+1,
  \end{cases}
  \end{aligned}
\]
When \(\varepsilon'=1\), the same constraints give
\[
  I'_{A'}(\mathbf s)\subset(0,qu),\qquad
  0<s_0-a'<T/4\quad(a'\in I'_{A'}(\mathbf s)).
\]
For \(\widehat t\in H_{A,j_*}(\mathbf s)\), define
\[
  \alpha(\mathbf s,\widehat t):=
  \begin{cases}
    s_0-\widehat t, & k=0,\\
    \widehat t-s_{k-1}, & 1\le k\le m,\\
    \widehat t-s_m, & k=m+1.
  \end{cases}
\]
In the formula below, \(\mathfrak a^{(\varepsilon')}_{k,\alpha}\)
means \(\mathfrak a^{(0)}_{k,\alpha}\) when \(\varepsilon'=0\), and
\(\mathfrak a^{(1)}_{k,\alpha}\) with
\(I'=I'_{A'}(\mathbf s)\) when \(\varepsilon'=1\). If \(j_*=0\),
\(z=\bar x_0\) and the \(z\)-integration in the lemma is the original
left-endpoint integration; if \(j_*>0\), then \(\bar x_0=x_0\).
Likewise, \(\varepsilon'=0\) gives \(t'_0=s_0\) and
\(\bar x'_0=x'_0\), while \(\varepsilon'=1\) changes variables from
\(t'_0\) to \(a'=s_0-t'_0\). Hence, by Tonelli's theorem and the definitions of \(\Gamma^{(1)}_{k,\alpha}\) and
\(\Gamma^{(2)}_{\varepsilon'}\), the residual integral is
\[
  \mathscr R_{A,A'}^{\mathrm{cr}}(B)
  =
  \int_{\mathcal S_{A,A'}}
  \int_{H_{A,j_*}(\mathbf s)}
  \mathfrak a^{(\varepsilon')}_{k,\alpha(\mathbf s,\widehat t)}(B)\,
  \dd\widehat t\,\dd\mathbf s .
\]
Thus the inner integral in \(\mathscr R_{A,A'}^{\mathrm{cr}}(B)\) is
the left-hand side of Lemma~\ref{lem:critical-anchor-chain}.

Put
\[
  d_\ell:=s_\ell-s_{\ell-1},\qquad 1\le\ell\le m.
\]
The residual time constraints imply
\[
  s_m-s_0\le qu,\qquad d_\ell\in(0,qu),\qquad
  |H_{A,j_*}(\mathbf s)|\le qu.
\]
Applying Lemma~\ref{lem:critical-anchor-chain} gives
\begin{align}
  \mathscr R_{A,A'}^{\mathrm{cr}}(B)
  &\le
  C^{m+1}\Lambda(T)^2\Lambda(2qu)^4
  \int_{\mathcal S_{A,A'}}
  |H_{A,j_*}(\mathbf s)|
  \prod_{\ell=1}^{m}\lambda(d_\ell)\,
  \dd\mathbf s .
  \label{eq:critical-residual-fiber-bound}
\end{align}
Writing \(\mathbf d=(d_1,\ldots,d_m)\) and
\[
  \mathbf s(s_0,\mathbf d)
  :=
  (s_0,s_0+d_1,\ldots,s_0+d_1+\cdots+d_m),
\]
and enlarging the \(d_\ell\)-domain, Tonelli's theorem gives
\begin{align}
  \mathscr R_{A,A'}^{\mathrm{cr}}(B)
  &\le
  C^{q+1}\Lambda(T)^2\Lambda(2qu)^4
  \int_0^T\dd s_0
  \int_{(0,qu)^m}
  |H_{A,j_*}(\mathbf s(s_0,\mathbf d))|
  \prod_{\ell=1}^{m}\lambda(d_\ell)\,\dd\mathbf d
  \notag\\
  &\le
  C^{q+1}T(qu)\Lambda(T)^2
  \Lambda(2qu)^{2m+4}
  =
  C^{q+1}T(qu)\Lambda(T)^2
  \Lambda(2qu)^{2(q-r)+4}.
  \label{eq:critical-remaining-integral}
\end{align}
When \(m=0\), \(k\in\{0,1\}\), while the \(d\)-integral and the product
are interpreted as \(1\).

Combining \eqref{eq:critical-path-reduction} and
\eqref{eq:critical-remaining-integral} gives
\begin{align}
  \mathcal D_{A,A'}(B)
  &\le
  C^{q+1}T(qu)\Lambda(T)^2
  \Lambda(2qu)^{2(q+r)+2}
  \notag\\
  &\le
  C^{q+1}T(qu)\Lambda(u)^{2(q+r)+4},
  \label{eq:critical-fixed-contraction-bound}
\end{align}
where we used
\[
  \Lambda(T)\asymp\Lambda(u),
  \qquad
  \Lambda(2qu)\asymp\Lambda(u).
\]
Dividing by the lower bound in
\eqref{eq:critical-sigma-bounds}, we obtain
\[
  \frac{\beta^{2r}\mathcal D_{A,A'}(B)}
       {(\sigma_{q,T}^{\mathrm{cr}})^2}
  \le
  C^{q+1}\frac{qu}{T}
  \beta^{2r}\Lambda(u)^{2r+3}.
\]
Summing over \(A,A'\) with \(|A|=|A'|=r\) and then over
\(r=1,\ldots,q\) proves
\eqref{eq:critical-shifted-variance}.

\subsection{The matching lower bound}
By \eqref{eq:replica-pinning-comparison}, it remains to estimate the pinning free energy associated with a critical potential.
\begin{lemma}
\label{lem:critical-pinning-upper}
Let \(V:\R^2\to[0,\infty)\) be bounded and satisfy
\[
  V(x)\le C_V(1+|x|)^{-2}.
\]
There are \(c,C,h_0>0\) such that
\begin{equation}
  F_V(h)
  \le C\exp\left\{-\frac{c}{\sqrt h}\right\},
  \qquad 0<h\le h_0.
  \label{eq:critical-pinning-upper}
\end{equation}
\end{lemma}

\begin{proof}
Set
\[
  \eta_V(T)
  :=
  \sup_{x\in\R^2}
  E_x\left[\int_0^T V(B_s)\dd s\right]
  =
  \sup_x\int_0^T(p_s*V)(x)\dd s.
\]
The \eqref{eq:critical-pQ} gives, for \(T\ge1\),
\[
  \eta_V(T)
  \le C_0+C_1\int_1^T\frac{\Lambda(s)}{1+s}\dd s
  \le C_0+C_1\Lambda(T)^2.
\]
Choose \(b>0\) so small that \(4C_1b^2<1\), and put
\[
  T_h:=\exp\{b/\sqrt h\}.
\]
For all sufficiently small \(h\),
\[
  h\eta_V(T_h)
  \le hC_0+C_1h\Lambda(T_h)^2
  \le\frac12.
\]
Lemma~\ref{lem:upper} gives
\[
  \sup_xE_x\exp\left\{
    h\int_0^{T_h}V(B_s)\dd s
  \right\}
  \le2.
\]
The Markov property gives submultiplicativity. Writing \(t=nT_h+r\), \(0\le r<T_h\) yields
\[
  \sup_xE_x\exp\left\{
    h\int_0^tV(B_s)\dd s
  \right\}
  \le2^{n+1}.
\]
Therefore
\[
  F_V(h)
  \le\frac{\log2}{T_h}
  =
  (\log2)\exp\left\{-\frac b{\sqrt h}\right\}.
\]
\end{proof}

Apply Lemma~\ref{lem:critical-pinning-upper} to
\(V(x)=Q(\sqrt2x)\). The replica comparison
\eqref{eq:replica-pinning-comparison} gives
\begin{equation}
  p(\beta)
  \ge
  -C F_{Q(\sqrt2\,\cdot)}(2\beta^2)
  \ge
  -C\exp\left\{-\frac c\beta\right\}.
  \label{eq:critical-free-energy-lower}
\end{equation}
Together with \eqref{eq:critical-free-energy-upper}, this proves
\eqref{eq:two-d-critical-free-energy} and completes the proof of
Theorem~\ref{thm:radial-regimes}(iii).

\appendix
\section{Some equalities}
\label{app:analytic}

\subsection{Convolution estimates}

\begin{lemma}
\label{lem:one-dimensional-convolution}
Let \(0<\alpha<1\) and let \(Z\sim N(0,1)\).  There is
\(C_\alpha<\infty\) such that, for every \(a\in\R\) and \(s\ge0\),
\begin{equation}
  \E\big[(1+|a+\sqrt sZ|)^{-\alpha}\big]
  \le C_\alpha(1+s)^{-\alpha/2}.
  \label{eq:one-dimensional-upper}
\end{equation}
For every \(L<\infty\), there is \(c_{\alpha,L}>0\) such that
\begin{equation}
  |a|\le L\sqrt s
  \quad\Longrightarrow\quad
  \E\big[(1+|a+\sqrt sZ|)^{-\alpha}\big]
  \ge c_{\alpha,L}(1+s)^{-\alpha/2}.
  \label{eq:one-dimensional-lower}
\end{equation}
\end{lemma}

\begin{proof}
The results \eqref{eq:one-dimensional-upper},\eqref{eq:one-dimensional-lower} are immediate for \(0\le s\le1\), after adjusting the
constants.  Let \(s\ge1\).  Both
\(
x\mapsto(1+|x|)^{-\alpha}
\)
and the centered Gaussian density are even and nonincreasing on
\([0,\infty)\). Their convolution is therefore maximized at zero.  Consequently,
\begin{align*}
  \sup_{a\in\R}
  \E(1+|a+\sqrt sZ|)^{-\alpha}
  &\le
  \E(1+\sqrt s|Z|)^{-\alpha} \\
  &\le
  \Pp(|Z|\le s^{-1/2})
  +s^{-\alpha/2}\E[|Z|^{-\alpha}\1_{\{|Z|>s^{-1/2}\}}] \\
  &\le Cs^{-1/2}+Cs^{-\alpha/2}
  \le Cs^{-\alpha/2}.
\end{align*}
Here \(\E|Z|^{-\alpha}<\infty\) because \(\alpha<1\).

For the lower bound, write \(b=a/\sqrt s\), so \(|b|\le L\).
On the event \(|Z|\le L+1\),
\[
  1+|a+\sqrt sZ|
  \le
  1+(2L+1)\sqrt s
  \le C_L\sqrt s.
\]
This event has a positive probability depending only on \(L\), which
proves \eqref{eq:one-dimensional-lower}.
\end{proof}

\begin{lemma}
\label{lem:product-gaussian-convolution}
Let \(G_{\bm\alpha}\) be defined by \eqref{eq:Galpha}, and let
\(Z\sim N(0,I_d)\).  There is \(C<\infty\) such that
\begin{equation}
  \sup_{a\in\R^d}
  \E\big[G_{\bm\alpha}(a+\sqrt sZ)\big]
  \le C(1+s)^{-\kappa/2}.
  \label{eq:product-gaussian-upper}
\end{equation}
For every \(L<\infty\), there is \(c_L>0\) such that
\begin{equation}
  |a|\le L\sqrt s
  \quad\Longrightarrow\quad
  \E\big[G_{\bm\alpha}(a+\sqrt sZ)\big]
  \ge c_L(1+s)^{-\kappa/2}.
  \label{eq:product-gaussian-lower}
\end{equation}
\end{lemma}

\begin{proof}
The coordinates of \(Z\) are independent, so the expectation factors
into \(d\) one-dimensional expectations.  Apply
Lemma~\ref{lem:one-dimensional-convolution} to every coordinate and use
\(\kappa=\sum_j\alpha_j\).
\end{proof}

\subsection{Pitt inequality} The one-dimensional Pitt inequality states that, for \(0<\alpha<1\), there exists a constant \(C_\alpha<\infty\) such that
\begin{equation}
  \int_{\R}|x|^{-\alpha}|g(x)|^2\dd x
  \le
  C_\alpha
  \int_{\R}|\xi|^\alpha|\widehat g(\xi)|^2\dd\xi
  \label{eq:one-dimensional-pitt}
\end{equation}
for all \(g \in \mathcal{S}(\mathbb{R})\) (the Schwartz space). See~\cite{Beckner1995}.

\begin{lemma}
\label{lem:product-pitt}
Let \(\alpha_j\in(0,1)\) and \(\kappa=\sum_j\alpha_j\le2\).
For every \(f\in H^1(\R^d)\),
\begin{align}
  \int_{\R^d}
  \prod_{j=1}^d|x_j|^{-\alpha_j}|f(x)|^2\dd x
  &\le
  C\int_{\R^d}
  \prod_{j=1}^d|\xi_j|^{\alpha_j}|\widehat f(\xi)|^2\dd\xi
  \label{eq:product-pitt} \\
  &\le
  C\norm{\nabla f}_2^\kappa
  \norm{f}_2^{2-\kappa}. \notag
\end{align}
Consequently,
\begin{equation}
  \int_{\R^d}G_{\bm\alpha}(x)|f(x)|^2\dd x
  \le
  C\norm{\nabla f}_2^\kappa
  \norm{f}_2^{2-\kappa}.
  \label{eq:Galpha-pitt}
\end{equation}
\end{lemma}

\begin{proof}
Apply the one-dimensional inequality \eqref{eq:one-dimensional-pitt} in
the variable \(x_1\), keeping the other variables fixed.  Apply it next
in \(x_2\) to the partial Fourier transform obtained at the first step,
and continue through all coordinates.  Tonelli's theorem gives
\eqref{eq:product-pitt}.  Since
\[
  \prod_{j=1}^d|\xi_j|^{\alpha_j}
  \le |\xi|^\kappa,
\]
H\"older's inequality and Plancherel's identity yield
\begin{align*}
  \int |\xi|^\kappa|\widehat f(\xi)|^2\dd\xi
  &\le
  \left(\int |\xi|^2|\widehat f(\xi)|^2\dd\xi\right)^{\kappa/2}
  \left(\int |\widehat f(\xi)|^2\dd\xi\right)^{1-\kappa/2}\\
  &=
  \norm{\nabla f}_2^\kappa\norm{f}_2^{2-\kappa}.
\end{align*}
Finally,
\(
(1+|x_j|)^{-\alpha_j}\le|x_j|^{-\alpha_j}
\)
away from the null set \(x_j=0\), proving \eqref{eq:Galpha-pitt}.
\end{proof}

\subsection{Hardy's inequality}
\begin{lemma}
\label{lem:hardy}
If \(d\ge3\), then every \(f\in H^1(\R^d)\) satisfies
\begin{equation}
  \int_{\R^d}\frac{|f(x)|^2}{|x|^2}\dd x
  \le
  \frac{4}{(d-2)^2}
  \int_{\R^d}|\nabla f(x)|^2\dd x.
  \label{eq:hardy}
\end{equation}
\end{lemma}

\begin{proof}
It is enough to consider
\(f\in C_c^\infty(\R^d\setminus\{0\})\), by density.  For any
\(a\in\R\),
\begin{align*}
  0
  &\le
  \int_{\R^d}
  \left|\nabla f(x)+a\frac{x}{|x|^2}f(x)\right|^2\dd x\\
  &=
  \int|\nabla f|^2
  +a^2\int\frac{|f|^2}{|x|^2}
  +a\int\frac{x}{|x|^2}\cdot\nabla(|f|^2).
\end{align*}
Since
\(\operatorname{div}(x/|x|^2)=(d-2)/|x|^2\), integration by parts gives
\[
  0\le
  \int|\nabla f|^2+
  \big(a^2-a(d-2)\big)\int\frac{|f|^2}{|x|^2}.
\]
Taking \(a=(d-2)/2\) proves \eqref{eq:hardy}.
\end{proof}

\subsection{Upper inequality }

\begin{lemma}
\label{lem:upper}
Let \(X=(X_t)_{t\ge0}\) be a time-homogeneous Markov process
on a measurable state space \(E\), and let
\(V:E\to[0,\infty]\) be measurable.  We write \(E_x\) for
expectation under the law of \(X\) started from \(x\in E\).
For \(T>0\), define
\[
  \eta_T
  :=
  \sup_{x\in E}
  E_x\left[\int_0^T V(X_s)\,ds\right].
\]
If \(\eta_T<\infty\), then, for every \(\lambda\ge0\) satisfying
\(\lambda\eta_T<1\), 
\begin{equation}
  \sup_xE_x\left[
    \exp\left\{
      \lambda\int_0^TV(X_s)\dd s
    \right\}
  \right]
  \le\frac1{1-\lambda\eta_T}.
  \label{eq:khasminskii}
\end{equation}
\end{lemma}

\begin{proof}
Put \(A_T=\int_0^TV(X_s)\dd s\).  Expanding the \(n\)-th power and
ordering the time variables,
\begin{equation}
  E_x[A_T^n]
  =
  n!E_x\left[
  \int_{0<s_1<\cdots<s_n<T}
  \prod_{j=1}^nV(X_{s_j})
  \dd s_1\cdots\dd s_n
  \right].
  \label{eq:ordered-moments}
\end{equation}
Starting with the last time variable and applying the Markov property
recursively gives
\begin{equation}
  \sup_xE_x[A_T^n]\le n!\eta_T^n.
  \label{eq:khasminskii-moments}
\end{equation}
Therefore
\[
  \sup_xE_x[\e^{\lambda A_T}]
  \le
  \sum_{n=0}^\infty(\lambda\eta_T)^n
  =
  \frac1{1-\lambda\eta_T}.
\]
\end{proof}

\bibliographystyle{unsrtnat}
\bibliography{references}

@article{Beckner1995,
  author  = {Beckner, William},
  title   = {Pitt's inequality and the uncertainty principle},
  journal = {Proceedings of the American Mathematical Society},
  volume  = {123},
  number  = {6},
  pages   = {1897--1905},
  year    = {1995},
  doi     = {10.1090/S0002-9939-1995-1254832-9}
}

@article{BergerLacoin2017,
  author  = {Berger, Quentin and Lacoin, Hubert},
  title   = {The high-temperature behavior for the directed polymer in dimension {$1+2$}},
  journal = {Annales de l'Institut Henri Poincar\'e, Probabilit\'es et Statistiques},
  volume  = {53},
  number  = {1},
  pages   = {430--450},
  year    = {2017},
  doi     = {10.1214/15-AIHP721}
}

@article{CarmonaHu2002,
  author  = {Carmona, Philippe and Hu, Yueyun},
  title   = {On the partition function of a directed polymer in a Gaussian random environment},
  journal = {Probability Theory and Related Fields},
  volume  = {124},
  number  = {3},
  pages   = {431--457},
  year    = {2002},
  doi     = {10.1007/s004400200213}
}

@article{CarmonaHu2006,
  author  = {Carmona, Philippe and Hu, Yueyun},
  title   = {Strong disorder implies strong localization for directed polymers in a random environment},
  journal = {ALEA, Latin American Journal of Probability and Mathematical Statistics},
  volume  = {2},
  pages   = {217--229},
  year    = {2006}
}

@book{Comets2017,
  author    = {Comets, Francis},
  title     = {Directed Polymers in Random Environments},
  series    = {Lecture Notes in Mathematics},
  volume    = {2175},
  publisher = {Springer},
  address   = {Cham},
  year      = {2017},
  doi       = {10.1007/978-3-319-50487-2}
}

@article{CometsShigaYoshida2003,
  author  = {Comets, Francis and Shiga, Tokuzo and Yoshida, Nobuo},
  title   = {Directed polymers in a random environment: Path localization and strong disorder},
  journal = {Bernoulli},
  volume  = {9},
  number  = {4},
  pages   = {705--723},
  year    = {2003},
  doi     = {10.3150/bj/1066223275}
}

@article{CometsVargas2006,
  author  = {Comets, Francis and Vargas, Vincent},
  title   = {Majorizing multiplicative cascades for directed polymers in random media},
  journal = {ALEA, Latin American Journal of Probability and Mathematical Statistics},
  volume  = {2},
  pages   = {267--277},
  year    = {2006}
}

@article{CometsYoshida2006,
  author  = {Comets, Francis and Yoshida, Nobuo},
  title   = {Directed polymers in random environment are diffusive at weak disorder},
  journal = {The Annals of Probability},
  volume  = {34},
  number  = {5},
  pages   = {1746--1770},
  year    = {2006},
  doi     = {10.1214/009117905000000828}
}

@book{denHollander2009,
  author    = {den Hollander, Frank},
  title     = {Random Polymers},
  series    = {Lecture Notes in Mathematics},
  volume    = {1974},
  publisher = {Springer},
  address   = {Berlin},
  year      = {2009},
  doi       = {10.1007/978-3-642-00333-2}
}

@article{HuseHenley1985,
  author  = {Huse, David A. and Henley, Christopher L.},
  title   = {Pinning and roughening of domain walls in {Ising} systems due to random impurities},
  journal = {Physical Review Letters},
  volume  = {54},
  number  = {25},
  pages   = {2708--2711},
  year    = {1985},
  doi     = {10.1103/PhysRevLett.54.2708}
}

@article{JunkLacoin2024,
  author  = {Junk, Stefan and Lacoin, Hubert},
  title   = {Strong disorder and very strong disorder are equivalent for directed polymers},
  journal = {Annales Scientifiques de l'\'Ecole Normale Sup\'erieure},
  volume  = {59},
  number  = {4},
  pages   = {827--865},
  year    = {2026},
  doi     = {10.24033/asens.1849},
  eprint  = {2402.02562},
  archivePrefix = {arXiv},
  primaryClass  = {math.PR}
}

@article{JunkLacoin2026,
  author  = {Junk, Stefan and Lacoin, Hubert},
  title   = {Coincidence of critical points for directed polymers for general environments and random walks},
  journal = {Orbita Mathematicae},
  volume  = {3},
  number  = {1},
  pages   = {89--125},
  year    = {2026},
  doi     = {10.2140/om.2026.3.89},
  eprint  = {2502.04113},
  archivePrefix = {arXiv},
  primaryClass  = {math.PR}
}

@article{Lacoin2010,
  author  = {Lacoin, Hubert},
  title   = {New bounds for the free energy of directed polymers in dimension {$1+1$} and {$1+2$}},
  journal = {Communications in Mathematical Physics},
  volume  = {294},
  number  = {2},
  pages   = {471--503},
  year    = {2010},
  doi     = {10.1007/s00220-009-0957-3}
}

@article{Lacoin2011,
  author  = {Lacoin, Hubert},
  title   = {Influence of spatial correlation for directed polymers},
  journal = {The Annals of Probability},
  volume  = {39},
  number  = {1},
  pages   = {139--175},
  year    = {2011},
  doi     = {10.1214/10-AOP553}
}

@article{RoviraTindel2005,
  author  = {Rovira, Carles and Tindel, Samy},
  title   = {On the Brownian-directed polymer in a Gaussian random environment},
  journal = {Journal of Functional Analysis},
  volume  = {222},
  number  = {1},
  pages   = {178--201},
  year    = {2005},
  doi     = {10.1016/j.jfa.2004.07.017}
}

@article{Zygouras2024,
  author  = {Zygouras, Nikos},
  title   = {Directed polymers in a random environment: A review of the phase transitions},
  journal = {Stochastic Processes and their Applications},
  volume  = {177},
  pages   = {104431},
  year    = {2024},
  doi     = {10.1016/j.spa.2024.104431}
}

@article{Viveros2023,
  author  = {Viveros, Roberto},
  title   = {Directed polymer for very heavy tailed random walks},
  journal = {The Annals of Applied Probability},
  volume  = {33},
  number  = {1},
  pages   = {490--506},
  year    = {2023},
  doi     = {10.1214/22-AAP1821}
}

@article{Nitzschner2025,
  author  = {Nitzschner, Maximilian},
  title   = {Absence of weak disorder for directed polymers on supercritical percolation clusters},
  journal = {Annales de l'Institut Henri Poincar\'e, Probabilit\'es et Statistiques},
  volume  = {61},
  number  = {2},
  pages   = {1319--1333},
  year    = {2025},
  doi     = {10.1214/24-AIHP1463}
}

@misc{CottiniNitzschner2025,
  author        = {Cottini, Francesca and Nitzschner, Maximilian},
  title         = {Non-coincidence of critical points for directed polymers on supercritical percolation clusters},
  year          = {2025},
  month         = oct,
  howpublished  = {arXiv preprint arXiv:2510.24532},
  eprint        = {2510.24532},
  archivePrefix = {arXiv},
  primaryClass  = {math.PR},
  doi           = {10.48550/arXiv.2510.24532}
}

@article{mukherjee2016weak,
  author  = {Mukherjee, Chiranjib and Shamov, Alexander and Zeitouni, Ofer},
  title   = {Weak and strong disorder for the stochastic heat equation and continuous directed polymers in {$d\geq 3$}},
  journal = {Electronic Communications in Probability},
  volume  = {21},
  number  = {61},
  pages   = {1--12},
  year    = {2016},
  doi     = {10.1214/16-ECP18}
}

@article{Nakashima2019,
  author       = {Nakashima, Makoto},
  title        = {Free energy of directed polymers in random environment
                  in {$1+1$}-dimension at high temperature},
  journal      = {Electronic Journal of Probability},
  volume       = {24},
  number       = {50},
  pages        = {1--43},
  year         = {2019},
  doi          = {10.1214/19-EJP292}
}

@misc{Nakashima2025,
  author       = {Nakashima, Makoto},
  title        = {A note on the asymptotics of the free energy of {$1+1$}
                  dimensional directed polymers in random environment at
                  high temperature},
  year         = {2025},
  month        = mar,
  howpublished = {arXiv preprint arXiv:2503.20192},
  eprint       = {2503.20192},
  archivePrefix = {arXiv},
  primaryClass = {math.PR},
  doi          = {10.48550/arXiv.2503.20192}
}

@misc{BergerNakajima2026,
  author       = {Berger, Quentin and Nakajima, Shuta},
  title        = {Sharp behavior of the free energy for the two-dimensional
                  directed polymer model},
  year         = {2026},
  month        = may,
  howpublished = {arXiv preprint arXiv:2605.30707},
  eprint       = {2605.30707},
  archivePrefix = {arXiv},
  primaryClass = {math.PR},
  doi          = {10.48550/arXiv.2605.30707}
}

@misc{LacoinSmooth2025,
  author       = {Lacoin, Hubert},
  title        = {The localization transition for the directed polymer in a
                  random environment is smooth},
  year         = {2025},
  month        = may,
  howpublished = {arXiv preprint arXiv:2505.13382},
  eprint       = {2505.13382},
  archivePrefix = {arXiv},
  primaryClass = {math.PR},
  doi          = {10.48550/arXiv.2505.13382}
}

@article{CoscoSeroussiZeitouni2021,
  author  = {Cosco, Cl{\'e}ment and Seroussi, Inbar and Zeitouni, Ofer},
  title   = {Directed polymers on infinite graphs},
  journal = {Communications in Mathematical Physics},
  volume  = {386},
  number  = {1},
  pages   = {395--432},
  year    = {2021},
  doi     = {10.1007/s00220-021-04034-w}
}

@article{ArendtBatty1996,
  author  = {Arendt, Wolfgang and Batty, Charles J. K.},
  title   = {The Spectral Bound of Schr{\"o}dinger Operators},
  journal = {Potential Analysis},
  volume  = {5},
  number  = {3},
  pages   = {207--230},
  year    = {1996},
  doi     = {10.1007/BF00282361}
}

\end{document}